\documentclass[a4paper]{amsart}
\usepackage{amsfonts,amssymb,verbatim,amsmath,amsthm,latexsym,textcomp,amscd,epsfig,etoolbox}
\usepackage[hidelinks]{hyperref}
\usepackage{tikz-cd}
\usepackage{graphics,textcomp}
\usepackage{graphicx}
\usepackage{color}
\usepackage{url}
\usepackage{enumitem}
\usepackage{comment}
\usepackage{placeins}

	\newtheorem{theorem}{Theorem}[section]
	\newtheorem{lemma}[theorem]{Lemma}
	
	\newtheorem{corollary}[theorem]{Corollary}
	\theoremstyle{definition}
	\newtheorem{definition}[theorem]{Definition}

	\newtheorem{remark}[theorem]{Remark}
	
	\newtheorem*{ack}{Acknowledgement}

\newcommand{\onlypaper}[1]{#1}
\newcommand{\onlythesis}[1]{}
\newcommand{\mythesis}{paper}
\newcommand{\myref}{section}
\newcommand{\mysubref}{subsection}
\excludecomment{thesisblock}       
        
\let\chapter\section
\let\section\subsection
\let\subsection\subsubsection
\let\subsubsection\paragraph

\title{Real rational knots in the quadric of signature~$(3,2)$}

\author{Shane D'Mello}
\address{Department of Mathematics, Indian Institute of Science Education and Research, Mohali, India}
\email{shane@iisermohali.ac.in}

\author{Priya Rani}
\address{Department of Mathematics, Indian Institute of Science Education and Research, Mohali, India}
\email{ph18014@iisermohali.ac.in}

\begin{document}

\maketitle
\begin{abstract}
	 \onlythesis{The classification of knots and links in real projective spaces is inspired by a question
	originating from Hilbert’s list of problems presented in 1900. One of the generalisations of this problem is the classification of real rational knots in different spaces. The study of real rational knots lies in the intersection of knot theory and real algebraic varieties. In $\mathbb{RP}^{4}$, there are two non-singular non-empty quadrics up to projective transformations, one with signature $(4,1)$ which is a $3$-sphere and  another one, a $3$-hyperboloid with signature $(3,2)$. Real rational knots of low degrees in the $3$-sphere have already been classified.} The \mythesis{}  primarily studies the space of real rational curves of low degree in the quadric of signature $(3,2)$ and provides a classificaton of real rational knots and nodal curves. Apart from the classification, we also study the relationship between the real rational knots in the quadric and the real rational knots in $\mathbb{RP}^{3}$. Furthermore, a construction for the representatives of all the real rational knots of degree $\leq 5$ in the quadric is presented. \onlythesis{The last part of the thesis focuses on the classification of the real rational knots in the space $\mathbb{RP}^{1} \times \mathbb{RP}^{2}$.}
		\end{abstract}

\chapter{Introduction}\label{chap-0}
\onlythesis{\section{Background and Motivation}
 The study of real algebraic knots focuses on the restrictions that the algebraic structure imposes on a real algebraic knot. This field is a generalization of the work of Hilbert.} In~1900, Hilbert introduced his famous list of problems, in which the first part of the 16th problem focuses on the classification of real algebraic curves in the projective plane. Since the non-singular real algebraic curves in $\mathbb{RP}^{2}$ are topologically just a union of circles, in this setting, the isotopy classification reduces to the relative arrangement of those circles. Yet, even in this simple setting, a complete classification has only succeeded up to degree 7.
 
 There are several ways to generalize this problem. One \onlythesis{is to drop the condition of non-singularity and allow even singular plane curves. Another way is to keep non-singularity, but increase the dimension and study real algebraic hypersurfaces. Yet another }generalization is to keep both non-singularity and the dimension of the object,  but increase the dimension of the ambient space. This increase in co-dimension makes the problem a lot more intricate because now the curves may be knotted, while the counterpart
 for plane curves were topologically simple and just a union of circles. Indeed, this study of real rational curves in $\mathbb{RP}^{3}$ was first studied by Bj\"{o}rklund \cite{MR2844809} who classified real rational curves up to degree~5. The concept of an isotopy between real algebraic curves may be strengthened to one of rigid isotopy, where one demands that the curves remain algebraic and non-singular throughout the isotopy. Rigid isotopy was introduced by Rohklin in \cite{MR511882} in the context of plane non-singular real algebraic curves. Rigid isotopy is a stronger notion than smooth isotopy. \onlythesis{Smooth isotopy allows any smooth deformations while rigid isotopy allows deformations in the space of non-singular real rational parametrisations of same degree. }In~\cite{MR2844809}, Bj\"{o}rklund classified these knots up to rigid isotopy. This was extended to degree 6 by Orevkov and Mikhalkin in~\cite{MR3571386}.

 While $\mathbb{RP}^{3}$ is a natural ambient space for algebraic curves, classical knots are usually studied in $\mathbb{R}^{3}$ or the 3-sphere $S^{3}$. \onlythesis{Attempts have been made to study polynomial knots in $\mathbb{R}^{3}$, but such knots are
 inevitably long knots (they are not closed) and not really classical knots. They are related to
 their classical counterparts in $S^{3}$ by taking the one point compactification which closes the long
 knot. However, this does not treat $S^{3}$ as a real algebraic variety and the closed knot no longer
 has an algebraic structure. However, if we consider real rational curves then we}\onlypaper{We} may easily treat
 the $3$-sphere as a sub-variety of $\mathbb{RP}^{4}$ and consider non-singular real algebraic curves that in $\mathbb{RP}^{4}$ that lie inside $S^{3}$. The classification of real rational knots up to degree 6 in $S^{3}$ treated as a sub-variety of $\mathbb{RP}^{4}$ has been done in~\cite{MR3632321}. 

However, $S^3$ is not the only quadric in $\mathbb{RP}^{4}$. There are two non-empty, non-singular quadrics: one is of signature $(4,1)$ which homeomorphic to $3$-sphere $S^{3}$ and another is of signature $(3,2)$ which is homeomorphic to a $3$-dimensional hyperboloid, which we will denote as $Q_{3,2}$.  

This \mythesis{} studies real rational knots in the quadric of signature $(3,2)$, i.e. $Q_{3,2}$, i.e. non-singular real rational curves in $\mathbb{RP}^{4}$ which lie on the quadric $Q_{3,2}$. In $\mathbb{RP}^4$, this quadric may be realized as the zero locus of the quadratic equation $x_{0}^{2}=x_{1}^{2}+x_{2}^{2}+x_{3}^{2}-x_{4}^{2}$. We study the space of real rational curves up to degree $5$ in this quadric and  give a classification for the real rational knots and nodal curves of degree $\leq 5$ up to rigid isotopy. We also understand the correspondence between the real rational curves in the quadric and the real rational curves in $\mathbb{RP}^{3}$. \onlythesis{The gluing technique of $\mathbb{RP}^{3}$ proposed by Bj\"{o}rklund in \cite{MR2844809} is extended to $Q_{3,2}$ to construct real rational knots in $Q_{3,2}$. This gluing technique in the quadric helps us to construct the representatives for each rigid isotopy class of real rational knots up to degree $5$.}

\begin{thesisblock}
	In the last part of the thesis, we begin the study of the space of real rational knots in another quadric: $\mathbb{RP}^{1} \times \mathbb{RP}^{2}$. We provide a classification for some low bi-degree knots in this space.
\end{thesisblock}
Now, we will introduce some basic definitions and terms in real rational knot theory which are used in this \mythesis{}:
\begin{definition} 
	A \textbf{real rational curve} of degree $d$ is an embedding of $\mathbb{RP}^{1}$ in $\mathbb{RP}^{4}$ given by homogeneous polynomials $p_{i}'s$ such that 
	\[C: \mathbb{RP}^{1} \rightarrow \mathbb{RP}^{4}\]
	\[[s:t]\rightarrow [p_{0}[s:t]:p_{1}[s:t]:p_{2}[s:t]:p_{3}[s:t]:p_{4}[s:t]]\]
	where $p_{i}'s$ are polynomials of degree $d$ and they have no common factor. Then, this $C$ is called a real rational parametrisation of degree $d$.
	
	When we extend this map to $\mathbb{CP}^{1}$, the image of this function is called complexification of the curve.
\end{definition}

\begin{definition}
	A \textbf{real rational knot} of degree $d$ is a real rational curve of degree $d$ without any singularities.
\end{definition}

In classical knot theory, curves are classified up to smooth isotopy i.e. deformation of one knot to another knot via path of knots.

In real algebraic knot theory, real rational curves are studied up to rigid isotopy i.e. deformation of one real rational knot to another real rational knot of same degree via a path of real rational knots of same degree.

\begin{definition}
	Two real rational knots of degree $d$ are said to be \textbf{rigidly isotopic} if there is a path between these knots which lies completely in the space of real rational knots of degree $d$.
\end{definition}

Rigid isotopy is a stronger notion than smooth isotopy as it preserves the degree of the real rational knot. Viro introduced a notion of the encomplexed writhe of a real algebraic knot in \cite{MR1819192} which is a rigid isotopy invariant in $\mathbb{RP}^{3}$.

\begin{definition}
	The \textbf{encomplexed writhe} of a real rational knot $K$ in $\mathbb{RP}^{3}$ is the sum of signs over all the double points in the projection of the complexification of knot $K$ on the projective plane $\mathbb{RP}^{2}$.
\end{definition}

By using this, we can define the encomplexed writhe $w$ for knots in $\mathbb{RP}^{4}$ lying on the quadric. For any knot $K$ lying on the quadric, consider a point $p$ on the quadric that does not lie on the knot. Project the knot from the point $p$ to a linear hypersurface in $\mathbb{RP}^4$, which will be a copy of $\mathbb{RP}^3$ (as discussed in lemma \ref{proj_not_on_curve}). The projection is a knot in $\mathbb{RP}^{3}$ and the encomplexed writhe of $K$ is defined as the encomplexed writhe of the projection. This is well-defined because projection of the knot from two different points are rigidly isotopic in $\mathbb{RP}^{3}$.
\begin{thesisblock}
\begin{definition}
	$K$ is a real rational curve of bi-degree $(m,n)$ in $\mathbb{RP}^{1} \times \mathbb{RP}^{2}$ if 
	\[K:\mathbb{RP}^{1} \rightarrow \mathbb{RP}^{1} \times \mathbb{RP}^{2}\]
	\[[s:t] \rightarrow ([p_{0}:p_{1}],[q_{0}:q_{1}:q_{2}]) \]
	where $p_{i}$'s \& $q_{i}$'s are homogeneous polynomials of degree $m$ and $n$ respectively.
	
	If $K$ is a non-singular curve then it is called real rational knot of bi-degree $(m,n)$.
\end{definition}
 
\begin{definition}
	Two real rational knots of bi-degree $(d_{1},d_{2})$ are said to be rigidly isotopic if there is a path between these knots which lies completely in the space of real rational knots of bi-degree $(d_{1},d_{2})$.
\end{definition}

\end{thesisblock}

\onlythesis{Now, we will briefly summarize the results of this thesis in the following section.

\section{Real rational knots in $Q_{3,2}$}}
The following theorem gives us a complete rigid isotopy classification of real rational knots up to degree $5$ in $Q_{3,2} \subset \mathbb{RP}^{4}$: 

\begin{theorem}
	There are $13$ rigid isotopy classes of real rational knots up to degree $5$ and their isotopy classes in $\mathbb{RP}^{3}$ after projection (as mentioned in lemma \ref{proj_not_on_curve}) are described below:
	\begin{enumerate}
		\item Degree~1: Two classes, both topologically isotopic to line and writhe number 0.
		\item Degree~2: Three classes, all topologically isotopic to circle i.e. unknots with same writhe 0.
		\item Degree~3: Two classes, both isotopic to a line with writhe number $\pm 1$. (Some of its representatives have writhe 1 and some have writhe -1).
		\item Degree ~4: Two classes, one is isotopic to the unknot and another one is isotopic to the two crossing knot with writhes $\pm1,\pm3$ respectively.
		\item Degree ~5: Four classes of degree 5 knots with writhes $0,2,4$ and ~6 up to reflections.
	\end{enumerate}
\end{theorem}

Note that encomplexed writhe is not an invariant for real rational knots lying on $Q_{3,2}$.

We will extend Bj\"{o}rklund's gluing method \cite{MR2844809}
to real rational knots in the quadric:

\begin{theorem}
	Two real rational knots of degree $d_{1}$ and $d_{2}$ in $Q_{3,2}$ intersecting at a point can be glued together to get a new real rational knot of degree $d_{1}+d_{2}$ in $Q_{3,2}$ which is rigidly isotopic to the union of the original knots except in a small neighbourhood of the intersection point. (See theorem \ref{gluing in hyperboloid} on page \pageref{gluing in hyperboloid})
\end{theorem}

This theorem helps us to show that all the knots up to degree~5 in the quadric can be obtained by gluing of lines in $Q_{3,2}$.

\begin{theorem}
	Every real rational knot of degree $d\leq 5$ in $\mathbb{RP}^{4}$ lying on the 3-dimensional hyperboloid is rigidly isotopic to a glued knot obtained by gluing a real rational knot of degree $d-1$ in $Q_{3,2}$ with a real rational knot of degree $1$ in $Q_{3,2}$. 
\end{theorem}

Every isotopy in $\mathbb{RP}^{3}$ can not be pulled back to the quadric. However, any given isotopy can be deformed into a complete isotopy that can be pulled back. The following theorem ensured that smooth isotopies in $\mathbb{RP}^{3}$ can be pulled back to the quadric:
\begin{theorem} 
	If $K_{1}$ and $K_{2}$ are two real rational knots of some degree in $Q_{3,2}$ such that their projections in $\mathbb{RP}^{3}$ are smoothly isotopic then $K_{1}$ and $K_{2}$ are also smoothly isotopic in $Q_{3,2}$.
	(proved in theorem \ref{B} on page \pageref{B})  
\end{theorem}

For degree~6, the complete classification of real rational knots is not yet known but we have an upper bound on the number of smooth isotopy classes.

\begin{theorem}
	There are at most $14$ smooth isotopy classes of degree $6$ real rational knots in $Q_{3,2}$.
\end{theorem}

\begin{thesisblock}
\section{Real rational knots in $\mathbb{RP}^{1} \times \mathbb{RP}^{2}$}
	\begin{theorem}
		The following theorem gives us the classification of real rational knots of some low bi-degrees in $\mathbb{RP}^{1} \times \mathbb{RP}^{2}$:
		\begin{enumerate}
			\item Bi-degree $(0,1)$: One rigid isotopy class smoothly isotopic to a line.
			\item Bi-degree $(1,0)$: Two rigid isotopy classes, both are smoothly isotopic to a line.
			\item Bi-degree $(1,1)$: Two rigid isotopy classes, both are smoothly isotopic to a line.
			\item Bi-degree $(0,2)$: Only one class smoothly isotopic to a circle.
			\item Bi-degree $(n,0), n \geq 2$: Does not exist any real rational knot.
			\item Bi-degree $(2,1)$: Two rigid isotopy classes, both are smoothly isotopic to circle.
			\item Bi-degree $(1,2)$: Two rigid isotopy classes, both are smoothly isotopic to lines
			\item Bi-degree $(2,2)$: Two rigid isotopy classes smoothly isotopic to circles
			\item Bi-degree $(0,n), n \geq 3$: Does not exist any real rational knot of this bi-degree
			\item Bi-degree $(3,1)$: Two rigid isotopy classes, smoothly isotopic to line
			\item Bi-degree $(1,3)$: Two rigid isotopy classes smoothly isotopic to circle.
			\item Bi-degree $(3,2)$: Two rigid isotopy classes, both are smoothly isotopic to line
			\item Bi-degree $(n,1)$: Two rigid isotopy classes, smoothly isotopic to line if ~n is odd and smoothly isotopic to circle if ~n is even
			\item Bi-degree $(n,2)$: Two rigid isotopy classes smoothly isotopic to line if ~n is odd and smoothly isotopic to circle if ~n is even
		\end{enumerate}
	\end{theorem}
	For bi-degree $(2,3)$, there are two classes for the real rational curves with exactly one double point in $\mathbb{RP}^{1} \times \mathbb{RP}^{2}$.\\
	The exact number of real rational knots of bi-degree $(2,3)$ is not known. But, it has been proved that number of rigid isotopy classes are either two or four.
\end{thesisblock}

\onlypaper{
		\begin{ack}
		The second author is grateful for the IISER Mohali PhD fellowship for its support.\\ 
		All the figures in this paper were created using GeoGebra~\cite{geogebra}.
	\end{ack}}

\chapter{Background}\label{chap-1}
	In this \myref{}, we will mention some previous results and theorems of real rational knot theory. \onlythesis{We also recall the background concepts that will be used throughout the thesis. This provides the necessary foundation for the results presented in the following \myref{}s. In the last \mysubref{} of this \myref{}, we present an analogy between knots on two dimensional hyperboloid and knots on the quadric $Q_{3,2}$. This comparison helps to understand the geometric ideas behind our results.
	\section{Real rational knots in $\mathbb{RP}^{3}$}}
	
	The number of double points in the real rational plane curve of degree $d$ are exactly $\frac{(d-1)(d-2)}{2}$ counted with multiplicity (including the complex ones). So, the encomplexed writhe $w(K)$ defined in \cite{MR1819192} of a real rational knot of degree $d$ in $\mathbb{RP}^{3}$  satisfies the following:
	\[ w(K)= \frac{(d-1)(d-2)}{2} \mathrm{mod}\  2\]
	\[|w(K)| \leq \frac{(d-1)(d-2)}{2} \]
	Real rational knots in $\mathbb{RP}^{3}$ were studied by Bj\"{o}rklund in \cite{MR2844809} and proved the classification of real rational knots up to degree $5$ as follows:  
	\begin{theorem}	In $\mathbb{RP}^{3}$, the classification of real rational knots is as follows: 
	\begin{enumerate}
		\item Degree 1: One class with writhe $0$ isotopic to line
		\item Degree 2: One class with writhe $0$ isotopic to circle
		\item Degree 3: Two classes, reflections of each other isotopic to the line with writhe $\pm1$
		\item Degree 4: Four classes, isotopic to the line and the two crossing knot with writhe $\pm1, \pm3$
		\item Degree 5: 7 classes with writhe $0,\pm2,\pm4, \pm6$
	\end{enumerate}
\end{theorem}

He proved the following ``gluing theorem'' that allows one to construct real rational knots in $\mathbb{RP}^{3}$ from knots of lower degree. 

\begin{theorem}\label{glu}
	Suppose $C_{1}$ and $C_{2}$ are two oriented real rational knots of degree $d_{1}$ and $d_{2}$ respectively in $\mathbb{RP}^{3}$ such that they intersect transversally at a point say $p$ then there exists knot parametrisations $[p_{o}:p_{1}:p_{2}:p_{3}]$ and $[q_{o}:q_{1}:q_{2}:q_{3}]$ such that $[p_{o}q_{0}:p_{1}q_{0}+p_{0}q_{1}:p_{2}q_{0}+p_{0}q_{2}:p_{3}q_{0}+q_{3}p_{0}]$ is a parametrisation of degree $d_{1}+d_{2}$ which defines a knot that is rigidly isotopic to the union of the two knots $C_{1}$ and $C_{2}$ except in a small neighbourhood of the point $p$. The resulting knot is the knot obtained by smoothing of strands of $C_{1}$ and $C_{2}$ at the intersection point.
\end{theorem}
By using this gluing method, he was able to construct the representatives of real rational knots up to degree $5$ with all possible values of encomplexed writhe. He also proved that encomplexed writhe is a complete invariant for real rational knots up to degree $5$ in $\mathbb{RP}^{3}$. He also constructed an example of two degree~6 non rigid isotopic knots having same writhe numbers which proves that encomplexed writhe is not a complete invariant for higher degrees.

In 2016, Mikhalkin and Orevkov study the knots and links of low degree in $\mathbb{RP}^{3}$ \cite{MR3571386}. They give the classification for the real algebraic curves of genus~0 and~1 up to degree~6. They proved that
\begin{theorem}
	There are~14 topological isotopy types of rational real algebraic knots of
	degree 6 embedded in the projective space $\mathbb{RP}^{3}$.
\end{theorem}

\chapter{Preliminaries}\label{chap-2}
In this \myref{}, we will study the transformations that can be applied to the real rational curves in $Q_{3,2}$. We also establish that the space of real rational curves in $Q_{3,2}$ forms a manifold. Further, we examine the correspondence between the space of real rational knots in $Q_{3,2}$ and the space of real rational knots in $\mathbb{RP}^{3}$. This \myref{} sets up the prerequisites to the main results of this \mythesis.


\section{Action of $O(3,2)$ on $Q_{3,2}$}
Projective linear transformations are the linear transformations which preserve the projective structure of the space. Projective transformations in $\mathbb{RP}^{4}$ are given by $GL_{5}(\mathbb{R}^{5})$ up to scalar multiples. $O(3,2)$ is the group of symmetries of $Q_{3,2}$, i.e. they are projective linear transformations in $\mathbb{RP}^4$ that preserve the quadric $Q_{3,2}$. 
Any matrix belongs to $O(3,2)$ if it satisfies following conditions:
\begin{center}
	\[\begin{pmatrix}
		a_{00} & a_{01} & a_{02} &a_{03}& a_{04}\\
		a_{10} & a_{11} & a_{12} &a_{13}& a_{14}\\
		a_{20} & a_{21} & a_{22} &a_{23}& a_{24}\\
		a_{30} & a_{31} & a_{32} &a_{33}& a_{34}\\
		a_{40} & a_{41} & a_{42} &a_{43}& a_{44}
	\end{pmatrix} \in O(3,2)\]
\end{center} if $-a_{0i}^{2}+	a_{1i}^{2}+a_{2i}^{2}+a_{3i}^{2}-a_{4i}^2=\begin{cases}-1,  & i=0,4\\
	1,  & i=1,2,3
\end{cases}  $\\
and  $-a_{0i}a_{0j}+a_{1i}a_{1j}+a_{2i}a_{2j}+a_{3i}a_{3j}-a_{4i}a_{4j}=0 \text{ for all } i,j \text{ such that } i\neq j$.\\

Now, we will recall a basic result about symmetric bilinear forms on finite-dimensional real vector spaces.
\begin{theorem}\label{A}
	Let $V$ be an $n$ dimensional vector space over the field of real numbers, and let $f$ be a symmetric bilinear form on $V$ which has rank $r$. Then there is an ordered basis $\{v_{1},v_{2},...,v_{n}\}$ for $V$ in which the matrix of $f$ is diagonal and such that $f(v_{j},v_{j})=\pm1, \; j=1,2,...,r$. Furthermore, the number of basis vectors $v_{j}$ for which $f(v_{j},v_{j})=1$ is independent of the choice of basis.
\end{theorem}

We will need the following theorem that says that such a map is transitive.

\begin{lemma} \label{trans}
	The linear transformation group $O(3,2)$ which preserves $Q_{3,2}$ acts transitively on $Q_{3,2}$.
\end{lemma}

\begin{proof}
	To prove this lemma, it is enough to show that we can always find a matrix in $O(3,2)$ which maps $[1:1:0:0:0]$ to any point $[p_{0}:p_{1}:p_{2}:p_{3}:p_{4}]$ on $Q_{3,2}$.  First, note that any matrix (not necessarily in $O(3,2)$) that takes $[1:1:0:0:0]$ to $[p_0:p_1:p_2:p_3:p_4]$ must be of the form,

		\[\begin{pmatrix}
			a_{00} & p_{0}-a_{00} & . &.& .\\
			a_{10} & p_{1}-a_{10} & . &.&.\\
			a_{20} & p_{2}-a_{20} & . &.&.\\
			a_{30} & p_{3}-a_{30} & . &.&.\\
			a_{40} & p_{4}-a_{40} & . &.& .
		\end{pmatrix}
                \]
so that it satisfies,
	\begin{center}
		\[\begin{pmatrix}
			a_{00} & p_{0}-a_{00} & . &.& .\\
			a_{10} & p_{1}-a_{10} & . &.&.\\
			a_{20} & p_{2}-a_{20} & . &.&.\\
			a_{30} & p_{3}-a_{30} & . &.&.\\
			a_{40} & p_{4}-a_{40} & . &.& .
		\end{pmatrix}
		\begin{pmatrix}
			1\\
			1\\
			0\\
			0\\
			0
		\end{pmatrix}=
		\begin{pmatrix}
			p_{0}\\
			p_{1}\\
			p_{2}\\
			p_{3}\\
			p_{4}
		\end{pmatrix}\]
	\end{center}

        We now investigate the conditions imposed on entries so that it belongs to $O(3,2)$. This matrix belongs to $O(3,2)$ if, and only if, we can find $a_{i0}, 0 \leq i \leq 4$ satisfying $-a_{00}^{2}+	a_{10}^{2}+a_{20}^{2}+a_{30}^{2}-a_{40}^2=-1, -a_{00}p_{0}+a_{10}p_{1}+a_{20}p_{2}+a_{30}p_{3}-a_{40}p_{4}+1=0$. In this $[p_{0}:p_{1}:p_{2}:p_{3}:p_{4}]$ is fixed. These are two equations in $5$ variables $a_{00}, a_{10}, a_{20}, a_{30}$ and $a_{40}$ and $p_{0}^{2}=	p_{1}^{2}+p_{2}^{2}+p_{3}^{2}-p_{4}^2$. Intersection of the equations $-a_{00}^{2}+	a_{10}^{2}+a_{20}^{2}+a_{30}^{2}-a_{40}^2=-1, -a_{00}p_{0}+a_{10}p_{1}+a_{20}p_{2}+a_{30}p_{3}-a_{40}p_{4}+1=0$ is always non empty in $\mathbb{R}^{5}$.
	
	We can find the further columns by using following:
	
	\begin{center}
		$x_{k}=y-\sum_{i=1}^{i=k-1} \frac{<x_{i},y>}{<x_{i},x_{i}>} x_{i}, \; 3 \leq k \leq 5$
	\end{center}
	
	where $x_{i}'s$ are previous columns, $y$ is any vector and $<x_{i},y>=-x_{i0}y_{0}+x_{i1}y_{1}+x_{i2}y_{2}+x_{i3}y_{3}-x_{i4}y_{4}$. Now $\text{B}= \{x_{i}, \; 1\leq i \leq 5\}$ is a set of linear independent vectors. Hence, forms a basis for $\mathbb{R}^{5}$. With respect to standard basis, bilinear form $<x,y>=-x_{1}y_{1}+x_{2}y_{2}+x_{3}y_{3}-x_{4}y_{4}$ is represented by the diagonal matrix $\text{diag}(-1,1,1,1,-1)$. By using theorem \ref{A}, we can say that in $B$ there are exactly two vectors such that $<x_{i},x_{i}>=-1$ and the remaining three have value 1.
\end{proof}

\begin{proof}[Alternative proof] We can also prove this by using Witt's theorem \cite{EN}. \\
	To prove this lemma, it is enough to prove that there is a transformation from $\mathbb{RP}^{4}$ to $\mathbb{RP}^{4}$ which preserves $Q_{3,2}$ and maps $[1:1:0:0:0]$ to any other point  $[p_{0}:p_{1}:p_{2}:p_{3}:p_{4}]$ on $Q_{3,2}$. Points in $\mathbb{RP}^{4}$ corresponds to lines passing through origin in $\mathbb{R}^{5}$ i.e. one dimensional subspaces of $\mathbb{R}^{5}$.\\
	Consider the  finite dimensional vector space $\mathbb{R}^{5}$ equipped with the quadratic form $q(x_{0},x_{1},x_{2},x_{3},x_{4})=-x_{0}^{2}+x_{1}^{2}+x_{2}^{2}+x_{3}^{2}-x_{4}^{2}$. As $[1:1:0:0:0]$ and $[p_{0}:p_{1}:p_{2}:p_{3}:p_{4}]$ lie on $Q_{3,2}$,  $q(1,1,0,0,0)=q(p_{0},p_{1},p_{2},p_{3},p_{4})=0$. Consider a subspace $U$ spanned by the vector $(1,1,0,0,0)$ and a subspace $U^{'}$ spanned by the vector $(p_{0},p_{1},p_{2},p_{3},p_{4})$. Define a linear transformation $T$ as\[T: U \rightarrow U^{'}\] 
	\[T((1,1,0,0,0))=  (p_{0},p_{1},p_{2},p_{3},p_{4})\]
	It is an isometry between $U$ and $U^{'}$ as $q(1,1,0,0,0)=q(p_{0},p_{1},p_{2},p_{3},p_{4})=0$. By using Witt's theorem, this map can be extended to an isometry say $F$ of $\mathbb{R}^{5}$. So, under this map $F: \mathbb{R}^{5} \rightarrow \mathbb{R}^{5}$, set $q^{-1}\{0\}$ is invariant. \[q(cx_{0},cx_{1},cx_{2},cx_{3},cx_{4})=c^{2}q(x_{0},x_{1},x_{2},x_{3},x_{4})\]
	Lines passing through origin in $\mathbb{R}^{5}$ corresponds to points in $\mathbb{RP}^{4}$. So, this map $F$ induces a map in $\mathbb{RP}^{4}$ which preserves the set  $\{[x_{0}:x_{1}:x_{2}:x_{3}:x_{4}]; -x_{0}^{2}+x_{1}^{2}+x_{2}^{2}+x_{3}^{2}-x_{4}^{2}=0\}$ which is $Q_{3,2}$ and maps $[1:1:0:0:0]$ to the point  $[p_{0}:p_{1}:p_{2}:p_{3}:p_{4}]$. 
	
\end{proof}

\section{Real rational knots in $Q_{3,2}$ and their relation to knots in $\mathbb{RP}^{3}$}

Let $\mathcal{C}_{d}$ denote the space of real rational curves of degree $d$ lying in $Q_{3,2}$. For any $C \in \mathcal{C}_{d}$, we project it from a point $p$ on $Q_{3,2}$ to a curve in a hyperplane in $\mathbb{RP}^{4}$ which is a copy of $\mathbb{RP}^{3}$. This projection map $\pi_{p}: Q_{3,2}\setminus \{p\} \rightarrow \mathbb{RP}^{3}$ projects any point $q$ on $Q_{3,2}$ to the intersection point of the line joining $p$ and $q$ 
with $\mathbb{RP}^{3}$ (on which we are projecting). For the image, we consider the tangent plane at the point $p$ denoted by $T_{p}$ and this map blows up $p$ to a projective plane in $\mathbb{RP}^{3}$, which we call the ``plane at infinity''. This map can be extended to $\mathbb{C}Q_{3,2}$. $\mathbb{C}T_{p} \cap \mathbb{C}Q_{3,2}$ is the complexification of a real cone with apex at point $p$ and this cone blows down to a conic (we call it a \emph{special conic}) lying on the plane at infinity. This map gives us a bijection between the points of $Q_{3,2} \setminus T_{p}$ and $\mathbb{RP}^{3} \setminus \text{plane at infinity}$. Note that here the special conic depends on the choice of point of projection $p$ and the hyperplane of projection. $\mathbb{C}C \cap \mathbb{C}T_{p}$ are exactly $d$ points when counted with multiplicity and these points lie on the cone.

\begin{lemma}\label{proj_not_on_curve}
	There is a bijective correspondence between curves in $\mathcal{C}_{d}$ and the real rational curves of degree $d$ in $\mathbb{RP}^{3}$ which intersects the plane at infinity at $d$ points and all these points lie on the special conic.
\end{lemma}

\begin{proof}
	Let $C$ be a real rational curve of degree $d$ parametrised by $[p_{0}:p_{1}:p_{2}:p_{3}:p_{4}]$ in $\mathbb{RP}^{4}$ lying on the 3-dimensional hyperboloid given by $x_{0}^{2}=x_{1}^{2}+x_{2}^{2}+x_{3}^{2}-x_{4}^{2}$. By applying some projective transformation we will map the point of projection $p$ to $[1:1:0:0:0]$. So, from now on we will take the point of projection $p$ as $[1:1:0:0:0]$ and this point does not lie on the curve. The tangent plane $T_{p}$ at point $p$ is given by $x_{0}-x_{1}=0$ and the cone of intersection is $x_{2}^{2}+x_{3}^{2}=x_{4}^{2},  x_{0}=x_{1}$. Choose the plane of projection to be $x_{0}=0$. So, the map, $\pi_{p}$ is given by
	\[\pi_{p}: Q_{3,2}\setminus{p} \rightarrow \mathbb{RP}^{3}\] 
	\[\pi_{p}( [x_{0}:x_{1}:x_{2}:x_{3}:x_{4}])= [0:x_{1}-x_{0}:x_{2}:x_{3}:x_{4}] \]
	
	The projection of the cone is $x_{2}^{2}+x_{3}^{2}=x_{4}^{2}, x_{0}=x_{1}=0$ which lies on the plane at infinity $x_{0}=x_{1}=0$.
	
	The projection of the curve is given by $[0:p_{1}-p_{0}:p_{2}:p_{3}:p_{4}]$. This projected curve intersects the plane at infinity in $d$ points which are the projections of the $\mathbb{C}C\cap \mathbb{C}T_{p}$. So the degree of the projected curve is $d$ and all these points lie on the special conic.
	
	We can also pull back the curves of degree $d$ in $\mathbb{RP}^{3}$ all of whose points of intersection with the plane at infinity lies on the special conic to the real rational curve of degree $d$ in $Q_{3,2} \in \mathbb{RP}^{4}$. This pull back is given by the inverse of the stereographic projection $[0:x_{1}:x_{2}:x_{3}:x_{4}]\rightarrow [x_{1}^{2}+x_{2}^{2}+x_{3}^{2}-x_{4}^{2}:-x_{1}^{2}+x_{2}^{2}+x_{3}^{2}-x_{4}^{2}:-2x_{1}x_{2}:-2x_{1}x_{3}:-2x_{1}x_{4}]$.
	
	Suppose $[0:p_{1}:p_{2}:p_{3}:p_{4}]$ is the real rational curve of degree $d$ in $\mathbb{RP}^{3}$ whose $d$ points lie on the special conic. These $d$ points lie on $x_{2}^{2}+x_{3}^{2}=x_{4}^{2}, x_{0}=x_{1}=0$ so $p_{1}^{2}+p_{2}^{2}+p_{3}^{2}-p_{4}^{2},-p_{1}^{2}+p_{2}^{2}+p_{3}^{2}-p_{4}^{2},-2p_{1}p_{2},-2p_{1}p_{3}$ and $-2p_{1}p_{4}$ have exactly $d$ common factors. After dividing out those common factors, we get polynomials of degree $d$ with no more common factors. So the pull back is a real rational curve of degree $d$ in $Q_{3,2}$.
\end{proof}

\begin{lemma}\label{proj_curve_pt}
	There is a bijective correspondence between the curves in $\mathcal{C}_{d}$ and the real rational curves of degree $d-1$ in $\mathbb{RP}^{3}$ which intersects the plane at infinity at $d-1$ points and out of these points $d-2$ lie on the special conic.
\end{lemma}

\begin{proof}
	Take a curve $C^{'} \in \mathcal{C}_{d}$ and a point $q$ on the curve $C^{'}$ which is the point of projection. By using a projective transformation, we can transform the point of projection to $[1:1:0:0:0]$. Suppose that the image of the curve $C^{'}$ is a curve $C\in \mathcal{C}_{d}$ that is parametrised by $[p_{0}:p_{1}:p_{2}:p_{3}:p_{4}]$. The tangent plane at point $[1:1:0:0:0]$ is given by $x_{1}=x_{0}$. The curve C passes through the point $[1:1:0:0:0]$ with multiplicity $m\geq2$ (say). The remaining $d-m$ points of intersection lie on the cone of intersection. Choose the plane of projection to be $x_{0}=0$. The projection of the curve from $[1:1:0:0:0]$ on $x_{0}=0$ is given by $[0:p_{1}-p_{0}:p_{2}:p_{3}:p_{4}]$. The projection of $T_{p}$ is the plane at infinity given by $x_{0}=x_{1}=0$.
	
	Suppose $[\alpha:\beta]$ is the preimage of $[1:1:0:0:0]$ under parametrisation of curve $C$. Then, $(\beta s-\alpha t)$ divides $p_{1}-p_{0},p_{2},p_{3}$ and $p_{4}$. Let $q= \frac{p_{1}-p_{0}}{(\beta s-\alpha t)},q_{i}=\frac{p_{i}}{(\beta s-\alpha t)}, i=2,3,4$. So, the projected curve in $\mathbb{RP}^{3}$ is given by $[0:q:q_{2}:q_{3}:q_{4}]$ and this is a curve of degree $d-1$. The curve $[p_{0}:p_{1}:p_{2}:p_{3}:p_{4}]$ lies on $Q_{3,2}$, therefore $p_{0}^{2}=p_{1}^{2}+p_{2}^{2}+p_{3}^{2}-p_{4}^{2}$. This implies that $\frac{p_{0}^{2}}{(\beta s-\alpha t)^{2}}=\frac{p_{1}^{2}}{(\beta s-\alpha t)^{2}}+\frac{p_{2}^{2}}{(\beta s-\alpha t)^{2}}+\frac{p_{3}^{2}}{(\beta s-\alpha t)^{2}}-\frac{p_{4}^{2}}{(\beta s-\alpha t)^{2}}$. $(\beta s-\alpha t)$ divides $p_{0}-p_{1}$ with multiplicity $m$ so $(\beta s-\alpha t)$ divides  $\frac{p_{0}-p_{1}}{(\beta s-\alpha t)^{2}}$ with multiplicity $m-2$. The projection of these $m-2$ points lie on the special conic. Other $d-m$ points of intersection lie on the cone which get projected to the special conic and the point $p$ gets projected to the point where the tangent of the curve at point $p$ intersects the plane of projection.
	
	We can also pull back the curves of degree $d-1$ in $\mathbb{RP}^{3}$ whose $d-2$ points of intersection with the plane at infinity lies on the special conic to the real rational curve of degree $d$ in $Q_{3,2} \in \mathbb{RP}^{4}$. This pull back is given by the inverse of the stereographic projection $[0:x_{1}:x_{2}:x_{3}:x_{4}]\rightarrow [x_{1}^{2}+x_{2}^{2}+x_{3}^{2}-x_{4}^{2}:-x_{1}^{2}+x_{2}^{2}+x_{3}^{2}-x_{4}^{2}:-2x_{1}x_{2}:-2x_{1}x_{3}:-2x_{1}x_{4}]$.
	
	Suppose $[0:p_{1}:p_{2}:p_{3}:p_{4}]$ is the real rational curve of degree $d-1$ in $\mathbb{RP}^{3}$ whose $d-2$ points lie on the special conic. These $d-2$ points lie on $x_{2}^{2}+x_{3}^{2}=x_{4}^{2}, x_{0}=x_{1}=0$ so $p_{1}^{2}+p_{2}^{2}+p_{3}^{2}-p_{4}^{2},-p_{1}^{2}+p_{2}^{2}+p_{3}^{2}-p_{4}^{2},-2p_{1}p_{2},-2p_{1}p_{3}$ and $-2p_{1}p_{4}$ have exactly $d-2$ common factors. After dividing out those common factors, we get polynomials of degree $d$ with no more common factors. So the pull back is a real rational curve of degree $d$ in $Q_{3,2}$.
	
\end{proof}

When we try to pull back a complete path from $\mathbb{RP}^{3}$ to $Q_{3,2}$, it may be possible that at some time, curve is tangential to plane at infinity. During this, the following lemma helps us to the pull back the isotopy.

\begin{lemma} \label{tangent}
	If $C$ is a curve of degree $d$ in $\mathbb{RP}^{3}$ such that it intersects the plane at infinity tangentially at a point which lies on the  special conic with multiplicity $m$ but is not tangent to the special conic at that point and the remaining points of intersection lie on the special conic then the pull back of this type of curve in $Q_{3,2}$ is a real rational curve of degree $d+m-1$ in $Q_{3,2}$.
\end{lemma}

\begin{proof}
	Suppose the curve $C$ in $\mathbb{RP}^{3}$ is parametrised by $[0:p_{1}:p_{2}:p_{3}:p_{4}]$ where $p_{i}$'s are polynomials of degree $d$. By using some projective transformation, consider the plane at infinity to be $x_{1}=0$ and the conic at infinity to be $x_{2}^{2}+x_{3}^{2}=x_{4}^{2}$. The given curve is tangent to the $x_{1}=0$ at a point say $p$ which is image of $[\alpha:\beta]$ but not to the special conic and the remaining points of intersection lie on the special conic $x_{2}^{2}+x_{3}^{2}=x_{4}^{2}$, $x_{1}=0$. The pull back is given by the inverse of the stereographic projection
	$[0:x_{1}:x_{2}:x_{3}:x_{4}]\rightarrow [x_{1}^{2}+x_{2}^{2}+x_{3}^{2}-x_{4}^{2}:-x_{1}^{2}+x_{2}^{2}+x_{3}^{2}-x_{4}^{2}:-2x_{1}x_{2}:-2x_{1}x_{3}:-2x_{1}x_{4}]$. As the curve is tangent at $p$ so $\alpha s-\beta t$ is a factor of $p_{1}$ with multiplicity $m \geq2$ but the curve is not tangent to special conic so $\alpha s-\beta t$ is a factor of $p_{2}^{2}+p_{3}^{2}-p_{4}^{2}$ with multiplicity $1$. Remaining $d-m$ points of intersection also contributes to the common factors $p_{1}^{2}+p_{2}^{2}+p_{3}^{2}-p_{4}^{2},-p_{1}^{2}+p_{2}^{2}+p_{3}^{2}-p_{4}^{2},-2p_{1}p_{2},-2p_{1}p_{3}$ and $-2p_{1}p_{4}$. So, $p_{1}^{2}+p_{2}^{2}+p_{3}^{2}-p_{4}^{2},-p_{1}^{2}+p_{2}^{2}+p_{3}^{2}-p_{4}^{2},-2p_{1}p_{2},-2p_{1}p_{3}$ and $-2p_{1}p_{4}$ have $d-m+1$ common factors. After dividing out these common factors, we get a parametrisation of degree $d+m-1$ in $Q_{3,2}$, which is the pull back of given curve $C$. 
\end{proof}

\begin{lemma}\label{proj_double_pt}
	\begin{enumerate}
		\item There is a bijective correspondence between the real rational curves of degree $d$ with one double point in $Q_{3,2}$ and curves of degree $d-2$ in $\mathbb{RP}^{3}$ which intersects the plane at infinity at $d-2$ points out of which $d-4$ points lie on the special conic formed by the projection of the cone, which is the intersection of the $Q_{3,2}$ and tangent plane at the point of projection.
		\item The projection of a degree $d$ curve with a cusp in $Q_{3,2}$ from the point of cusp is a degree $d-2$ curve in $\mathbb{RP}^{3}$ which intersects the plane at infinity tangentially (with multiplicity $2$) on a point which does not lie on the special conic.
	\end{enumerate}
\end{lemma}

\begin{proof}
	\begin{enumerate}
		\item 
		Suppose $C$ be a real rational curve of degree $d$ with one double point parametrised by $[p_{0}:p_{1}:p_{2}:p_{3}:p_{4}]$ where $p_{i}$'s are homogeneous polynomials of degree $d$. By using some projective transformation which preserves $Q_{3,2}$ we can map its double point to point $p$ given by $[1:1:0:0:0]$. Tangent plane $T_{p}$ at point [1:1:0:0:0] to the hyperboloid $x_{1}^{2}+x_{2}^{2}+x_{3}^{2}-x_{4}^{2}=x_{0}^{2}$ is given by $x_{0}-x_{1}=0$. $\mathbb{C}T_{p} \cap \mathbb{C}Q_{3,2}$ is a cone given by $x_{2}^{2}+x_{3}^{2}-x_{4}^{2}=0$, $x_{0}=x_{1}$. Curve has a double point at p so each branch of curve intersects $T_{p}$ with multiplicity at least 2. Suppose two branches of the curve intersects $T_{p}$ at point p with multiplicity $a_{1}$ and $a_{2}$ respectively and other $d-a_{1}-a_{2}$ points of intersection lie on the cone. Let us choose the plane of projection to be $x_{0}=0$. Projection from [1:1:0:0:0] is given by $[x_{0}:x_{1}:x_{2}:x_{3}:x_{4}] \rightarrow [0:x_{1}-x_{0}:x_{2}:x_{3}:x_{4}]$. Under this projection $T_{p}$ get mapped to a hyperplane in $\mathbb{RP}^{3}$ say $X_{p}$ given by $x_{0}=x_{1}=0$. Projection of $\mathbb{C}T_{p} \cap \mathbb{C}Q_{3,2}$ is the conic given by $x_{2}^{2}+x_{3}^{2}-x_{4}^{2}=0, x_{0}=x_{1}=0$ which lies on $X_{p}$. Now double point may be solitary or non solitary.
		
		If double point is non solitary we can choose parametrisation in such a way that $[1:0]$ and $[0:1]$ get mapped to $[1:1:0:0:0]$. Therefore, $s, t$ divide $p_{1}-p_{0}, p_{2}, p_{3}, p_{4}$ i.e. $st$ divides  $p_{1}-p_{0}, p_{2}, p_{3}, p_{4}$. Under the projection map, curve $[p_{0}:p_{1}:p_{2}:p_{3}:p_{4}]$ got mapped to $[0:p_{1}-p_{0}:p_{2}:p_{3}:p_{4}]$. $p_{1}-p_{0},p_{2},p_{3},p_{4}$ all have two common factors $s$ and $t$ so image curve is given by  $[0:\frac{p_{1}-p_{0}}{st}:\frac{p_{2}}{st}:\frac{p_{3}}{st}:\frac{p_{4}}{st}]$ which is of degree d-2. Now we check the number of points of intersection of the image curve $[0:\frac{p_{1}-p_{0}}{st}:\frac{p_{2}}{st}:\frac{p_{3}}{st}:\frac{p_{4}}{st}]$ with the conic given by $x_{2}^{2}+x_{3}^{2}-x_{4}^{2}=0$, $x_{0}=x_{1}=0$. As the curve lies on $Q_{3,2}$ so $p_{0}^{2}=p_{1}^{2}+p_{2}^{2}+p_{3}^{2}-p_{4}^{2}$. Therefore, $p_{0}^{2}-p_{1}^{2}=p_{2}^{2}+p_{3}^{2}-p_{4}^{2}$. $s, t$ divides $p_{1}-p_{0}$ with $a_{1},a_{2} \geq 2$ multiplicity respectively. Therefore, $s, t$ divides $\frac{p_{1}-p_{0}}{(st)^{2}}$ with $a_{1}-2,a_{2}-2$ multiplicity. At these $a_{1}+a_{2}-4$ points, the image curve intersects $x_{0}=0=x_{1}$ and these points lie on the special conic as   $\frac{(p_{0}-p_{1})(p_{0}+p_{1})}{s^{2}t^{2}}=\frac{p_{2}^{2}}{s^{2}t^{2}}+\frac{p_{3}^{2}}{s^{2}t^{2}}-\frac{p_{4}^{2}}{s^{2}t^{2}}$. Other $d-a_{1}-a_{2}$ points of intersection of the curve with $T_{p}$ also get projected to the conic. In all these $d-4$ points lie on the special conic. So the image curve is of degree $d-2$ in $\mathbb{RP}^{3}$ and it intersects the plane at infinity in $d-2$ points out of which $d-4$ points lie on the special conic.
		
		If double point is solitary then suppose $[1:\iota]$ and $[1:-\iota]$ get mapped to [1:1:0:0:0] and proceed in the same way.
		\item Let $p$ be the point of projection. Notice the image of point $p$ under projection is the intersection of tangent line at $p$ with the plane of projection. In case the point of projection is a cusp, these images will coincide and the projected image will become tangent to the plane at infinity. Rest of things will follow from proof of part (1).    
		\end {enumerate}
	\end{proof}
	
	The special conic on the plane at infinity divides the plane at infinity in two components, one is outer component and another one is inner component.
	
	\begin{lemma} \label{rk1}
		In $\mathbb{RP}^{3}$, any real rational curve $C$, of degree $d \leq 4$ such that exactly $d-1$ points of intersection with the plane at infinity lie on the special conic and one point of intersection $p$ does not lie on the special conic (say inside the conic) then this can be isotoped to a curve isotopic to $C$ such that its exactly $d-1$ points of intersection lies on the special conic and one point lies outside the conic.
	\end{lemma}

	\begin{proof}
		
		Isotopy between these two can be constructed as follows: During the isotopy we are taking the point $p$ towards a point of intersection on the conic say $q$ (near to it) along the curve, correspondingly changes the plane at infinity which intersects at curve at $p$ and $q$. This plane intersects the curve at points $p$, $q$ and $d-2$ more points. $4$ points determines a pencil of conics. As $d\leq 4$, along with $q$ we can always able to choose the other three points such that no three are collinear and we can avoid the situation of degenerate conic and conic must be non empty. So, we can always choose a conic of signature $(2,1)$ passing through these $d-2$ points and the point $q$. During the isotopy when $p$ is merged with $q$, the plane at infinity becomes tangent at the point $q$ as shown in figure \ref{figure 5.1} (this can be pulled back to $Q_{3,2}$ by using lemma \ref{tangent}). To move $p$ outside the conic, now we move $p$ away from $q$ along the curve but in the opposite direction and choose the plane at infinity and the special conic as we did previously. This way we are done. 
		
		\begin{figure}[h]
			\centering
			\includegraphics[width=0.4\textwidth]{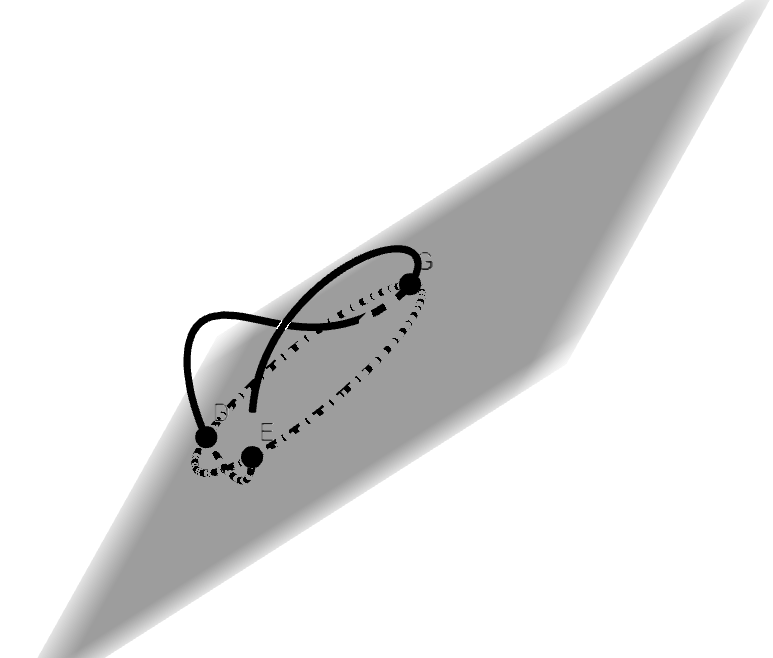}
			\includegraphics[width=0.4\textwidth]{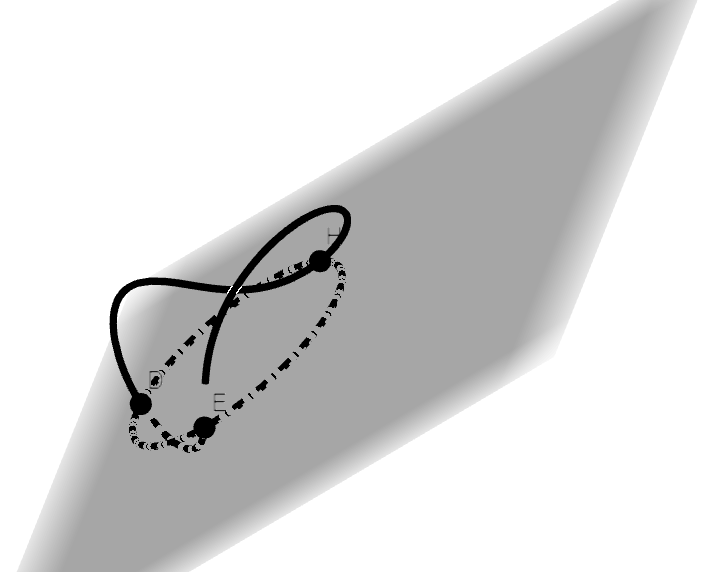}
			\includegraphics[width=0.4\textwidth]{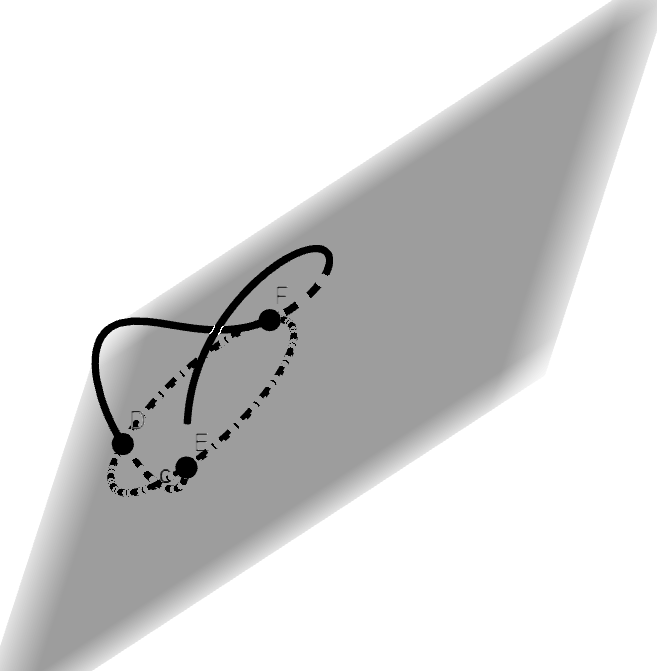}
			\caption{Curve of degree 4 having (a) one point not on the conic inside the conic, (b) plane at infinity tangent to the curve and (c) one point not on the conic outside the conic.}
			\label{figure 5.1}
		\end{figure}
	\end{proof}


	\section{The Space of Real Rational parametrisations of degree $d$ curves in $Q_{3,2}$:}
	The space of real rational parametrisations of degree $d$ in $\mathbb{RP}^{4}$ is a projective space of dimension $\mathbb{RP}^{5d+4}$. 
	
	Now, we state the following lemmas whose proofs are obvious.

	\begin{lemma}\label{1}
		If $h_{0}, h_{1},..., h_{d}$ are linearly independent polynomials of degree $\leq$ d,
		then the determinant of the following matrix will be non-zero:
		\begin{center}
			\[\begin{pmatrix}
				h_{0}(t_{1}) & h_{1}(t_{1}) & ... & h_{d}(t_{1})\\
				h_{0}(t_{2}) & h_{1}(t_{2}) & ... & h_{d}(t_{2})\\
				. & & & .\\
				. & ... & ... & .\\
				. & & & .\\
				h_{0}(t_{d+1}) &  h_{1}(t_{d+1}) & ... & h_{d}(t_{d+1})
			\end{pmatrix}\]
		\end{center}
		as long as all the $t_{i}$’s are distinct.
	\end{lemma}

	\begin{lemma} \label{2}
		Let $p_{0}(t)$ and $p_{1}(t)$ be two polynomials of degree d that do not share
		a root. Then the set 
		\[ \{p_{0}(t), t p_{0}(t), ..., t^{d-1}p_{0}(t), p_{1}(t), t p_{1}(t), ..., t^{d-1}p_{1}(t), t^{d}p_{1}(t), ..., t^{d+k} p_{1}(t)\} \]
		is linearly independent in the space of polynomials of degree 2d + k.
	\end{lemma}

	\begin{theorem}\label{dimension}
		The space of real parametrisations of degree $d$ curves in $\mathbb{RP}^{4}$ that lies on $3$-dimensional hyperboloid is a manifold of dimension $3(d+1)$.
	\end{theorem}

	\begin{proof}
		$[p_{0}:p_{1}:p_{2}:p_{3}:p_{4}]$ is a real rational parametrisation of a degree $d$ curve where $p_{i}'s$ are the homogeneous polynomials of degree d. Define $q(t)=p_{1}(t)^{2}+p_{2}(t)^{2}+p_{3}(t)^{2}-p_{4}(t)^{2}-p_{0}(t)^{2}$ for any $t \in \mathbb{RP}^{1} $. If for any real parametrisation, $q(t)$ is zero for 2d+1 points then that parametrisation completely lies on the $Q_{3,2}$. The space of real rational parametrisations of degree d in $\mathbb{RP}^{4}$ is a projective space of dimension 5d+4. Fix 2d+1 points $t_{1},t_{2},...,t_{2d+1}$ in $\mathbb{RP}^{1}$ So we define a map 
		\[f:\mathbb{R}^{5d+4}\rightarrow\mathbb{R}^{2d+1}\]
		\[f([p_{0}:p_{1}:p_{2}:p_{3}:p_{4}])=[q(t_{1}),q_(t_{2}),...,q(t_{2d+1})] \]
		$f^{-1}(0,0,0,...,0)$ is the set of real parametrisations of the curves which lies on the hyperboloid in $\mathbb{RP}^{4}$. By using regular value theorem, we will show that 	$f^{-1}(0,0,0,...,0)$ is a closed manifold of dimension $3d+3$. Here $[df]_{K}$ is a matrix of order $(2d+1) \times (5d+4)$ with entries $([df]_{K})_{i,j}= \pm 2p_{k}(t_{i})t_{i}^{l}$ where $a_{j}$ is the coefficient of $t^{l}$ in the polynomial $p_{k}$. Entries of the matrix is negative when $k=0,4$. To prove that this map $[df]_{K}$ is surjective we have to show that its rank is 2d+1 that is we have to find a minor of order 2d+1 which has non zero determinant.
		
		Firstly we suppose among $p_{i}^{'}$s, there are two polynomials which do not share any root say $p'$ and $p''$. Now by using lemma \ref{2} the set
		\[ \{p'(t), t p'(t), ..., t^{d-1}p'(t), p''(t), t p''(t), ..., t^{d-1}p''(t), t^{d}p''(t)\} \] 
		is linearly independent set of 2d+1 polynomials. Now by using lemma \ref{1} following matrix
		
		\begin{center}
			\[\begin{pmatrix}
				p'(t_{1})  & ...&t_{1}^{d-1} p'(t_{1})&p''(t_{1})&t_{1}p''(t_{1})&... &t_{1}^{d}p''(t_{1})\\
				p'(t_{2})  & ...&t_{2}^{d-1} p'(t_{2})&p''(t_{2})&t_{2}p''(t_{2})&... &t_{2}^{d}p''(t_{2})\\
				. & & & .& & & .\\
				. & ... & ... & .&...&...&.\\
				. & & & . & & &.\\
				p'(t_{2d+1})  & ...&t_{2d+1}^{d-1} p'(t_{2d+1})&p''(t_{2d+1})&t_{2d+1}p''(t_{2d+1})&... &t_{2d+1}^{d}p''(t_{2d+1})\\
			\end{pmatrix}\]
		\end{center}
		
		has non zero determinant. Notice that this is a $(2d+1) \times (2d+1)$ minor. So in this case $[df]_{K}$  is surjective.
		
		Now suppose among these five polynomials $p_{i}, 0 \leq i \leq 4$ any pair of polynomials share a common root then this means that the curve intersects two codimensional linear space. By using transversality theorem, we can find a diffeomorphism which preserves $Q_{3,2}$ by which we can avoid that intersection. Compose this with $f$ and we will do the same as above. This implies that $[df]_{K}$ is surjective for each $K \in f^{-1}(0,0,0,...,0)$ i.e. $(0,0,...,0)$ is a regular value. So $f^{-1}(0,0,0,...,0)$ is a manifold of dimension $3d+3$.
	\end{proof}

	\begin{theorem}\label{dim}
		The space of real rational parametrisations of degree $d \geq 4$ in $Q_{3,2}$ with exactly one double point is a manifold of dimension $3d+2$.
	\end{theorem}

	\begin{proof}
		Suppose $X_{d}$ denote the space of parametrisations of curves of degree d with at least one double point lying on $Q_{3,2}$ in $\mathbb{RP}^{4}$. $X^{'}_{d}$ denote the space of parametrisations of curves of degree d with more than one double point lying on $Q_{3,2}$ in $\mathbb{RP}^{4}$ and $X_{d}^{q}$ be the space of parametrisations of degree d curves with exactly one double point at point $q=[1:1:0:0:0]$. If C is a real rational parametrisation of curve of degree d with exactly one double point say p then we can define a map 
		\[g:X_{d} \setminus X^{'}_{d} \rightarrow Q_{3,2}\times X_{d}^{q}\]
		which take $C$ to the pair $(p,C^{'})$ where $C^{'}$ is the image of $C$ under a transformation which takes point $p$ to $q$ and preserves $Q_{3,2}$. To prove $g$ an isomorphism it is enough to show that there is a transformation which preserves $Q_{3,2}$ and maps point $p$ to $q$.
		
		For a given $p=[x_{0}:x_{1}:x_{2}:x_{3}:x_{4}]$ if we take a matrix whose first two columns are given by
		
		\begin{center}	
			\[\begin{pmatrix}
				a_{00}&x_{0}-a_{00}&.&.&.\\
				a_{10}&x_{1}-a_{10}&.&.&.\\
				a_{20}&x_{2}-a_{20}&.&.&.\\
				a_{30}&x_{3}-a_{30}&.&.&.\\
				a_{40}&x_{4}-a_{40}&.&.&.	
			\end{pmatrix}\]
		\end{center}
		
		satisfies the conditions 
		\[-a_{00}^{2}+a_{10}^{2}+a_{20}^{2}+a_{30}^{2}-a_{40}^{2}=-1\]
		\[-x_{0}a_{00}+x_{1}a_{10}+x_{2}a_{20}+x_{3}a_{30}-x_{4}a_{40}+1=0\]
		and the subsequent columns are determined as follows:
		
		Define
		\[<x,y>=-x_{0}y_{0}+x_{1}y_{1}+x_{2}y_{2}+x_{3}y_{3}-x_{4}y_{4}\]
		\[y'=[y_{0}:y_{1}:y_{2}:y_{3}:y_{4}]-\dfrac{-x_{0}y_{0}+x_{1}y_{1}+x_{2}y_{2}+x_{3}y_{3}-x_{4}y_{4}}{-x_{0}^{2}+x_{1}^{2}+x_{2}^{2}+x_{3}^{2}-x_{4}^{2}} [x_{0}:x_{1}:x_{2}:x_{3}:x_{4}] \]
		This is a process similar to Gram Schmidt's Process for orthogonalisation But here we had changed orthogonal condition. Continuing in this way we get a matrix which preserves $Q_{3,2}$ and maps $q$ to $p$.
		
		Now we will count dimension of $X_{d}^{q}$. Suppose $C \in X_{d}^{q}$ and $[\alpha : \beta],[\gamma: \delta]$ got mapped to [1:1:0:0:0] under the parametrisation $[p_{0}:p_{1}:p_{2}:p_{3}:p_{4}]$ of the curve C. $(\alpha s- \beta t),(\gamma s- \delta t)$ divides $p_{0}-p_{1},p_{2},p_{3},p_{4}$. $p_{2}=r_{2}m, p_{3}=r_{3}m, p_{4}=r_{4}m, p_{1}-p_{0}=rm$  where $m=(\alpha s- \beta t)(\gamma s- \delta t)$ and $r_{2},r_{3},r_{4}$ and $r$ are the homogeneous polynomials of degree d-2. This double point q with 2d-1 other points ensure that this curve lies on the $Q_{3,2}$. So this space $X_{d}^{q}$ is a manifold of dimension 3d-1.
		
		So $X_{d} \setminus X_{d}^{'}$ is a manifold of dimension 3d+2.
	\end{proof}

\chapter{Classification of knots on $Q_{3,2}$ in $\mathbb{RP}^{4}$}\label{chap-3}
In this \myref{}, we studied the space of real rational curves up to degree 5 which lies on $Q_{3,2}$ in $\mathbb{RP}^{4}$ and give the
classification of real rational knots up to rigid isotopy.
A two dimensional projective plane in $\mathbb{RP}^{4}$ is determined by three points. So any curve in $\mathbb{RP}^{4}$ having degree $\leq 2$ has to lie on two dimensional plane in $\mathbb{RP}^{4}$. Intersection of a $2$ dimensional plane in $\mathbb{RP}^{4}$ with $Q_{3,2}$ is either a conic or a pair of straight lines. So, knots up to degree $2$ lying on $Q_{3,2}$ have to be these of two.
\begin{lemma}\label{conic}
	A two dimensional plane intersects $Q_{3,2}$ at a conic. This conic is singular if and only if it contains a line of the cone.
\end{lemma}
\begin{proof}
	Suppose the conic contains a line of the cone. Conics which contains a line are always singular.
	
	Conversely, suppose the two dimensional plane $P$ intersects $Q_{3,2}$ in a singular conic. There is a pencil of hyperplanes containing that plane $P$. So, there must be a hyperplane $H$ which contains the plane $P$ and is tangent to $Q_{3,2}$. This hyperplane $H$ intersects $Q_{3,2}$ in a cone denoted by $C$. This means intersection of $P$ with $Q_{3,2}$ is contained in the cone $C$.  
\end{proof}

\section{Degree 1} Degree~1 knots are planar curves and lines are the only degree~1 planar curves in $\mathbb{RP}^{4}$. Some of these lines may lie on the quadric.

Suppose the image of $C$ under the projection is a line in $\mathbb{RP}^{3}$ which intersects the plane at infinity at a point which lies on the special conic. Denote the intersection point by $q$. Consider the unique 2-dimensional plane $X$ containing the line $C$ and the point of projection $p$. This plane also contains the line $L$ joining point $p$ and $q$ which is a line of the cone (intersection of $T_{p}$ and $Q_{3,2}$). By using lemma \ref{conic}, the intersection of $X$ and $Q_{3,2}$ is a singular conic which is a pair of straight lines containing the line $L$. The other line in this conic is the pull back of the line $C$ under the projection map from point $p$ and a degree $1$ curve in $Q_{3,2}$.

\begin{lemma}\label{lemma 1}
	In $Q_{3,2}$, degree~1 curves cannot have any double points.
\end{lemma}
\begin{proof}
	As degree~1 curves are planar curves so these type of curves are specifically lines. Lines must not contain any self intersections. So this lemma trivially holds.
\end{proof}

\begin{lemma}
	Any two degree~1 knots (considered without orientation) are rigidly isotopic in $Q_{3,2}$.
\end{lemma}
\begin{proof}
	
	Suppose $K_{1}$ and $K_{2}$ are two degree $1$ knots in $Q_{3,2}$. Consider points $p_{1}$ and $p_{2}$ on each of the lines $K_{1}$ and $K_{2}$ respectively. By using lemma \ref{trans}, we can find a projective transformation which maps $p_{1}$ to $p_{2}$. Under this transformation, the respective tangent planes of both the points will coincide. Now the knot $K_{1}$ and $K_{2}$ lie on the same cone which is the intersection of the tangent plane with $Q_{3,2}$. So, we can isotope $K_{1}$ and $K_{2}$ via a rotation around the axis of the cone. This proves a stronger result that any two degree $1$ knots are projectively equivalent: Any two degree $1$ knots can be transformed to each other by a projective linear transformation in $\mathbb{RP}^{4}$ which preserves $Q_{3,2}$.
	
	We will use alternative method to classify degree~1 knots which is also applicable for the higher degrees. If we project a degree one knot from a point which does not lie on the knot but lies on the $Q_{3,2}$, it projects to a degree 1 knot in $\mathbb{RP}^{3}$ which intersects the plane at infinity in one real point and that point lies on the special conic at infinity as described in lemma \ref{proj_not_on_curve}.
	
	Suppose $K_{0}$ and $K_{1}$ are any two knots of degree 1 on $Q_{3,2}$. Project them from the same point which does not lie on both of them then they get projected to degree 1 knots in $\mathbb{RP}^{3}$ whose point of intersection with the plane at infinity lies on the special conic. We can always choose the point of projection and plane of projection (copy of $\mathbb{RP}^{3}$) in such a way that the plane at infinity is $x_{3}=0$ and the special conic is $x_{0}^{2}=x_{1}^{2}+x_{2}^{2}$. Denote the points of intersection of projected curves with the plane at infinity by $\mu_{0}=[a_{0}:b_{0}:c_{0}:0]$ and $\mu_{1}=[a_{1}:b_{1}:c_{1}:0]$ respectively. Any two degree 1 knots in $\mathbb{RP}^{3}$ are rigidly isotopic. So there is a path $K_{t}, 0 \leq t \leq 1$ between $K_{0}$ and $K_{1}$. Every $K_{t}$ intersects the plane at infinity in one real point, say $\lambda_{t}=[a_{t}^{'}:b_{t}^{'}:c_{t}^{'}:0]$. Since we are dealing with degree 1 curves, which are lines, we can always choose an isotopy whose points of intersection with the plane at infinity avoid a specific line $x_{0}=0$ in the plane at infinity so that the points of intersection of the line with the plane at infinity are in a single affine chart throughout the isotopy. So, $a_{t}^{'} \neq 0$. As the conic is path connected so there exists a path $[a_{t}:b_{t}:c_{t}:0], a_{t} \neq 0$ between $\mu_{0}$ and $\mu_{1}$ along the conic. Define a translation in the plane at infinity which translates $\lambda_{t}$ to $\mu_{t}$ given by 
	\begin{center}
		\[\begin{pmatrix}
			1&0&0\\
			\frac{b_{t}}{a_{t}}-\frac{b_{t}^{'}}{a_{t}^{'}}&1&0\\
			\frac{c_{t}}{a_{t}}-\frac{b_{t}^{'}}{a_{t}^{'}}&0&1
		\end{pmatrix}\]
	\end{center} We extend this map to $\mathbb{RP}^{3}$ by defining it as 
	\begin{center}
		\[\begin{pmatrix}
			1&0&0&0\\
			\frac{b_{t}}{a_{t}}-\frac{b_{t}^{'}}{a_{t}^{'}}&1&0&0\\
			\frac{c_{t}}{a_{t}}-\frac{b_{t}^{'}}{a_{t}^{'}}&0&1&0\\
			0&0&0&1
		\end{pmatrix}\]
	\end{center}

	Clearly $T_{0}$ and $T_{1}$ are identity transformations. So $T_{t}(K_{t})$ gives us the required path which can be pulled back. So there is only rigid isotopy class of degree 1 knots smoothly isotopic to line whose projection is a line in $\mathbb{RP}^{3}$ of writhe $0$.

	One representative of this class can be given by $[s:0:s:t:t]$.
\end{proof}

\begin{lemma}\label{lemma 2}
	In $\mathbb{RP}^{3}$, there are two rigid isotopy classes of degree $1$ real rational oriented knots upto orientation which has one point lying on the special conic. 
\end{lemma}

\begin{proof}
	Degree~1 knots with same orientation can be isotoped to each other as above but a degree~1 knot cannot be isotoped to its reverse orientation knot. Any degree~1 knot can be isotoped to $K_{0}$ parametrised by $[s:t:t:0]$ and its reverse orientation knot $K_{1}$ is $[-s:t:t:0]$. Take the plane at infinity as $x_{0}=0$ so that the special conic is $x_{1}^{2}+x_{2}^{2}=x_{3}^{2}$. Both $K_{0}$ and $K_{1}$ intersects the plane at infinity at $[0:1:1:0]$ which lies on the special conic. The images of $[1:1]$ under the parametrisations of $K_{0}$ and $K_{1}$ are $[1:1:1:0]$ and $[-1:1:1:0]$. If there is an isotopy $K_{i}$ between $K_{0}$ and $K_{1}$, first coordinate $x_{0}$ has to change from $1$ to $-1$. Then, there must be some $\lambda_{0}$, $0 \leq \lambda_{0} \leq 1$, such that first coordinate $x_{0}=0$ for that $\lambda_{0}$ and then $K_{\lambda_{0}}$ will completely lie on the plane at infinity which cannot be pulled back to $Q_{3,2}$ to degree $1$ knot.
	
	So, a degree $1$ knot in $Q_{3,2}$ cannot be rigidly isotoped to the same knot with opposite orientation.
\end{proof}

\begin{theorem}
	The space of the parametrisations of all real rational curves of degree $1$ has two connected components.
\end{theorem}
\begin{proof}
	Proof directly follows from lemma \ref{lemma 1} and \ref{lemma 2}.
\end{proof}

\section{Degree 2} 
\begin{lemma}\label{lemma 4}
	In $Q_{3,2}$, there does not exist any degree~2 real rational curve with one or more double points.
\end{lemma}

\begin{proof}
	Suppose $C$ is a real rational curve of degree~2 in $Q_{3,2}$ with at least one double point. When we project this curve $C$ from a point on the curve, which is not the double point, on a copy of $\mathbb{RP}^{3}$ we get a degree $1$ curve with at least double point in $\mathbb{RP}^{3}$ which is not possible. Hence, there does not exist any real rational curve of degree~2 with one or more double points in $Q_{3,2}$. 
\end{proof}

In this degree, one can see the correspondence between the conics in $Q_{3,2}$ and lines in $\mathbb{RP}^3$ geometrically:

Suppose $C$ is a line in $\mathbb{RP}^{3}$ which intersects the plane at infinity at a point $q$ which does not lie on the special conic. Consider the unique 2-dimensional plane $X$ containing the line $C$ and the point of projection $p$.  This plane does not contain any line of the cone (intersection of $T_{p}$ and $Q_{3,2}$) so by lemma \ref{conic}, the intersection of $X$ with $Q_{3,2}$ is a non-singular conic which is the pull back of $C$ in $Q_{3,2}$. 

We now give a geometric argument for the rigid isotopy classes before giving an algebraic proof: Any conic in $\mathbb{RP}^{4}$ which is an intersection of a two dimensional plane and $Q_{3,2}$, is completely determined by the two dimensional plane. Consider the pencil of hyperplanes containing the two dimensional plane in $\mathbb{RP}^{4}$. The conic is singular if and only if the two dimensional plane that it lies on is the intersection of two such hyperplanes which are tangent to $Q_{3,2}$. Let us consider the dual space of $\mathbb{RP}^{4}$. A pencil of hyperplanes in $\mathbb{RP}^{4}$ represents a line in the dual space and tangent planes to $Q_{3,2}$ represents a quadric $Q$ of signature $(3,2)$ in the dual space. Suppose $q$ is the quadratic form associated to the quadric $Q$. Any line in the dual space intersects the quadric $Q$ in either two real points or two imaginary points. 

$q$ can not be zero for any point on the line which intersects the quadric $Q$ in imaginary points, either $q$ is positive for all the points lying on such lines or $q$ is negative for all the points on the line. In the dual space, the space of lines which intersects the quadric in imaginary points have two components. Any two points on such lines in the actual space corresponds to two hyperplanes in $\mathbb{RP}^{4}$ which intersects $Q_{3,2}$ in a non-singular conic. Therefore, there are exactly two classes of non-singular conics in $Q_{3,2}$.

The above argument maybe summarized in this lemma, which we shall now prove algebraically.
\begin{lemma}
	There are two rigid isotopy classes of degree ~2 knots (considered without orientation) in $Q_{3,2}$.
\end{lemma}
\begin{proof}
	
	Suppose $K$ is a knot of degree 2 in $Q_{3,2} \in \mathbb{RP}^{4}$. By using lemma \ref{proj_curve_pt}, the projection of the curve from a point on the knot is a knot of degree~1 in $\mathbb{RP}^{3}$ such that point of intersection of the plane at infinity with the projected knot does not lie on the special conic. The special conic divides the plane at infinity in two components. So, the intersection point of the knot with the plane at infinity can lie in one of these two components. The degree 1 knot will lie on a two dimensional projective plane, say $X$. The intersection of $X$ with the plane at infinity is a line (copy of $\mathbb{RP}^{1}$), say $L$. Therefore, the special conic which lies on the plane at infinity intersects the plane X in two points, say $a_{1}$ and $a_{2}$. These two points divide $L$ in two components. Now the point of intersection of the knot with the plane at infinity can lie on either of these two components.

	Suppose $K_{1}$ and $K_{2}$ are two degree two knots in $Q_{3,2}$. By using a projective transformation in $\mathbb{RP}^{4}$ which preserves $Q_{3,2}$, these can be transformed so that they intersect at a point, say $p$. The projection of these two from $p$ are two degree $1$ knots in $\mathbb{RP}^{3}$ whose point of intersection with the plane at infinity lies in one of the components of $L \setminus \{a_{1},a_{2}\}$. Any two knots of degree 1 in $\mathbb{RP}^{3}$ are rigidly isotopic. Therefore, degree ~1 knots whose point of intersection lie in same component can be isotoped to each other. Knots whose points of intersection with the plane at infinity are in different components cannot be isotoped as during the isotopy between them, the knot must intersect the conic which is not allowed. 
	
	This will correspond to two different isotopy classes of degree $2$ real rational knots in $Q_{3,2}$. The projection of both classes in $\mathbb{RP}^{3}$ are smoothly isotopic to a circle having writhe $0$.

	Representatives of these two classes are parametrised as $[2st:0:4st:-3s^{2}+t^{2}:-3s^{2}-t^{2}]$ and $[2st:0:4s^{2}:t^{2}-3s^{2}:t^{2}-5s^{2}]$.
	
\end{proof}

However, when we consider the orientation then one of the classes splits into two.

\begin{lemma} \label{lemma 3}
	There are three rigid isotopy classes of degree $2$ real rational oriented knots in $Q_{3,2}$ up to orientation.
\end{lemma}

\begin{proof}
	
	Two classes of degree~2 knots (considered without orientation) are the pullback of the lines in $\mathbb{RP}^{3}$ whose points of intersection lies inside and outside the special conic respectively. If we try to isotope a line to a line with its reverse orientation, this can be done by rotating the line to itself. A line whose point of intersection with the plane at infinity lies inside the conic cannot be isotoped to its reverse because if we try to rotate the line in any direction it will either lie entirely on the plane at infinity or it will intersect the conic at the plane at infinity which will prevent it from being pulled back to a conic in $Q_{3,2}$ (it will then pull back to a line). Lines whose point of intersection with the plane at infinity lie outside the conic can be isotoped to its reverse orientation line via rotating the line around a point which lies on the intersection of line and a plane parallel to the plane at infinity (in an affine chart where the plane at infinity is visible in affine coordinates) as shown in figure \ref{fig 1}. During this rotation, at every instance the line intersects the plane at infinity outside the special conic which makes the knots isotopic.
	\begin{figure}[h]
		\centering
		\includegraphics[width=0.8\textwidth]{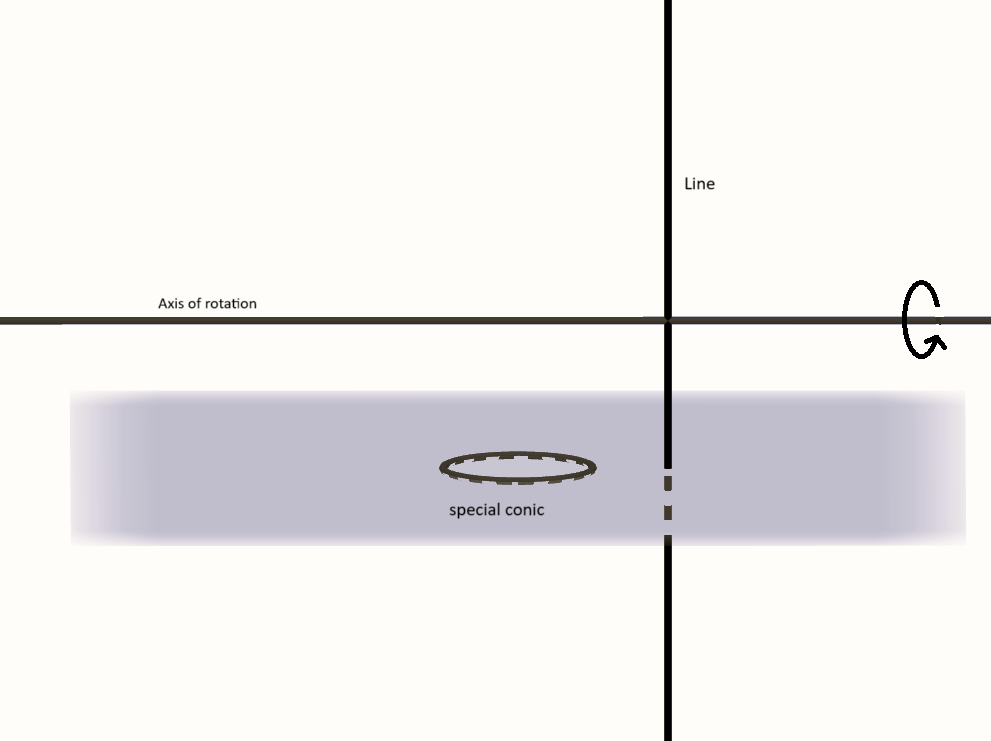}\hfill
		\caption{}
		\label{fig 1}
	\end{figure}
	\end{proof}
	\begin{proof}[Alternative proof]
	Now we check the isotopy of these two classes of knots via their projection in $\mathbb{RP}^{3}$ as degree $2$ knots. In $\mathbb{RP}^{3}$, there is always a plane which intersects the degree $2$ knot in imaginary points. This plane may or may not be parallel to the plane at infinity. The position of this plane with respect to the knot decides the isotopy class. If the plane which intersects the knot in only complex points is parallel to the plane at infinity (in an affine chart where the plane at infinity is visible in affine coordinates) then the knot can be isotoped to its reverse orientation via the rotation described in figure \ref{fig 2}. But if this plane is not parallel then any type of isotopy does not work and this gives rise to the two classes.\\\\
	$[s^{2}+t^{2}:0:s^{2}-t^{2}:2st]$ is a degree 2 knot in $\mathbb{RP}^{3}$ as in figure \ref{fig 2}. Take the plane at infinity to be $x_{3}=0$ and the conic at infinity to be the set in the plane satisfying $x_{0}^{2}=x_{1}^{2}+x_{2}^{2}$. From figure \ref{fig 2} (a), it is clear that plane which intersects the knot in imaginary points is parallel to the plane at infinity. This type of knot can be isotoped to its reverse orientation via the rotation in $\mathbb{RP}^{3}$ given by $[s^{2}+t^{2}:(s^{2}-t^{2})sin\theta:(s^{2}-t^{2})cos\theta:2st]$. The pull back of this isotopy in $\mathbb{RP}^{4}$ on $Q_{3,2}$ is given $[s^{2}+t^{2}:(s^{2}-t^{2})sin\theta:(s^{2}-t^{2})cos\theta:2st:0]$ which isotopes one class of degree $2$ knots having representative $[s^{2}+t^{2}:0:s^{2}-t^{2}:2st:0]$ to the knot with reverse orientation. 
	\begin{figure}[h]
		\centering
		\includegraphics[width=0.49\textwidth]{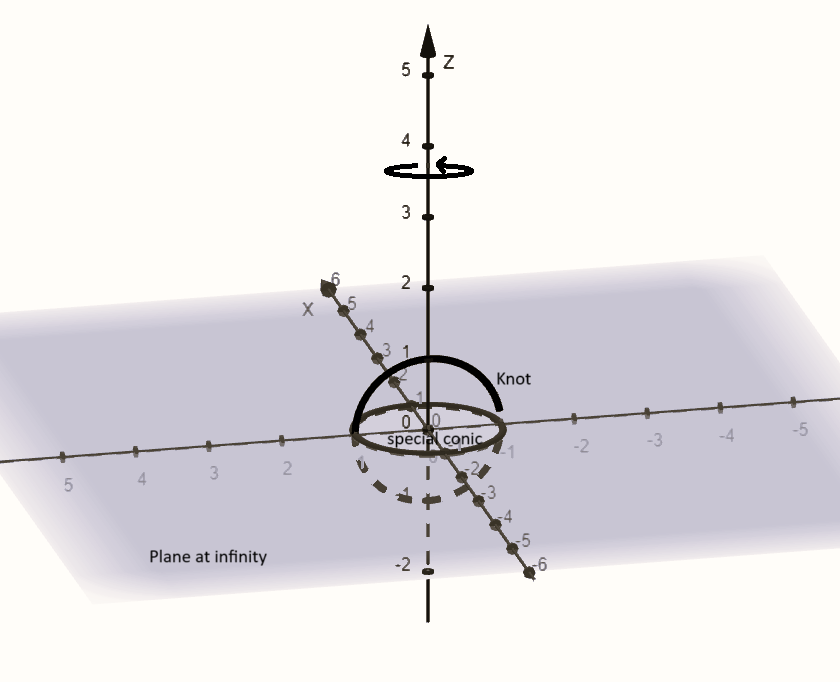}\hfill
		\includegraphics[width=0.49\textwidth]{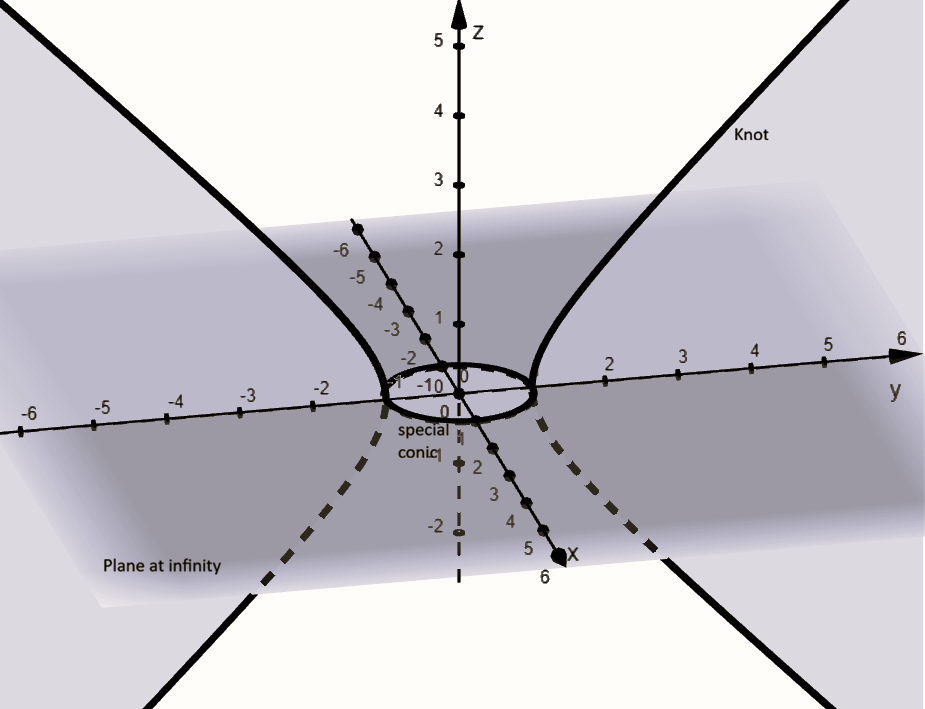}\hfill
		\caption{(a) degree $2$ knot which can be isotoped to its reverse orientation (b) degree $2$ knot which cannot to be isotoped to its reverse orientation}
		\label{fig 2}
	\end{figure}\\
	This type of isotopy does not work for the second class of degree $2$ knots (see figure \ref{fig 2} (b)).\\
\end{proof}

\begin{theorem}
	The space of real rational parametrisations of degree $2$ curves has three connected components.
\end{theorem}
\begin{proof}
	Trivially follows from lemma \ref{lemma 4} and lemma \ref{lemma 3}	
\end{proof}

\begin{remark}
	There can be non-rigid isotopy classes of knots in $Q_{3,2}$ which have same writhe number when projected in $\mathbb{RP}^{3}$.
\end{remark}

\section{Degree 3}

\begin{lemma} \label{lemma 5}
	There does not exist any degree~3 curve with double points lying on $Q_{3,2}$ in $\mathbb{RP}^{4}$.
\end{lemma}
\begin{proof}
	Suppose $C$ is real rational curve of degree~3 with at least one double point in $Q_{3,2}$. By lemma \ref{proj_curve_pt}, the projection of this curve, from a point on the curve but not from the double point to a copy of $\mathbb{RP}^{3}$, is a degree ~2 curve in $\mathbb{RP}^{3}$ with some double points. But such a curve can not be rational in $\mathbb{RP}^{3}$. 
\end{proof}

\begin{lemma} \label{lemma 7}
	There is one rigid isotopy class of degree $3$ real rational knots (without considering orientation) in $Q_{3,2}$.
\end{lemma}

\begin{proof}
	
	Suppose $K_{1}$ and $K_{2}$ are two real rational knots of degree $3$ in $\mathbb{RP}^{4}$ lying on $Q_{3,2}$. By using a projective transformation in $\mathbb{RP}^{4}$ which preserves $Q_{3,2}$ we can make $K_{1}$ and $K_{2}$ intersect at a point. When we project $K_{1}$ and $K_{2}$ from their common point we get degree $2$ knots in $\mathbb{RP}^{3}$ which intersects the plane at infinity in two points out of which one lies on the conic. Notice that these points of intersection have to be real because if they are imaginary then they must be in conjugate pairs so either both lie on the conic or neither does because the conic is real. Any degree $2$ knot is projectively equivalent to a circle so we assume that the projection is a circle and one point of intersection with the plane at infinity lies on the special conic and one does not. The degree $2$ knot is a planar curve. The intersection of the plane containing the curve with the plane at infinity is a line. The special conic lying at the plane at infinity divides the plane at infinity in two components and intersects the plane containing the curve at exactly two points. The point of intersection of the curve with the plane at infinity not lying on the special conic lies in either of these two components. If two circles are such that the point of intersection not on the conic is in the same component of the plane at infinity then they can be easily isotoped to each other such that the other point of intersection always lies on the special conic. So, such an isotopy can be lifted to $Q_{3,2}$. That is why we have to check only for the case where the points of intersection lie in different components. By using lemma \ref{rk1}, we can move the point which does not lie on the conic from one component to another component by making it tangent to line at infinity as shown in figure \ref{fig 3f}.
	
		\begin{figure}[h]
		\centering
		\includegraphics[width=0.49\textwidth]{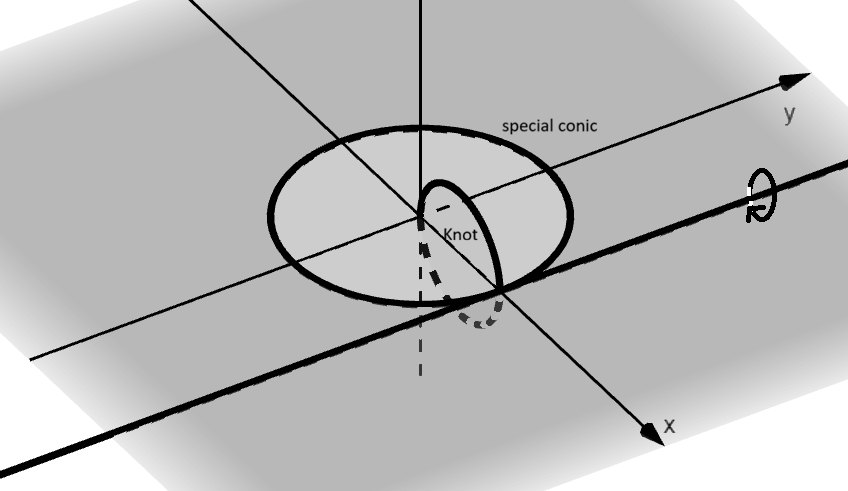}\hfill
		\caption{degree $2$ knot whose point not on the conic lies inside the conic is isotoped to degree $2$ knot whose point of intersection lies outside the conic via rotation}
		\label{fig 3f}
	\end{figure}
	
	Hence, in all there is only one rigid isotopy class of degree $3$ knots in $\mathbb{RP}^{4}$ which lies on $Q_{3,2}$. The projection of this class in $\mathbb{RP}^{3}$ is smoothly isotopic to a line. One representative of this class is given by $[t^{3}+s^{3}:t^{3}-s^{3}:0:-st^{2}-s^{2}t:-st^{2}+s^{2}t]$.
\end{proof}

\begin{lemma} \label{lemma 6}
	An oriented degree $3$ real rational knot in $\mathbb{RP}^{4}$ which lies on $Q_{3,2}$ can not be isotoped to its reverse orientation.
\end{lemma}
\begin{proof}
	When we project a degree $3$ knot in $Q_{3,2}$ from a point on the knot we get degree $2$ knot in $\mathbb{RP}^{3}$ such that exactly one point lies on the special conic or to a degree $2$ knot such that the plane at infinity is tangent to the knot at a point on the special conic but the knot is not tangent to the special conic. 
	
	The degree $2$ knot in $\mathbb{RP}^{3}$ has one point on the special conic and the other point not on the conic. Suppose there is an isotopy between the oriented degree $2$ knot whose one point is on the conic, say $q_{1}$, and the other is not on the conic, say $q_{2}$, and its reverse orientation, then by using a continuously varying family of projective transformation we can transform this isotopy in such a way that $q_{1}$ and $q_{2}$ are fixed throughout the isotopy. As we have fixed $q_{1}$ and $q_{2}$ throughout the isotopy and it isotopes to its reverse orientation then the transformed isotopy is forced to be the rotation about the axis passing through $q_{1}$ and $q_{2}$. This makes the axis of rotation is along the intersection of the plane at infinity and the plane containing the knot then during the rotation, for some t, the knot will completely lie on the plane at infinity as shown in figure \ref{fig 3"}. So, the degree $2$ knot cannot be isotoped to its reverse while maintaining right number of intersection with the special conic.\\

	\begin{figure}[h]
		\centering
		\includegraphics[width=0.49\textwidth]{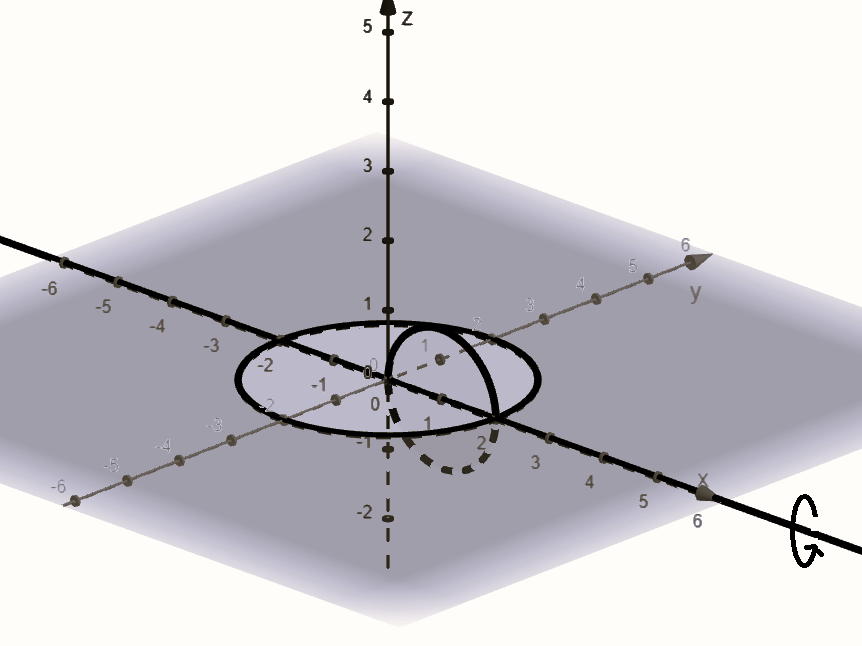}\hfill
		\caption{ axis of rotation is along intersection of plane at infinity $\&$ the plane containing knot }
		\label{fig 3"}
	\end{figure}
	
	If the plane at infinity is tangent to the degree $2$ curve at a point on the special conic but not tangent to the special conic, then as before, any isotopy that takes it to its orientation reversed, can be transformed to one which fixes the tangent vector of the knot at the point of intersection. As before, this will transform the isotopy to a rotation around this tangent vector as an axis as shown in figure \ref{fig 3}. In figure \ref{fig 3}, the knot is represented by $[s^{2}+t^{2}:2s^{2}:0:2st-s^{2}-t^{2}]$, the plane at infinity by $x_{3}=0$ and the special conic at infinity is $x_{0}^{2}=x_{1}^{2}+x_{2}^{2}$. This knot is tangent to the plane at infinity at the point $[1:1:0:0]$ but not tangent to the special conic. This can be rigidly isotoped to its reverse orientation via a rotation. Without loss of generality assume the isotopy to be  $[s^{2}+t^{2}:(s^{2}-t^{2}) \mathrm{cos} \theta+s^{2}+t^{2}:-(s^{2}-t^{2})\mathrm{sin} \theta:2st-s^{2}-t^{2}], 0\leq \theta \leq \pi$. For $\theta \neq \frac{\pi}{2}$, by using lemma \ref{tangent} pull back of curves in this isotopy are degree~3 knots in $Q_{3,2}$. For $\theta=\frac{\pi}{2}$, the knot $C_{\frac{\pi}{2}}$ is given by $[s^{2}+t^{2}:s^{2}+t^{2}:-(s^{2}-t^{2}):2st-s^{2}-t^{2}]$ is tangent to the plane at infinity $x_{3}=0$ and also tangent to the special conic $x_{0}^{2}=x_{1}^{2}+x_{2}^{2}$, $x_{3}=0$. The pull back of this curve in $Q_{3,2}$ is a degree~2 real rational knot parametrised by $[-2s^{2}-2t^{2}:-2s^{2}-2t^{2}:2s^{2}-2t^{2}:-4st:-2s^{2}-2t^{2}]$. So, the isotopy $C_{\theta}$ can not be pulled back to a rigid isotopy in $Q_{3,2}$. 
	\begin{figure}[h]
		\centering
		\includegraphics[width=0.8\textwidth]{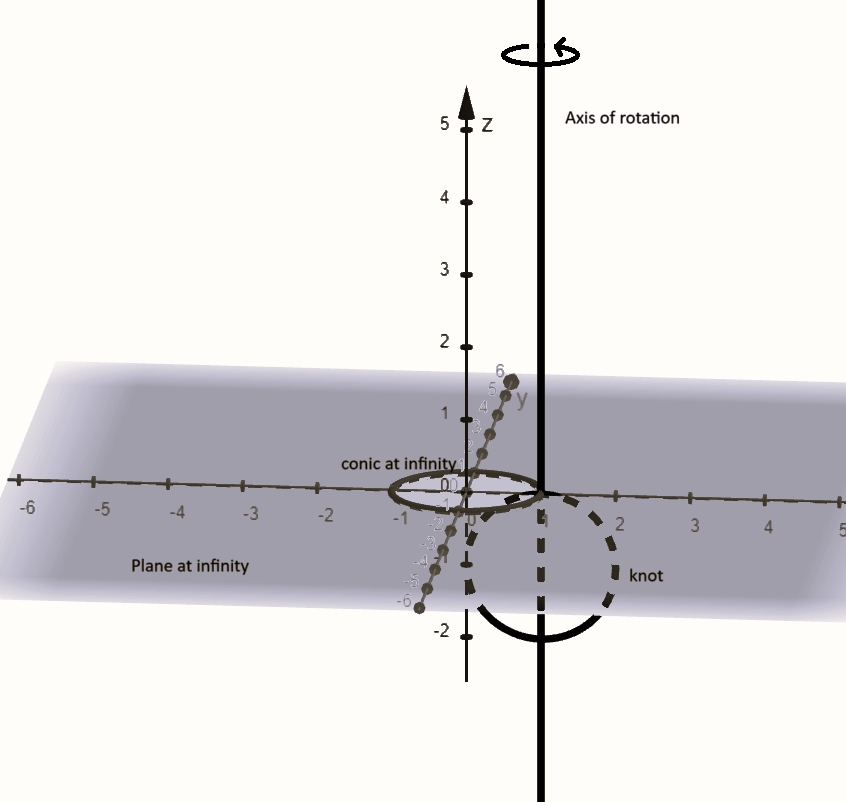}\hfill
		\caption{}
		\label{fig 3}
	\end{figure}
\end{proof}

\begin{lemma}
	There are $2$ rigid isotopy classes of degree $3$ real rational knots in $Q_{3,2}$.
\end{lemma}
\begin{proof}
	The proof follows from lemma \ref{lemma 7} and lemma \ref{lemma 6}.
\end{proof}

\begin{theorem}
	The space of real rational parametrisations of degree $3$ curves has two connected components.
\end{theorem}
\begin{proof}
	This directly follows from lemma \ref{lemma 5} and lemma \ref{lemma 6}
\end{proof}

Reflections in $Q_{3,2}$ are defined as the restrictions of reflections in $\mathbb{RP}^{4}$ to $Q_{3,2}$ which preserves $Q_{3,2}$.

\begin{remark}
	In $\mathbb{RP}^{3}$, a degree~3 knot is not isotopic to its mirror image. However, in $Q_{3,2}$, they are isotopic because the reflection can be achieved by rotating the ambient space $Q_{3,2}$ in $\mathbb{RP}^{4}$. The main idea which we use is that a reflection in a one dimension lower subspace can be carried out by using a rotation in the one dimension higher space $\mathbb{RP}^{4}$. 
	
	Any degree $3$ curve in $\mathbb{RP}^{4}$ lies on a linear hypersurface. $[t^{3}+s^{3}:t^{3}-s^{3}:0:-st^{2}-s^{2}t:-st^{2}+s^{2}t]$ say $K$ is a knot of degree $3$ in $\mathbb{RP}^{4}$ lying on the $3$-hyperboloid given by $x_{1}^{2}+x_{2}^{2}+x_{3}^{2}-x_{4}^{2}=x_{0}^{2}$. It lies on the hypersurface given by $x_{2}=0$. This knot K can be isotoped to its reflection $[t^{3}+s^{3}:-t^{3}+s^{3}:0:-st^{2}-s^{2}t:-st^{2}+s^{2}t]$ via using rotation $[t^{3}+s^{3}:(t^{3}-s^{3})\mathrm{cos} \theta:(t^{3}-s^{3})\mathrm{sin} \theta:-st^{2}-s^{2}t:-st^{2}+s^{2}t]$ which preserves $Q_{3,2}$. We project these from the point $[0:0:1:0:-1]$ to the hyperplane $x_{2}=0$. Notice that point of projection lies on the $Q_{3,2}$ but not on this path, for any $0 \leq \theta \leq \pi$. Projection of this path are degree 3 curves in $\mathbb{RP}^{3}$ but this does not give us an isotopy in $\mathbb{RP}^{3}$ as at $\theta= \pi/{2}$, the projected curve is a planar curve and planar degree 3 curve is singular. So by this, we can isotope reflections in the $Q_{3,2}$ in $\mathbb{RP}^{4}$.
	
	This happens for any isotopy in $Q_{3,2}$. 
\end{remark}

\begin{remark}
	It may be possible that pull back of two rigidly non-isotopic knots of the same degree in $\mathbb{RP}^{3}$ are rigidly isotopic in $Q_{3,2}$.
\end{remark}

\begin{theorem}{\label{reflection}}
	Any $d$ degree real rational curve lying on a hypersurface in $Q_{3,2}$ can be rigidly isotoped to its reflection by using a rotation.
\end{theorem}

\section{Degree 4}

By using lemma \ref{proj_curve_pt}, the projection of degree $4$ curve in $\mathbb{RP}^{3}$ is a degree $3$ curve whose two points of intersection with the plane at infinity lie on the special conic and one does not lie on the special conic. Points which lie on the special conic can be complex conjugates but the point not on the conic is a real point. If we consider the point not on the conic and two real points on the curve then these three points define a plane which intersects the curve in only real points. These two planes are isotopic via the rotation along the line which lies on both of the planes. Throughout, the isotopy we can choose the special conic to contain the two points of intersection. So, now onward we suppose that the plane at infinity intersects degree $3$ curve in only real points of intersection. 

\begin{lemma} The space of real rational parametrisations of degree 4 curves in $\mathbb{RP}^{4}$ which lie on $Q_{3,2}$ is connected.
\end{lemma}

\begin{proof}

	Suppose $C_{0}$ and $C_{1}$ are two degree $4$ real rational curves in $Q_{3,2}$. By using a projective transformation which preserves $Q_{3,2}$, we can make these two curves intersect at a point say $p$. By using lemma \ref{proj_curve_pt}, when we project these two curves from the point $p$, the projected curves $C_{0}^{'}$ and $C_{1}^{'}$ are two degree $3$ curves in $\mathbb{RP}^{3}$ whose two points of intersection with the plane at infinity lie on the special conic and the third one is not on the special conic. Denote the points of intersection of these curves with the plane at infinity by $a_{t},b_{t}$ and $c_{t}$ for $t=0,1$. Suppose $a_{0},b_{0},a_{1},b_{1}$ lie on the special conic and $c_{0}$, $c_{1}$ are points not on the special conic. Without loss of generality, we can assume that both $c_{0}$ and $c_{1}$ lie in the same component of the plane at infinity minus the special conic (if not we can switch the position of the point by using lemma \ref{rk1}).
	
	The space of real rational curves of degree $3$ in $\mathbb{RP}^{3}$ is a connected space. Therefore, there exists a path $C_{t}^{'}$ between $C_{0}^{'}$ and $C_{1}^{'}$  in the space of degree $3$ curves in $\mathbb{RP}^{3}$. Each curve $C_{t}^{'}$ intersects the plane at infinity at three points say $a_{t}^{'},b_{t}^{'}$ and $c_{t}^{'}$.
	
	As the special conic is connected, we can always find a path $a_{t}$ between points $a_{0}$ and $a_{1}$ and a path $b_{t}$ between points $b_{0}$ and $b_{1}$ such that $a_{t}$ and $b_{t}$ lie completely on the special conic. $c_{0}$ and $c_{1}$ lie in the same component of the plane at infinity minus the special conic so we can find a path between $c_{0}$ and $c_{1}$ which lies entirely in that component. Take a point $d$ outside the plane at infinity. Define a projective transformation $T_{t}$ from $\mathbb{RP}^{3}$ to $\mathbb{RP}^{3}$ which maps $a_{t}^{'}$ to $a_{t}$, $b_{t}^{'}$ to $b_{t}$, $c_{t}^{'}$ to $c_{t}$ and fixes the point $d$. $T_{t}(C_{t}^{'})$ gives us a path between $C_{0}^{'}$ and $C_{1}^{'}$ which can be pulled back. The pull back of this path is a path which connects $C_{0}$ and $C_{1}$ in the space of degree $4$ curves in $Q_{3,2}$.
\end{proof}

\begin{proof}[Alternative proof]
	Suppose $C_{0}$ and $C_{1}$ are two real rational curves of degree $4$ in $Q_{3,2}$. By using lemma \ref{proj_not_on_curve}, when we project these curves we get two pairs $(C_{0}^{'},X_{0})$ and $(C_{1}^{'},X_{1})$ such that $C_{0}^{'}$ and $C_{1}^{'}$ are degree $4$ curves in $\mathbb{RP}^{3}$ which intersects planes $X_{0}$ and $X_{1}$ respectively at four points and all these points lie on the special conic. 
	
	Imaginary points of intersection always exists in conjugate pairs. If all the points of intersection with the plane at infinity are imaginary, consider the real line $L$ passing through a conjugate pair and isotope this plane at infinity to a plane which lie in pencil of planes through line $L$ and intersects the curve in two real points. This plane intersects the curve in two real and two imaginary points. Now consider the pencil of planes through real line passing through the two real intersection points. There must be a plane in this pencil which intersects the curve in only real points. Any four points determines a conic. We can avoid the situation that any three points out of the four intersection points are collinear so, we can always choose a conic of signature $(2,1)$ throughout the isotopy containing the all the four points of intersection. This way we can move the plane at infinity such that all the intersection points of plane with the curve are real points.
	
	So, we can assume that $C_{0}^{'}$ and $C_{1}^{'}$ intersects $X_{0}$ and $X_{1}$ planes respectively in real points. The space of real rational curves of degree $4$ in $\mathbb{RP}^{3}$, denoted by $Y$ is connected. So, we can connect $C_{0}^{'}$ and $C_{1}^{'}$ via a path in $Y$. $C_{0}^{'}$ intersects the plane $X_{0}$ in four points say, $a^{'},b^{'},c^{'},d^{'}$ and all the points lie on the special conic. During the homotopy between $C_{0}^{'}$ and $C_{1}^{'}$, suppose the intersection points $a^{'},b^{'},c^{'},d^{'}$ continously deformed to points $a_{0},b_{0},c_{0},d_{0}$. Consider the unique plane determined by $a_{0},b_{0}$ and $c_{0}$ and denote it by $X_{0}^{'}$. This gives us a path between $(C_{0}^{'},X_{0})$ and $(C_{1}^{'},X_{0}^{'})$. $C_{1}^{'}$ intersects the planes $X_{0}^{'}$ and $X_{1}$ in four points, denoted by $a_{0},b_{0},c_{0},d_{0}$ and $a_{1},b_{1},c_{1},d_{1}$ respectively and all these points lie on the special conic. Denote the preimages of $b_{0},c_{0},d_{0}$ under the parametrisation of knot $C_{1}^{'}$ by $p_{0}^{1}, p_{0}^{2},p_{0}^{3}$ and the preimages of $b_{1},c_{1},d_{1}$ under the parametrisation of knot $C_{1}^{'}$ by $p_{1}^{1}, p_{1}^{2},p_{1}^{3}$. We can always find a path $p_{t}$ in $\mathbb{RP}^{1}$ which takes $\{p_{0}^{1}, p_{0}^{2},p_{0}^{3}\}$ to $\{p_{1}^{1}, p_{1}^{2},p_{1}^{3}\}$. Any plane in $\mathbb{RP}^{3}$ is determined by three points. Consider the plane determined by the images of $\{p_{t}^{1}, p_{t}^{2},p_{t}^{3}\}$ under parametrisation of $C_{1}^{'}$ and denote it by $X_{t}$. This $X_{t}; 0 \leq t \leq 1$ gives an isotopy between the planes $X_{0}^{'}$ and $X_{1}$. $X_{t}$ intersects the knot $C_{1}^{'}$ in exactly four points say $a_{t},b_{t},c_{t}$ and $d_{t}$ for $0 \leq t \leq 1$. There is always an isotopy between $X_{0}^{'}$ and $X_{1}$ such that all these points are in generic position for any $0 \leq t \leq 1$. Four points determine a pencil of conics. As the points are in generic position, we can always find a conic of signature $(2,1)$ passing through points $a_{t},b_{t},c_{t}$ and $d_{t}$. This way, $(C_{1}^{'},X_{t})$ gives a isotopy between $(C_{1}^{'},X_{0}^{'})$ and $(C_{1}^{'},X_{1})$.  
	
\end{proof}

The next lemma will show that the space of of degree~4 real rational knots does not have edges.

\begin{lemma}
	There does not exist any degree ~4 real rational curve in $Q_{3,2}$ with two or more double points.
\end{lemma}

\begin{proof}
	Suppose $C$ is a real rational curve of degree~4 in $Q_{3,2}$ with at least two double points. By using lemma \ref{proj_double_pt}, project this curve from one of its double points to a hyperplane in $\mathbb{RP}^{4}$. The projected curve is a degree ~2 curve with at least one double point in $\mathbb{RP}^{3}$ such that both the points of intersection with plane at infinity do not lie on the special conic. But, there does not exist any real rational curve of degree ~2 with double points in $\mathbb{RP}^{3}$. Hence, $Q_{3,2}$ does not have any real rational curve of degree $4$ with at least two double points.
\end{proof}

Now, we will classify degree 4 real rational curves with exactly one double point in $Q_{3,2}$. We can easily see that degree 4 curves with one double point in $Q_{3,2}$ lie on a hyperplane in $\mathbb{RP}^{4}$.

\begin{theorem} \label{wall}
	There are two rigid isotopy classes of degree $4$ curves with exactly one double point (without considering orientation) in $Q_{3,2}$.
\end{theorem}

\begin{proof}
	
	By using lemma \ref{proj_curve_pt}, when we project a curve of degree $4$ with one double point from a point which lies on the curve but is not the double point, we get a degree $3$ singular curve in $\mathbb{RP}^{3}$ which intersects the plane at infinity in three points out of which two points lie on the special conic. Degree $3$ singular curves are planar curves. This plane (which contains the curve) intersects the plane at infinity in a line (copy of $\mathbb{RP}^{1}$, say line at infinity) such that intersection points of the curve with plane at infinity lie on this line at infinity. There is only one class of degree $3$ singular curves in $\mathbb{RP}^{3}$. The special conic intersects the plane (which contains the curve) in two points. Based upon the position of intersection points there are three possibilities as shown in figure \ref{fig 3a}.
	
	\begin{figure}[h]
		\centering
		\includegraphics[width=1\textwidth]{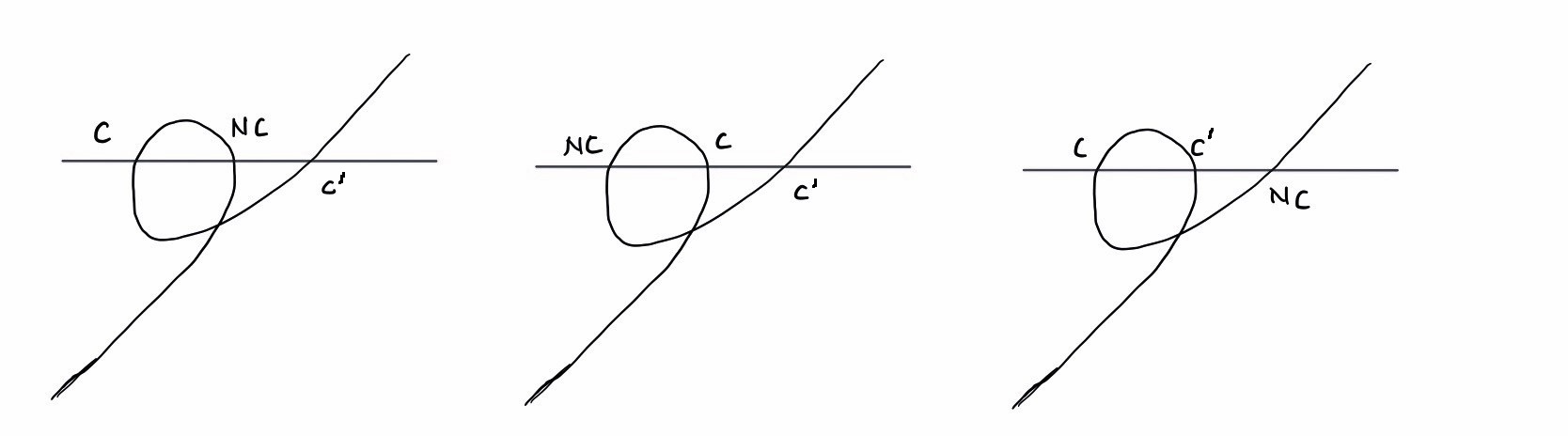}\hfill
		\caption{points $C$, $C^{'}$ and $NC$ are the intersection points of curve and the plane at infinity where $C$, $C^{'}$ lie on the special conic and $NC$ is not on the conic}
		\label{fig 3a}
	\end{figure}
	
	The first and the second possibility in figure \ref{fig 3a} can be rigidly isotoped to each other by using lemma \ref{rk1}. For this, we have to interchange the position of the point not on the conic. These can  be isotoped by first making the plane at infinity tangent to the curve by merging point $C$ and $NC$ and after that taking the point not on the conic outside the conic. So this forms a class for degree $4$ curves with exactly one double point.
	
	This class can not be isotoped to the third possibility as in the third case we can isotope the plane at infinity to a plane which intersects the curve in two conjugate and one real point such that two imaginary conjugate points lie on the special conic. This cannot be done in the first class. So, this gives us the second class of degree $4$ curves with one double point.

	These two classes can be parametrised as:
	
	$[0:-2s^{3}t+2st^{3}:-2s^{2}t^{2}+2st^{3}:-s^{4}-2t^{4}+2s^{2}t^{2}+2st^{3}:s^{4}+2t^{4}-2st^{3}]$
	
	$[2s^{3}t-2st^{3}:-2\sqrt{2}s^{2}t^{2}+2\sqrt{2}st^{3}:0:s^{4}-3s^{2}t^{2}+4st^{3}-t^{4}:-s^{4}+5s^{2}t^{2}-4st^{3}+t^{4}]$.
\end{proof}

\begin{proof}[Alternative Proof]\label{4.4.1}

	By using Lemma \ref{proj_double_pt}, when we project a degree $4$ curve with a double point from the double point of the curve we get a degree 2 real rational knot in $\mathbb{RP}^{3}$ intersecting the plane at infinity in two points, such that both of them do not lie on the special conic. If the  double point is non-solitary, then the  intersection points are real; if the double point is solitary, then the intersection points are the images of non-real conjugate pair of points in $\mathbb{RP}^1$.
	
	The special conic divides the plane at infinity into two components: one which is a disc and the another which is homeomorphic to a mobius band. The former will be called the inside component and the latter will be called the outside component.  So we have three possibilities either both points of intersection  lie in the interior of the conic, both points lie   in the exterior of the conic, or one of them is  in the interior  and the another one is  in the exterior (See Figure \ref{fig1}).
	
	Degree $2$ knots are planar. So the degree~2 knot that we get after projecting the degree~4 knot from the double point will intersect the plane at infinity in points that lie on a line.
	
	\begin{figure}[h]
		\centering
		\includegraphics[width=0.4\textwidth]{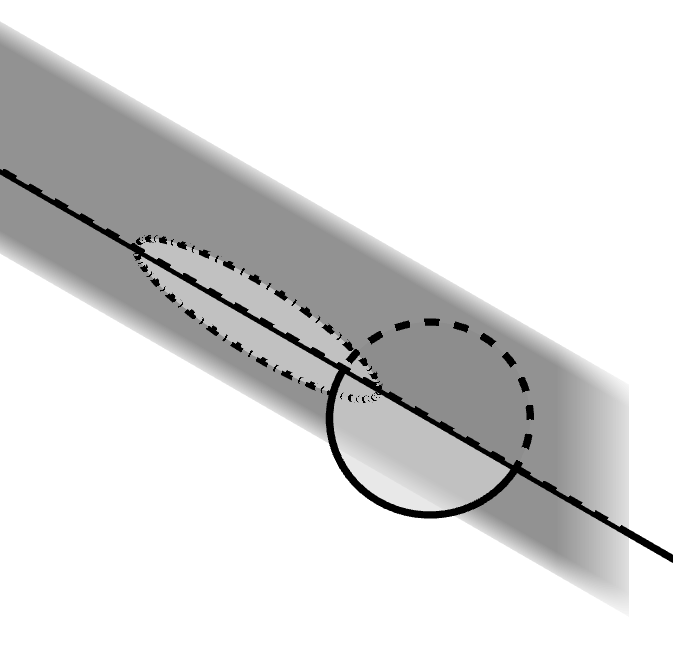}\hfill
		\includegraphics[width=0.4\textwidth]{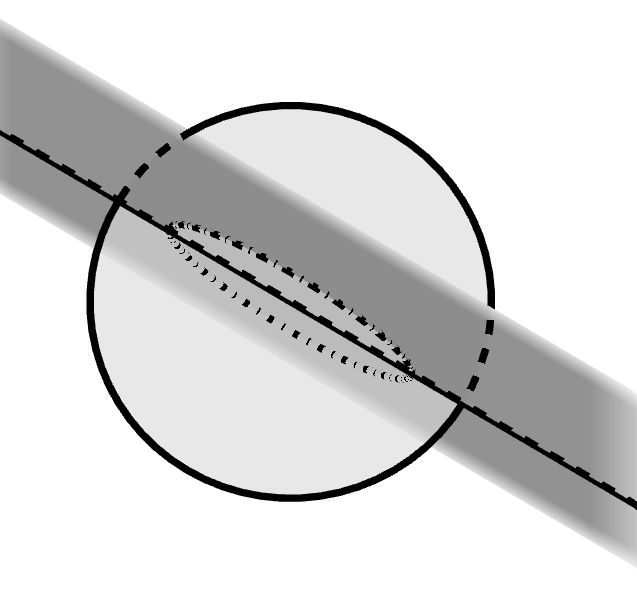}\hfill
		\includegraphics[width=0.4\textwidth]{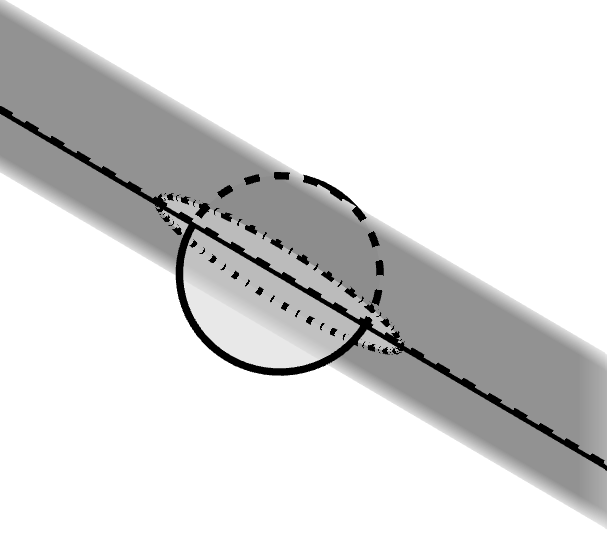}\hfill
		\caption{}
		\label{fig1}
	\end{figure}
	
	Suppose $C_{1}$ and $C_{2}$ are two curves of degree $4$ with exactly one double point in $Q_{3,2}$. By using a projective transformation which preserves $Q_{3,2}$, we can make sure that both curves have same double point. Project these from the double point. Projections of the curves in $\mathbb{RP}^{3}$ with same position of intersection points can be rigidly isotoped to each other.
	
	Degree 2 knots with both points of intersection inside the conic and degree 2 knots with both points of intersection outside the conic can be isotoped to each other by making it first tangent to the plane at infinity at a point inside the conic (By lemma \ref{proj_double_pt} in $Q_{3,2}$ its pull back is a curve with cusp singularity). Then move the degree 2 knot  to make points of intersection imaginary (in $Q_{3,2}$, it will be a curve with a solitary double point). Then make it tangent to the plane at infinity on a point outside the conic (in $Q_{3,2}$ it will be a curve with cusp singularity). Finally, make the intersection points distinct as shown in figure \ref{fig2} (in $Q_{3,2}$ it will be a curve with a solitary double point).

	\begin{figure}[h]
		\centering
		\includegraphics[width=0.9\textwidth]{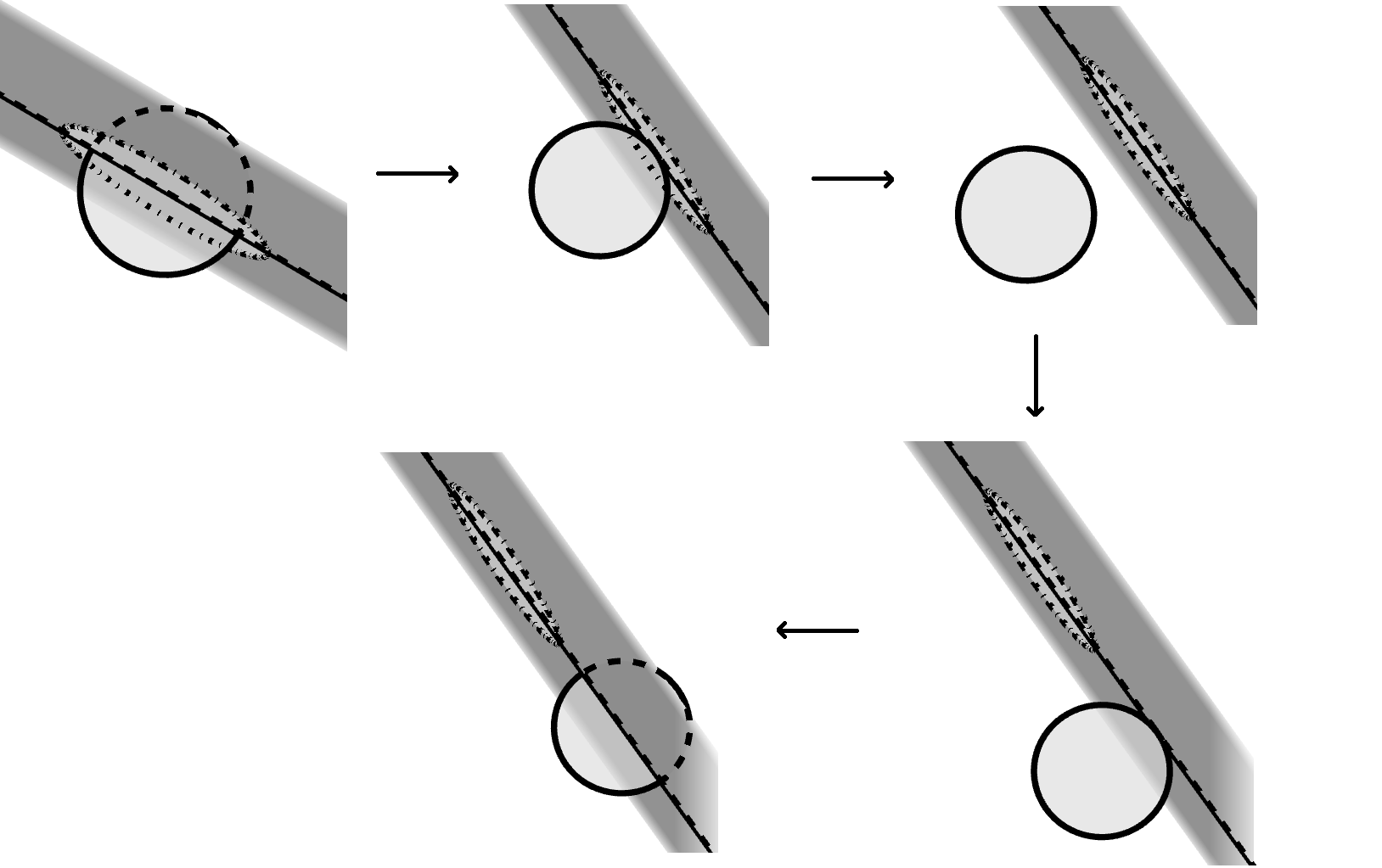}\hfill
		\caption{}
		\label{fig2}
	\end{figure}

	As figure \ref{fig1} clearly depicts that knot with point of intersection in different components can not be unlinked from the special conic. In an affine chart, they have linking number $\pm{1}$ and therefore can not be isotoped to the unlink which is the case when the knot intersects the plane at infinity in points that lie in the same component.
	
	Therefore, there are exactly two rigid isotopy classes of degree 4 real rational curves with exactly one double point (without orientation) in $Q_{3,2}$.
\end{proof}

	

\begin{lemma}\label{wall1}
	\begin{enumerate}
		\item The degree~4 oriented wall in $Q_{3,2}$ which corresponds to a degree~2 knot in $\mathbb{RP}^{3}$ whose both points of intersection with plane at infinity lies in the same component can be isotoped to its reverse orientation wall. 
		\item The degree~4 oriented wall in $Q_{3,2}$ which corresponds to a degree~2 knot in $\mathbb{RP}^{3}$ whose points of intersection with plane at infinity lies in the different components cannot be isotoped to its reverse orientation wall. 
	\end{enumerate}
\end{lemma}
\begin{proof}
	\begin{enumerate}
		\item Consider a degree $2$ knot in $\mathbb{RP}^{3}$ whose both points of intersection lie in the same component of the plane at infinity minus the special conic. By using some projective transformation, this can be transformed to $[2s^{2}+2t^{2}:0:2st:s^{2}-t^{2}:0]$ and plane at infinity to be $x_{3}=0$. Then the special conic is given by $x_{0}^{2}=x_{1}^{2}+x_{2}^{2}$. This can be isotoped to its reverse orientation via isotopy given by $[2s^{2}+2t^{2}:2st sin{\theta}:2st cos{\theta}:s^{2}-t^{2}:0]$ as shown in figure \ref{fig 4}.\\
		By using lemma \ref{proj_curve_pt}, if we project a degree $4$ curve with one double point from a point on the curve then it get projected to degree $3$ curve with one double point in $\mathbb{RP}^{3}$. We now try to understand the above isotopy in the degree $3$ setup.\\
		The pull back of the above isotopy in $\mathbb{RP}^{4}$ is given by $[4s^{4}-4t^{4}:4st(s^{2}-t^{2})sin{\theta}:4st(s^{2}-t^{2})cos{\theta}:5s^{4}+5t^{4}+2s^{2}t^{2}:3s^{4}+3t^{4}+6s^{2}t^{2}]$. Then we project this from the point $[-4:0:0:5:3]$ (image of $[0:1]$), lying on the curves throughout the isotopy on the hyperplane $x_{0}=0$. This projection is given by $[0:x_{1}:x_{2}:\frac{5}{4}x_{0}+x_{3}:\frac{3}{4}x_{0}+x_{3}]$. So the projection of the isotopy is given by $[0:4t(s^{2}-t^{2})sin{\theta}:4t(s^{2}-t^{2})cos{\theta}:10s^{3}+2st^{2}:6s^{3}+6st^{2}]$. This gives us an isotopy between the degree $3$ curve with one double point and its orientation reverse curve as shown in figure \ref{fig 5}.
		\begin{figure}[h]
			\centering
			\includegraphics[width=0.8\textwidth]{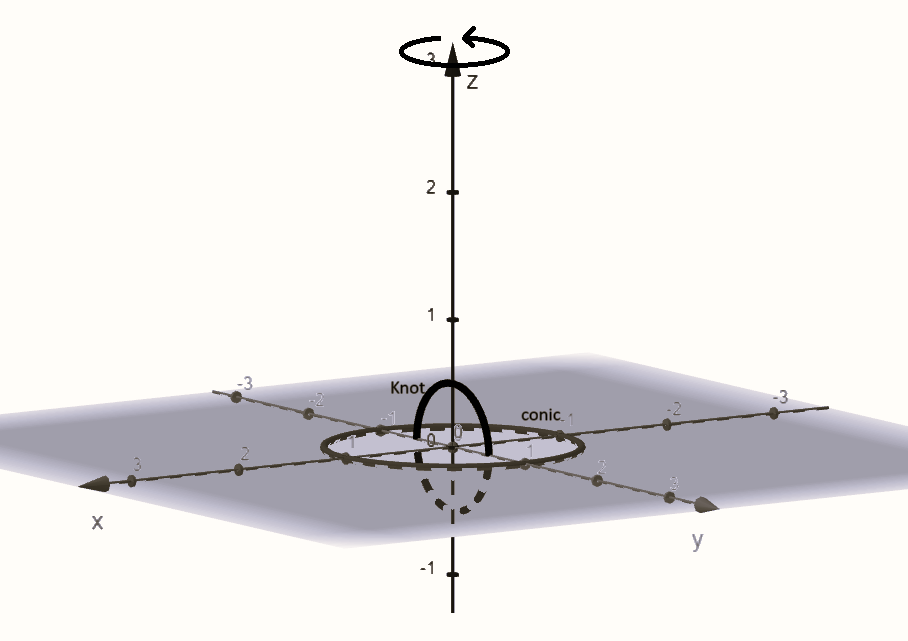}\hfill
			\caption{}
			\label{fig 4}
		\end{figure}
		\begin{figure}[h]
			\centering
			\includegraphics[width=0.6\textwidth]{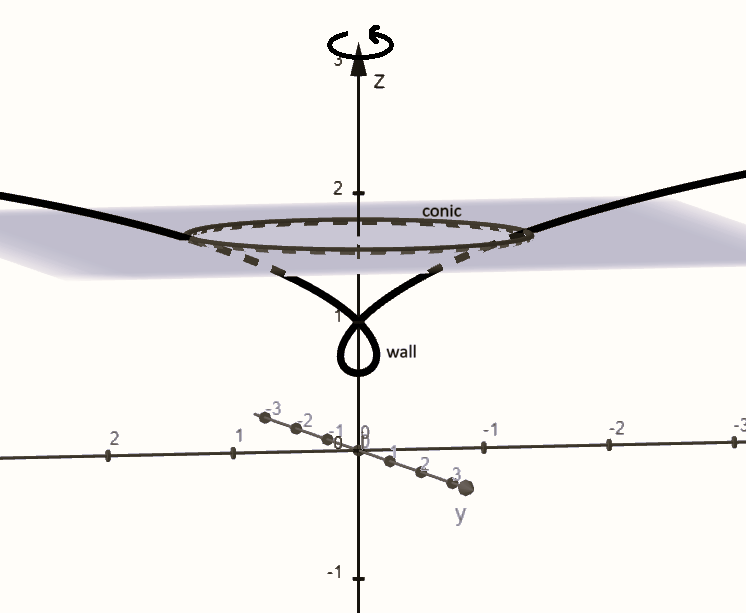}\hfill
			\caption{}
			\label{fig 5}
		\end{figure}
		\item Consider the degree ~4 wall in $Q_{3,2}$ whose projection from double point is a ~2 degree knot in $\mathbb{RP}^{3}$ such that points of intersection say $q_{1}$ and $q_{2}$ with plane at infinity are in different components.  Suppose there is an isotopy between such a oriented degree $2$ knot and the its reverse orientation then we can transform this isotopy in such a way that $q_{1}$ and $q_{2}$ are fixed throughout the isotopy. As we have fixed $q_{1}$ and $q_{2}$ throughout the isotopy and it isotopes to its reverse orientation then the transformed isotopy is forced to be a rotation about the axis passing through $q_{1}$ and $q_{2}$. The axis of rotation is along the intersection of the plane at infinity and the plane containing the knot. Then during the rotation, for some t, the knot is forced to completely lie on the plane at infinity as shown in figure \ref{fig 3"'} and cannot be lifted to $Q_{3,2}$.
		
		So, this type of degree~2 knot can not isotoped to its reverse orientation.
		\begin{figure}[h]
			\centering
			\includegraphics[width=0.49\textwidth]{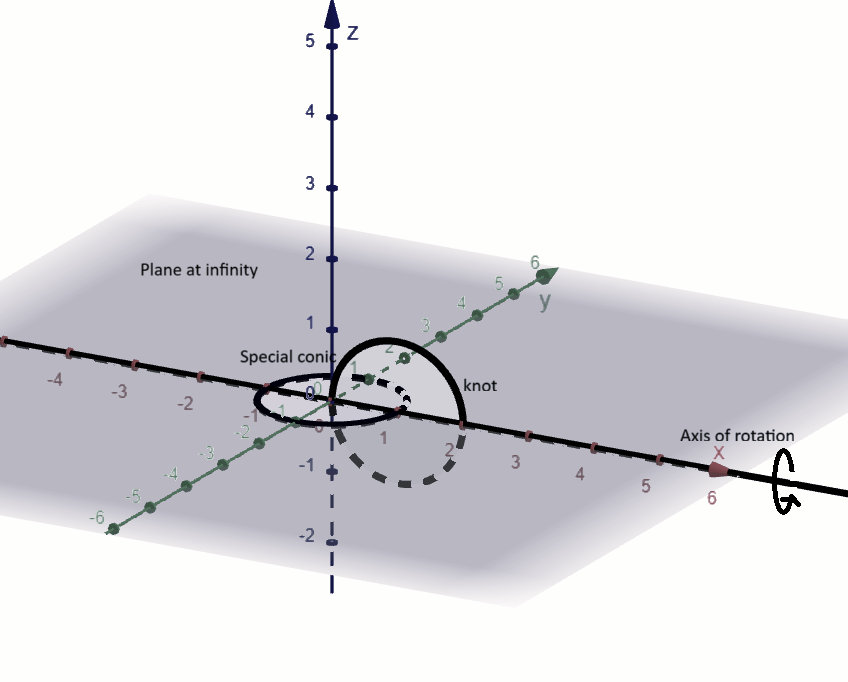}\hfill
			
			\caption{ axis of rotation is along intersection of plane at infinity $\&$ the plane containing knot }
			\label{fig 3"'}
		\end{figure} 
	\end{enumerate}
\end{proof}

\begin{remark}
	The pull back of the degree $2$ knot both of whose points are inside the conic is isotopic to the pull back of degree $2$ knot both of whose points are outside the conic. Lemma \ref{reflection} can also be used to isotope these as when we project these pull backs from a point which does not lie on the curve they get projected to curves which are reflections of each other.  
\end{remark}

\begin{theorem}
	There are three rigid isotopy classes of degree $4$ curves with exactly one double point in $Q_{3,2}$.
\end{theorem}
\begin{proof}
	The proof follows from theorem \ref{wall} and lemma \ref{wall1}.
\end{proof}

We now use the classification of degree $4$ real rational curves with exactly one double point in $Q_{3,2}$ to find an upper bound for the rigid isotopy classes of degree $4$ real rational knots lying in $Q_{3,2}$.

\begin{lemma}
	The space of real parametrisations of degree $4$ curves in $\mathbb{RP}^{4}$ that lies on $Q_{3,2}$ is a manifold of dimension 15.
\end{lemma}
\begin{proof}
	By using theorem \ref{dimension} on page \pageref{dimension}.
\end{proof}

Suppose $X_{4}$ denote the space of parametrisations of degree~4 curves with at least one double point lying on $Q_{3,2}$ in $\mathbb{RP}^{4}$, $X^{'}_{4}$ denote the space of parametrisations of curves of degree ~4 with more than one double point lying on $Q_{3,2}$ in $\mathbb{RP}^{4}$.

\begin{lemma}
	$X_{4} \setminus X_{4}^{'}$ is a manifold of dimension 14.
\end{lemma}
\begin{proof}
	By using theorem \ref{dim} on page \pageref{dim}.
\end{proof}

\begin{theorem}
	There are exactly $2$ rigid isotopy classes of degree $4$ real rational knots in $Q_{3,2}$.
\end{theorem}

\begin{proof}
	Using lemma \ref{proj_curve_pt}, when we project a knot of degree 4 from a point which lies on the knot, we get a degree 3 knot in $\mathbb{RP}^{3}$ such that out of three points of intersection, two lie on the special conic (a third point may lie outside or inside the conic actually it does not matter as this is degree 3 we can always switch the position of the point by using lemma \ref{rk1}). 
	
	Suppose $K_{i},i=1,2$ are two degree $4$ knots in $Q_{3,2}$. By using a projective transformation in $\mathbb{RP}^{4}$ which preserves $Q_{3,2}$ we make these two intersect at a point then project these from their common point to $\mathbb{RP}^{3}$. Assume that the projection of these two are $3$ degree knots rigidly isotopic to each other and have points of intersection with the plane at infinity as $a_{i},b_{i}$ and $c_{i}$. Suppose $a_{i}$'s and $b_{i}$'s lie on the special conic and $c_{1},c_{2}$ are points not on the conic. Without loss of generality we assume that both $c_{1}$ and $c_{2}$ either lies in the same component of plane at infinity minus the special conic (If not we can switch the position of the point). In $\mathbb{RP}^{3}$, $K_{1}$ and $K_{2}$ are isotopic so there is a path $K_{t}$ between $K_{1}$ and $K_{2}$. The isotopy $K_{t}$ gives us a path $a_{t}$ between $a_{1}$ and $a_{2}$, $b_{t}$ between $b_{1}$ and $b_{2}$, $c_{t}$ between $c_{1}$ and $c_{2}$.
	
	As the conic is connected, we can always connect $a_{1}$ and $a_{2}$ via a path say $a_{t}^{'}$, $b_{1}$ and $b_{2}$ via path say $b_{t}^{'}$ lying on the conic. $c_{1}$ and $c_{2}$ both lies in the same component so we can find a path $c_{t}^{'}$ between $c_{1}$ and $c_{2}$ which lies entirely in that component. Take a point $x$ outside the plane at infinity. Now for each $t$, we define a projective transformation $T_{t}$ such that it fixes x, maps $a_{t} \rightarrow a_{t}^{'}$, $b_{t} \rightarrow b_{t}^{'}$ and $c_{t} \rightarrow c_{t}^{'}$. $T_{t}(K_{t})$ gives us a path between $K_{1}$ and $K_{2}$ which can be pulled back because now two points of intersection lie on the special conic throughout the isotopy. 
	
	In $\mathbb{RP}^{3}$, we have two degree~3 knots one is having $+1$ writhe and another one is with $-1$ writhe. So, there are only $2$ rigid isotopy classes of degree $4$ real rational knots. The projections of these classes in $\mathbb{RP}^{3}$ are writhe $1$ and writhe $3$ knots respectively which are smoothly isotopic to a circle and two crossing knot.
	
	The representatives of these two classes of degree 4 knots in $Q_{3,2}$ are given by
	$[-s^{2}t^{2}+t^{4}+2s^{4}-2s^{3}t+6st^{3}:-2s^{4}+2s^{3}t+6s^{2}t^{2}-4st^{3}-2t^{4}:2s^{2}t^{2}+2s^{4}-4s^{3}t+6st^{3}:-2s^{3}t-s^{2}t^{2}+2st^{3}+t^{4}:2s^{4}-4s^{3}t-2s^{2}t^{2}+2st^{3}+2t^{4}]$,\\
	$[-3s^{2}t^{2}+2s^{4}-t^{4}-2s^{3}t+6st^{3}:-2s^{4}+4s^{3}t+2s^{2}t^{2}-6st^{3}+2t^{4}:2s^{4}+2s^{2}t^{2}-4s^{3}t+2st^{3}:s^{2}t^{2}+2st^{3}-t^{4}-2s^{3}t:-6s^{3}t+2s^{2}t^{2}-2t^{4}+4st^{3}+2s^{4}]$.
\end{proof}

\section{Degree 5} 
To investigate the space of real rational knots of degree~5 in $Q_{3,2}$ we first prove that in degree~5, the space of real rational curves (even with singularities) is connected.

\begin{theorem}
	The space of all real rational parametrisations of degree~5 curves in $Q_{3,2}$ is a connected space. 
\end{theorem}
\begin{proof}
	Suppose $C_{1}$ and $C_{2}$ are two real rational parametrisations of degree $5$ curves in $Q_{3,2}$. By using lemma \ref{proj_curve_pt}, when we project these curves from the point on the curve to a copy of $\mathbb{RP}^{3}$, we get two pairs $(C_{1}^{'},X)$ and $(C_{2}^{'},X^{'})$ such that $C_{1}^{'}$ and $C_{2}^{'}$ are degree $4$ curves in $\mathbb{RP}^{3}$ which intersects the planes $X$ and $X^{'}$ respectively at four points out of which three points of intersection lie on the special conic and one point does not lie on the special conic. 
	
	The two points of intersection which lie on the special conic can be complex conjugates, but the other two points of intersection cannot be complex otherwise being conjugate pairs, they would either have to both lie on the conic or both not lie on it. Take the line passing through these pair of real points of intersection and consider the pencil of planes containing that line. Choose a real point on the curve and the (unique) plane from the pencil containing that point. Such a plane must intersect the curve in at least 3 real points, but then the 4th point of intersection must also be real since any non real points must appear in conjugate pairs. Moving the plane along this pencil and we get a path between the plane whose intersection points are imaginary to a plane whose points of intersection with curve are all real. Any four points determines a conic. We can avoid the situation that the three points of intersection which lie on the conic are collinear so, we can always choose a conic of signature $(2,1)$ throughout the isotopy containing the three points of intersection and one does not lies on the conic.
	
	The space of real rational parametrisations of degree~4 in $\mathbb{RP}^{3}$ is a connected space. So, there is a homotopy between the curves $C_{1}^{'}$ and $C_{2}^{'}$ in $\mathbb{RP}^{3}$. $C_{1}^{'}$ intersects the plane $X$ in four points say, $a^{'},b^{'},c^{'},d^{'}$ out of which three points $a^{'},b^{'},c^{'}$ lie on the special conic and $d^{'}$ does not lies on the conic. During the homotopy between $C_{1}^{'}$ and $C_{2}^{'}$, suppose the intersection points $a^{'},b^{'},c^{'},d^{'}$ continously deformed to points $a_{0},b_{0},c_{0},d_{0}$. Consider the unique plane determined by $a_{0},b_{0}$ and $c_{0}$ and denote it by $X^{''}$. This gives a homotopy between the pairs $(C_{1}^{'},X)$ and $(C_{2}^{'},X^{''})$. Note that $X^{''}$ intersects the curve $C_{2}^{'}$ in $4$ points out of which exactly three points lie on the special conic and all points of intersection are real. $C_{2}^{'}$ intersects each of the planes $X^{''}$ and $X^{'}$ in $4$ points denoted by $a_{0},b_{0},c_{0},d_{0}$ and $a_{1},b_{1},c_{1}, d_{1}$ respectively where $a_{0}$ and $a_{1}$ are the points which do not lie on the special conic. Denote the preimages of $b_{0},c_{0},d_{0}$ under the parametrisation of knot $C_{2}^{'}$ by $p_{0}^{1}, p_{0}^{2},p_{0}^{3}$ and the preimages of $b_{1},c_{1},d_{1}$ under the parametrisation of knot $C_{2}^{'}$ by $p_{1}^{1}, p_{1}^{2},p_{1}^{3}$. We can always find a path $p_{t}$ in $\mathbb{RP}^{1}$ which takes $\{p_{0}^{1}, p_{0}^{2},p_{0}^{3}\}$ to $\{p_{1}^{1}, p_{1}^{2},p_{1}^{3}\}$. Any plane in $\mathbb{RP}^{3}$ is determined by three points. Consider the plane determined by the images of $\{p_{t}^{1}, p_{t}^{2},p_{t}^{3}\}$ under parametrisation of $C_{2}^{'}$ and denote it by $X_{t}$. This $X_{t}; 0 \leq t \leq 1$ gives an isotopy between the planes $X^{''}$ and $X^{'}$. $X_{t}$ intersects the knot $C_{2}^{'}$ in exactly four points say $a_{t},b_{t},c_{t}$ and $d_{t}$ for $0 \leq t \leq 1$ where $a_{0}$ and $a_{1}$ are the points which do not lie on the special conic. There is always an isotopy between $X^{''}$ and $X^{'}$ such that  $b_{t},c_{t}$ and $d_{t}$ are not collinear for any $0 \leq t \leq 1$. 
	
	If $a_{0}$ and $a_{1}$ lie in the same component of the plane at infinity minus the special conic, we can isotope $(C_{2}^{'},X^{''})$ to $(C_{2}^{'},X^{'})$ as follows:
	As $b_{t},c_{t}$ and $d_{t}$ are non collinear, we can always find a conic of signature $(2,1)$ which is special conic passing through $b_{t},c_{t},d_{t}$ and $a_{t}$ does not lie on the conic. If $a_{t}$ and $a_{0}$ lie in the different components then by using lemma \ref{rk1} it can be isotoped such that $a_{t}$ will lie in the same component of $a_{0}$ and $a_{1}$. Therefore, $(C_{2}^{'},X_{t})$ gives us an isotopy between $(C_{2}^{'},X^{''})$ to $(C_{2}^{'},X^{'})$.
	
	If $a_{0}$ and $a_{1}$ lie in the different components of the plane at infinity minus the special conic, we can switch the position of $a_{1}$ by using lemma \ref{rk1} and proceed as above.

\end{proof}

\begin{lemma}
	There does not exist any real rational curve of degree $5$ in $Q_{3,2}$ with more than two nodes.
\end{lemma}
\begin{proof}
	Suppose $C$ is a real rational curve of degree $5$ in $Q_{3,2}$ with at least three nodes. When we project this curve from one of the double points in $\mathbb{RP}^{3}$, we get a degree $3$ curve with atleast two nodes. But this type of curve does not exist in $\mathbb{RP}^{3}$. So, $Q_{3,2}$ does not have any real rational curve of degree $5$ with more than two double points. 
\end{proof}

We will now classify the degree ~5 curves with ~2 nodes and then knots.

\begin{theorem}
	There are $4$ rigid isotopy classes of real rational parametrisations of degree $5$ curves with exactly two nodes.
\end{theorem}

\begin{proof}
	By using lemma \ref{proj_double_pt}, when we project a curve $C$ with two double points in $Q_{3,2}$ from one of its double points we get a degree 3 curve in $\mathbb{RP}^{3}$ with one double point such that it intersects the plane at infinity in 3 points and exactly one point of intersection lies on the special conic. Degree~$3$ curves with exactly one double point are planar curves and there is only one class of these type of curves in $\mathbb{RP}^{3}$. The plane on which the degree~$3$ nodal curve lies and plane at infinity intersect in a projective line and the special conic intersects the plane of the curve in two points. So all the points of intersection must lie on this projective line.
	
	Depending upon the position of these points, there are 9 possibilities of these type of curves as shown in figure \ref{edges}. $c$ and $c^{'}$ are the points of intersection of the special conic and the plane of the degree~$3$ nodal curve. $Nc$ and $Nc^{'}$ points are points of intersection of degree~3 planar curve with the plane at infinity which does not lie on the special conic.
	
	\begin{figure}[h]
		\centering
		\includegraphics[width=0.5\textwidth]{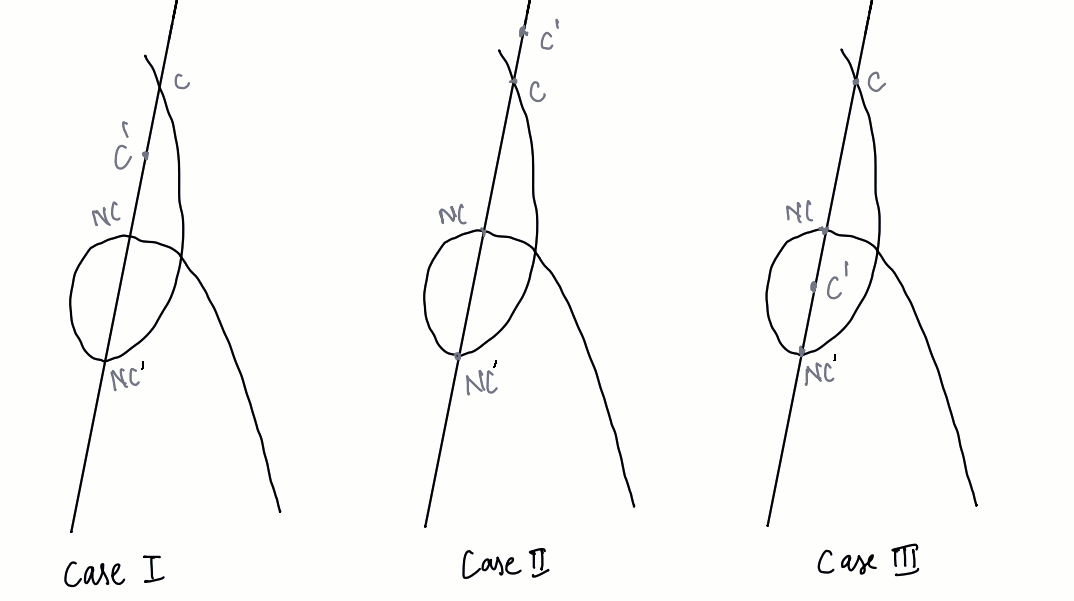}\hfill
		\includegraphics[width=0.5\textwidth]{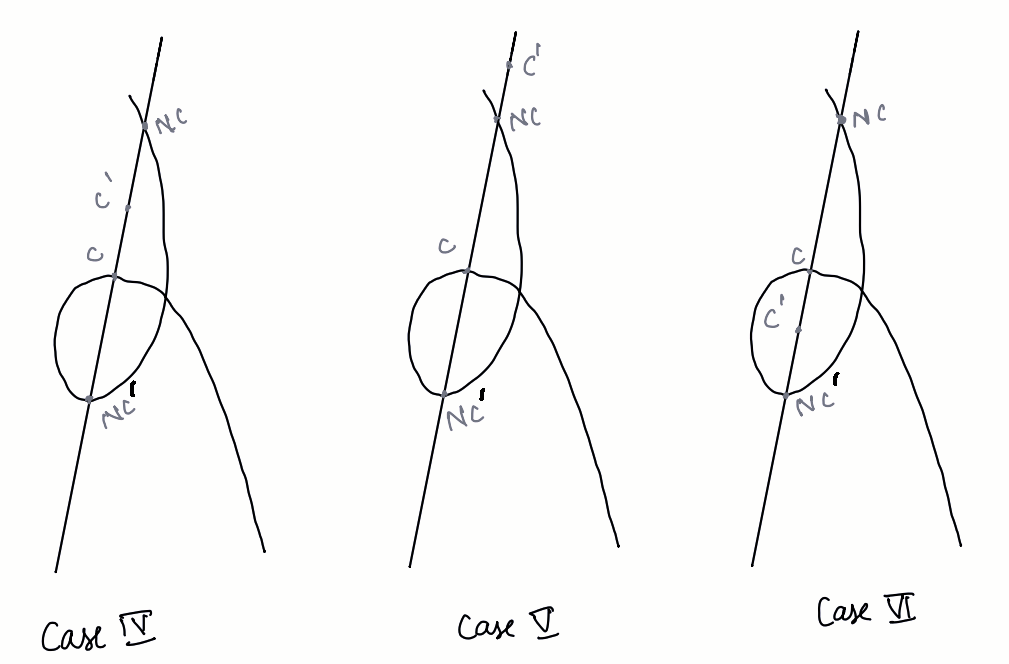}\hfill
		\includegraphics[width=0.5\textwidth]{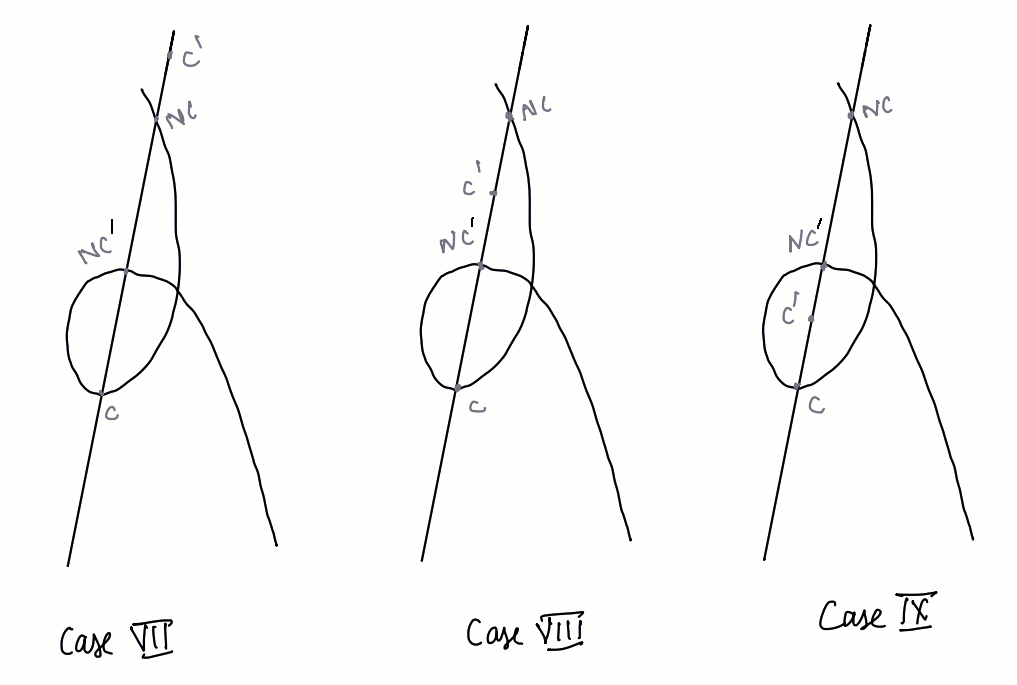}\hfill
		\caption{possibilities for degree ~3 walls with the required intersection with plane at infinity, $c,c^{'}$ are the points of intersection of the special conic and plane containing the curve. $NC, NC^{'}$ are the intersection points of curve with the plane at infinity which do not lie on the special conic}
		\label{edges}
	\end{figure}
	
	Case I and case II can be isotoped to each other as the conic is connected we will move the point from $c$ to $c^{'}$. Similarly, case IV can be isotoped to case VI and case VII can be isotoped to case IX. Case III is different from case I and case II as in case I and case II points not at the conic can be made imaginary but not in case III. Case VII and case VIII can be isotoped to case V and case IV respectively by using lemma \ref{rk1}. So we are left with four cases, I, III, IV, and V. Pull back of these four cases corresponds to degree $5$ real rational curves in $Q_{3,2}$. By using lemma \ref{proj_not_on_curve}, the projection of the curve with exactly two double points in $Q_{3,2}$ from a point which does not lie on the curve but on the $Q_{3,2}$ gives us a degree $5$ curve with exactly two double points in $\mathbb{RP}^{3}$, all of whose points of intersection with the plane at infinity lie on the special conic. Projections of the pull backs of these four cases correspond to different writhe (up to sign) edges of degree $5$ in $\mathbb{RP}^{3}$.
	
	So, there are $4$ rigid isotopy classes of degree $5$ edges which lie on $Q_{3,2}$.
\end{proof}

\begin{remark}
	Degree $5$ curves with exactly two double points lie on a hyperplane in $Q_{3,2}$. Therefore, by using lemma \ref{reflection}, these can be isotoped to their reflections in $\mathbb{RP}^{4}$. 
	
	By using lemma \ref{proj_not_on_curve}, the projection of the curve with exactly two double points in $Q_{3,2}$ from a point which does not lie on the curve but on the $Q_{3,2}$ gives us a degree $5$ curve with exactly two double points in $\mathbb{RP}^{3}$, all of whose points of intersection with the plane at infinity lie on the special conic. Four classes of degree $5$ edges in $Q_{3,2}$ correspond to different writhe (up to sign) edges of degree $5$ in $\mathbb{RP}^{3}$.
\end{remark}


\begin{theorem}
	There are~3 smooth isotopy classes of degree~5 real rational knots in $Q_{3,2}$.
\end{theorem}

\begin{proof}
	By using lemma \ref{proj_not_on_curve}, when we project a curve in the degree~$5$ edge from a point which does not lie on the curve then we get a degree $5$ curve with exactly two double points, say $E$, in $\mathbb{RP}^{3}$ whose five points of intersection with the plane at infinity lie on the special conic. As we know that in $\mathbb{RP}^{3}$, any degree~$5$ knot is adjacent to an edge so we can smoothly isotope $E$ to a knot while maintaining the right number of intersections with the plane at infinity.  This way this path of smooth isotopy can also be pulled back to the $Q_{3,2}$. There are $7$ classes of degree $5$ knots in $\mathbb{RP}^{3}$ with writhe $0,\pm2,\pm4, \pm6$ such that $0,\pm2$ are smoothly isotopic to line, $\pm4$ are long trefoil and $\pm6$ are projective $5_{3}$ knot. Moreover, knots belonging to the same smooth isotopy class can be smoothly isotoped while maintaining the right intersection number with the plane at infinity. So these can be pulled back and we have $3$ smooth isotopy classes of degree $5$ knots which lie on the $Q_{3,2}$ in $\mathbb{RP}^{4}$.
\end{proof}

\subsubsection{Rigid Isotopy classification of degree~5 knots in $Q_{3,2}$}
Suppose $K$ is a degree~$5$ knot lying on $Q_{3,2}$. The projection of this from a point which lies on the knot into a copy of $\mathbb{RP}^{3}$ is a real rational knot of degree $4$ which intersects the plane at infinity in $4$ points out of which $3$ points lie on the special conic. Suppose $K_{4,3}$ is the set of pairs $(C,X)$ in $\mathbb{RP}^{3}$ where $C$ is a degree $4$ knot and $X$ is a hyperplane in $\mathbb{RP}^{3}$ such that $C$ intersects $X$ in $4$ points and exactly three of them lies on a conic of signature $(2,1)$. The classification of knots of degree $5$ in $Q_{3,2}$ corresponds to the classification of pairs $(C,X)$ in $K_{4,3}$. In the pair $(C,X)$, only the points of intersection which lie on the conic can be imaginary because imaginary points exists in conjugate pairs and there is only one point not on the conic. So, there can be at most two imaginary points of intersection and these lie on the special conic. Take the line passing through the real points of intersection and consider the pencil of planes through that line. There is a plane in this pencil which intersects the curve in only real points. Move the plane along this and we get a path between the plane whose intersection points are imaginary to a plane whose points of intersection with curve are all real. Any four points determines a conic. We can avoid the situation that the three points of intersection which lie on the conic are collinear so, we can always choose a conic of signature $(2,1)$ throughout the isotopy containing the three points of intersection and one does not lies on the conic.

Suppose $(C,X)$ and $(C^{'},X^{'})$ are two pairs in $K_{4,3}$ where $a$ and $a^{'}$ are the points of intersection of $C$ and $C^{'}$ with $X$ and $X^{'}$ respectively, which do not lie on the conic. Now there are two possibilities, depending on whether $a$ and $a^{'}$ lie in the same or in different components of the plane at infinity minus the special conic. There are~4 knots of degree~4 in $\mathbb{RP}^{3}$ with writhe $\pm1,\pm3$ smoothly isotopic to a circle and a two crossing knot, respectively.

\begin{theorem}\label{deg5knots}
	There are exactly~4 rigid isotopy classes of degree~5 real rational knots lying on $Q_{3,2}$ in space $\mathbb{RP}^{4}$.
\end{theorem}

\begin{proof}
	
	Consider $(C_{0},X_{0})$ and $(C_{1},X_{1})$ are two pairs such that points of intersection of the curves with the respective planes are real.
	
	Suppose $C_{0}$ and $C_{1}$ have same writhe number either $1$ or $-1$ so $C_{0}$ and $C_{1}$ are rigidly isotopic. $C_{0}$ intersects the plane $X_{0}$ in four points say, $a^{'},b^{'},c^{'},d^{'}$ out of which three points $a^{'},b^{'},c^{'}$ lie on the special conic and $d^{'}$ does not lies on the conic. During the isotopy between $C_{1}^{'}$ and $C_{2}^{'}$, suppose the intersection points $a^{'},b^{'},c^{'},d^{'}$ continously deformed to points $a_{0},b_{0},c_{0},d_{0}$. Consider the unique plane determined by $a_{0},b_{0}$ and $c_{0}$ and denote it by $X_{0}^{'}$. This gives a isotopy between the pairs $(C_{0},X_{0})$ and $(C_{1},X_{0}^{'})$ so we can isotope pair $(C_{0},X_{0})$ to $(C_{1},X_{0}^{'})$. Note that the points of intersection of $C_{1}$ with $X_{0}^{'}$ are real. $C_{1}$ intersects each of the planes $X_{0}^{'}$ and $X_{1}$ in $4$ points denoted by $a_{0},b_{0},c_{0},d_{0}$ and $a_{1},b_{1},c_{1}, d_{1}$ respectively where $a_{0}$ and $a_{1}$ are the points which do not lie on the special conic. Denote the preimages of $b_{0},c_{0},d_{0}$ under the parametrisation of knot $C_{1}$ by $p_{0}^{1}, p_{0}^{2},p_{0}^{3}$ and the preimages of $b_{1},c_{1},d_{1}$ under the parametrisation of knot $C_{1}$ by $p_{1}^{1}, p_{1}^{2},p_{1}^{3}$. We can always find a path $p_{t}$ in $\mathbb{RP}^{1}$ which takes $\{p_{0}^{1}, p_{0}^{2},p_{0}^{3}\}$ to $\{p_{1}^{1}, p_{1}^{2},p_{1}^{3}\}$. Any plane in $\mathbb{RP}^{3}$ is determined by three points. Consider the plane determined by the images of $\{p_{t}^{1}, p_{t}^{2},p_{t}^{3}\}$ and denote it by $X_{t}$. This $X_{t}; 0 \leq t \leq 1$ gives an isotopy between the planes $X_{0}^{'}$ and $X_{1}$. Each $X_{t}$ intersects the knot $C_{1}$ in exactly four points say $a_{t},b_{t},c_{t}$ and $d_{t}$ for $0 \leq t \leq 1$ where $a_{0}$ and $a_{1}$ are the points which do not lie on the special conic. There is always an isotopy between $X_{0}^{'}$ and $X_{1}$ such that $b_{t},c_{t}$ and $d_{t}$ are not collinear for any $0 \leq t \leq 1$. 
	
	If $a_{0}$ and $a_{1}$ lie in the same component of the plane at infinity minus the special conic, we can isotope $(C_{1},X_{0}^{'})$ to $(C_{1},X_{1})$ as follows:
	As $b_{t},c_{t}$ and $d_{t}$ are non collinear, we can always find a conic of signature $(2,1)$ which is special conic passing through $b_{t},c_{t},d_{t}$ and $a_{t}$ does not lie on the conic. If $a_{t}$ and $a_{1}$ lie in different components of the plane at infinity minus the special conic then by using lemma \ref{rk1}, it can be isotoped such that $a_{t}$ will lie in the same component of $a_{0}$ and $a_{1}$. Therefore, $(C_{1},X_{t})$ gives us an isotopy between $(C_{1},X_{0}^{'})$ to $(C_{1},X_{1})$.
	
	If $a_{0}$ and $a_{1}$ lie in the different components of the plane at infinity minus the special conic, then we can change the position of $a_{1}$ by using lemma \ref{rk1}, as shown in figure \ref{figure 5} and proceed as above.
	
	
	\begin{figure}[h]
		\centering
		\includegraphics[width=0.4\textwidth]{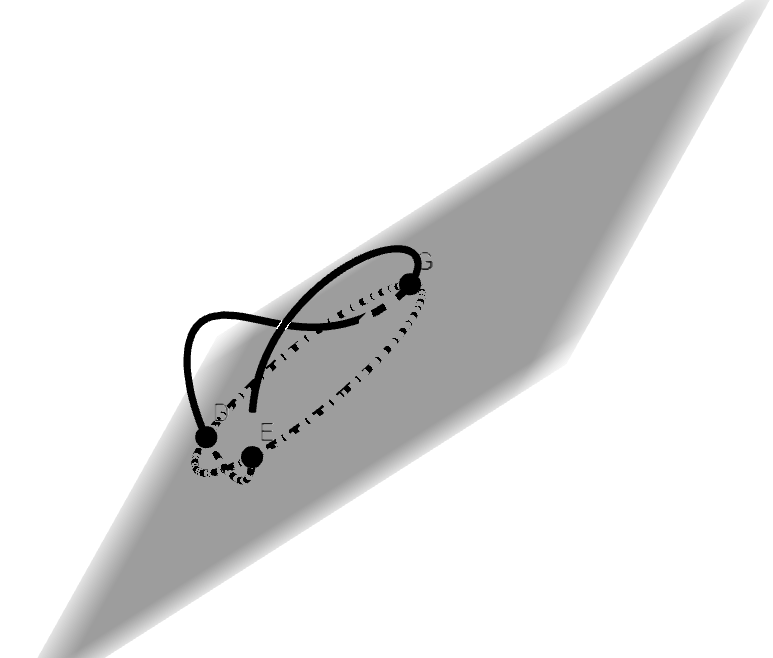}
		\includegraphics[width=0.4\textwidth]{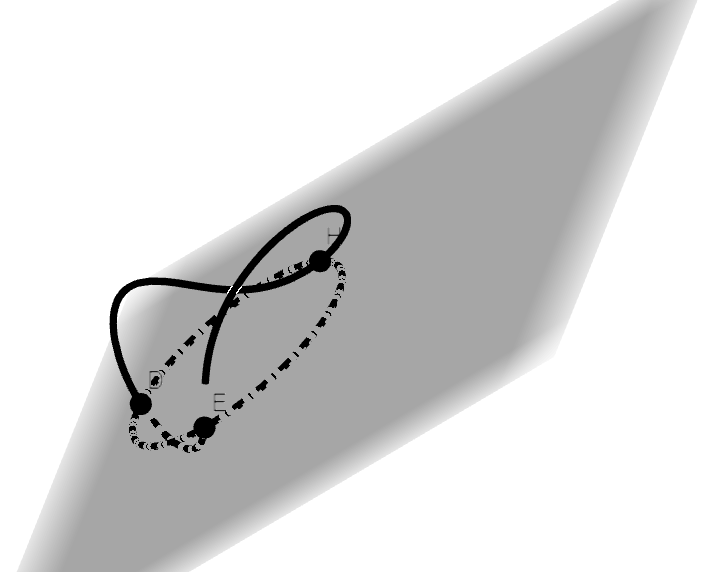}
		\includegraphics[width=0.4\textwidth]{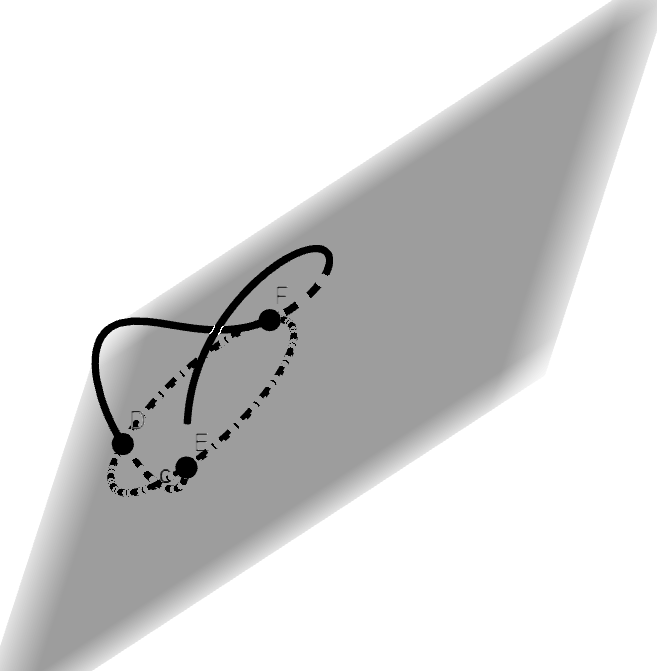}
		\caption{isotopy between pairs $(C,X)$ and $(C^{'},X^{'})$ where $C$ and $C^{'}$ are same writhe knots of $\pm{1}$ and points of intersection not on conic lie in different components.}
		\label{figure 5}
	\end{figure}
	This way we get two rigid isotopy classes of $(C,X)$ each one when $C$ is a knot of writhe $\pm1$.
	
	If $C$ and $C^{'}$ are knots of writhe $\pm 3$ which are isotopic to the two crossing knot, we can find the isotopy classes in the similar way. As in figure \ref{figure 6}, consider a plane intersecting it in the desired way then it intersects both strands of the two crossing knot at two points. One point out of these does not lie on the special conic. If in any other pair $(C^{'},X^{'})$  where $C^{'}$ is isotopic to $C$, the point not on the conic is in the different component then by using lemma \ref {rk1} we can move the plane $X$ in such a way that it becomes tangent to the curve at the point on the conic obtained as shown in figure \ref{figure 6}. So in this way we can interchange the position of the point not on the conic in different components by moving the plane.
	\begin{figure}[h]
		\centering
		\includegraphics[width=0.45\textwidth]{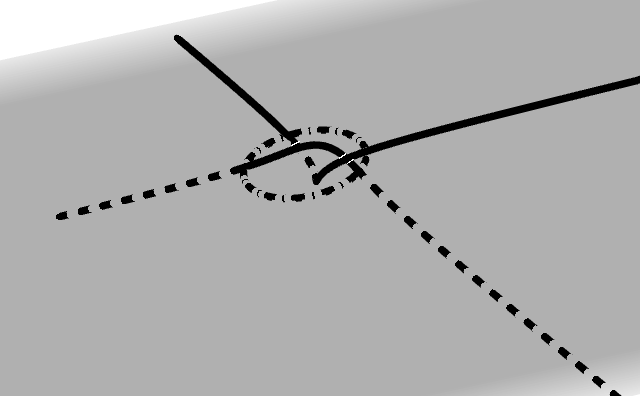}\hfill
		\includegraphics[width=0.45\textwidth]{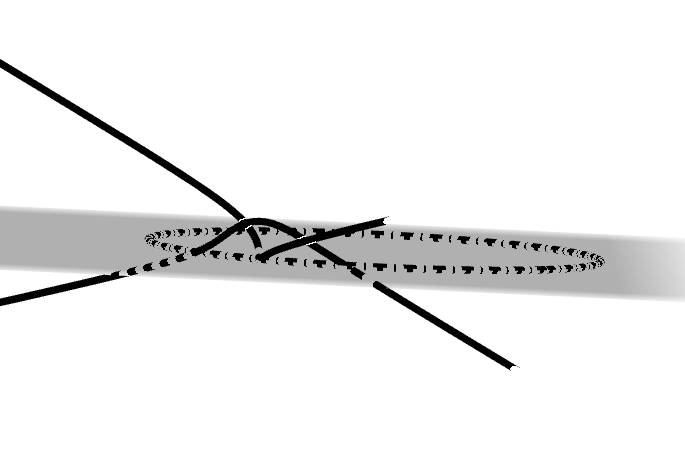}\hfill
		\vspace{0.5 cm}
		\includegraphics[width=0.45\textwidth]{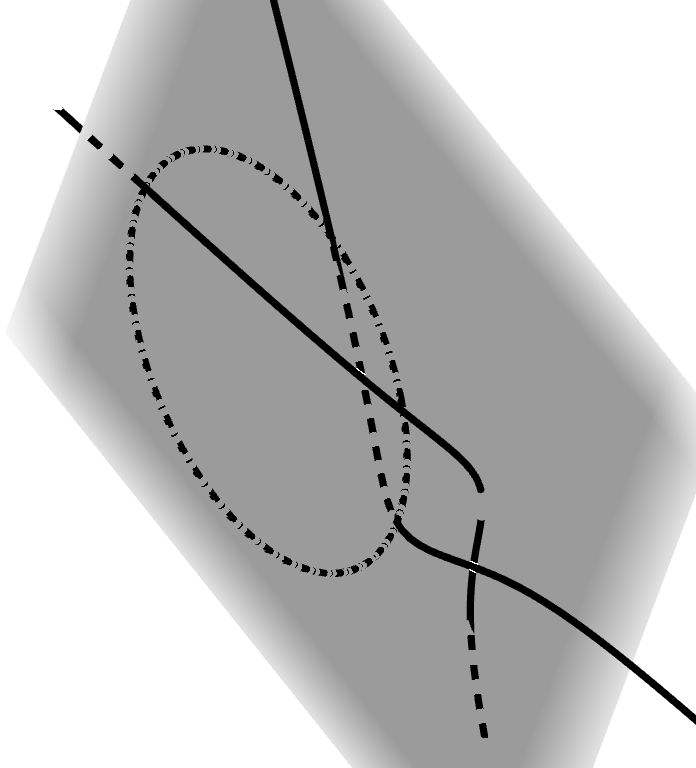}\hfill
		\caption{isotopy between pairs $(C,X)$ and $(C^{'},X^{'})$ where $C$ and $C^{'}$ are same writhe knots of $\pm{3}$ and points of intersection not on conic lie in different components.}
		\label{figure 6}
	\end{figure}
	
	This gives two other classes of $(C,X)$ pairs with required intersection and where $C$ is a knot of degree $4$ with encomplexed writhe number $\pm3$.
	
	These $4$ pairs corresponds to $4$ rigid isotopy classes of degree $5$ knots in $Q_{3,2}$. Projections of these four rigid isotopy classes in $\mathbb{RP}^{3}$ gives us four different rigid isotopy classes of degree $5$ in $\mathbb{RP}^{3}$ up to reflections. Pull backs of degree $5$ knots of writhe $0,2,4$ and $6$ in $\mathbb{RP}^{3}$ give us different isotopy classes in $Q_{3,2}$ of degree $5$ as shown in figure \ref{b}, \ref{c}, \ref{d} and \ref{e}.
\end{proof}

\chapter{Constructing knots in $Q_{3,2}$ and some general results}\label{chap-5}
In this \myref{}, we extend the gluing technique for the real rational knots in $Q_{3,2}$ and
able to construct representatives for each class up to degree 5. Some general results are
also presented in this \myref{}.

\section{Gluing}
Bj\"{o}rklund gives us a method to construct knots of higher degree in $\mathbb{RP}^{3}$ by using knots of smaller degree. He stated that:

\begin{theorem}\textbf{(Bj\"{o}rklund)}\label{glu}\cite{MR2844809}
	Suppose $C_{1}$ and $C_{2}$ are two oriented real rational knots of degree $d_{1}$ and $d_{2}$ respectively in $\mathbb{RP}^{3}$ such that they intersect transversally at a point, say $p$ then there exists knot parametrisations $[p_{o}:p_{1}:p_{2}:p_{3}]$ and $[q_{o}:q_{1}:q_{2}:q_{3}]$ such that $[p_{o}q_{0}:p_{1}q_{0}+p_{0}q_{1}:p_{2}q_{0}+p_{0}q_{2}:p_{3}q_{0}+q_{3}p_{0}]$ of degree $d_{1}+d_{2}$ is rigidly isotopic to union of two knots $C_{1}$ and $C_{2}$ except in a small neighbourhood of point $p$. The resulting knot is the knot obtained by smoothing of strands of $C_{1}$ and $C_{2}$ at the intersection point. 
\end{theorem}

By using this and projections we can extend the gluing technique to knots lying on $Q_{3,2}$ in $\mathbb{RP}^{4}$.

\begin{theorem} \label{gluing in hyperboloid} If $K_{1}$ and $K_{2}$ are two real rational knots with degrees $d_{1}$ and $d_{2}$ respectively in $\mathbb{RP}^{4}$ lying on $Q_{3,2}$ which intersects each other only at one point $p$, then we can construct a real rational knot of degree $d_{1}+d_{2}$ which is isotopic to the union of original two knots except in a small neighbourhood of point p.
\end{theorem}

\begin{proof} Consider two knots $K_{1}$ and $K_{2}$ in $\mathbb{RP}^{4}$ lying on $Q_{3,2}$  with degrees $d_{1}$ and $d_{2}$ respectively which intersects at only one point $p$. Now we project these on a hyperplane from a point q which does not lies on both the knots but lies in $Q_{3,2}$. After applying the projection $\pi_{q}$, we get two curves $\pi_{q}(K_{1})$ and $\pi_{q}(K_{2})$ of degree $d_{1}$ and $d_{2}$ respectively in $\mathbb{RP}^{3}$ which intersects the plane at infinity (say $x_{0}=0$) in $d_{1}$ and $d_{2}$ points respectively and by Lemma~\ref{proj_not_on_curve}, all these points lie on the special conic. Notice that $\pi_{q}(K_{1})$ and $\pi_{q}(K_{2})$ also intersect in one point $\pi_{q}(p)$. Now by the previous theorem \ref{glu}, there are parametrisations $[p_{0}:p_{1}:p_{2}:p_{3}]$ and $[q_{0}:q_{1}:q_{2}:q_{3}]$ of $\pi_{q}(K_{1})$ and $\pi_{q}(K_{2})$ respectively of degree $d_{1}$ and $d_{2}$ such that $[p_{0}q_{0}:p_{1}q_{0}+q_{1}p_{0}:p_{2}q_{0}+q_{2}p_{0}:p_{3}q_{0}+q_{3}p_{0}]$ is a knot parametrisation of $K$ of degree $d_{1}+d_{2}$ which is isotopic to the union of two knots outside a small neighbourhood of point $\pi_{q}(p)$. Suppose $[a_{0}:a_{1}:a_{2}:a_{3}]$ is one of the points of intersection of $\pi_{q}(K_{1})$ with the plane at infinity. Then this lies on the special conic and $a_{0}=0$. Suppose the preimage of $[a_{0}:a_{1}:a_{2}:a_{3}]$ under the parametrisation $[p_{0}:p_{1}:p_{2}:p_{3}]$ of $\pi_{q}(K_{1})$  is $[\alpha:\beta]$ and image of $[\alpha:\beta]$ under $[q_{0}:q_{1}:q_{2}:q_{3}]$  of $\pi_{q}(K_{2})$ is $[b_{0}:b_{1}:b_{2}:b_{3}]$. Now image of $[\alpha:\beta]$ under the parametrisation $[p_{0}q_{0}:p_{1}q_{0}+q_{1}p_{0}:p_{2}q_{0}+q_{2}p_{0}:p_{3}q_{0}+q_{3}p_{0}]$ of K is $[a_{0}b_{0}:a_{1}b_{0}+a_{0}b_{1}:a_{2}b_{0}+b_{2}a_{0}:a_{3}b_{0}+b_{3}a_{0}]=[0b_{0}:a_{1}b_{0}+0b_{1}:a_{2}b_{0}+b_{2}0:a_{3}b_{0}+b_{3}0]= [0:a_{1}b_{0}:a_{2}b_{0}:a_{3}b_{0}]=[0:a_{1}:a_{2}:a_{3}]$. Therefore, any point on the curve that lies on the plane at infinity remains unchanged even after gluing.  Since before gluing, the points of intersections of the plane at infinity with both $\pi_{q}(K_{1})$ and $\pi_{q}(K_{2})$  lie on the special conic, they continue to do so even after gluing. So the parametrisation of the glued knot intersects the plane at infinity in $d_{1}+d_{2}$ points and all lie on the special conic. Therefore, by Lemma~\ref{proj_not_on_curve}, it can be pulled back to $Q_{3,2}$.
\end{proof}

\label{lines} In \cite{MR2844809}, it was shown that any real rational knot of degree $d \leq 5$ in $\mathbb{RP}^{3}$ can be constructed by gluing a line with a real rational knot of degree $d-1$ in $\mathbb{RP}^{3}$. This, along with the previous theorem will help us to construct representatives of the rigid isotopy class allowed earlier by gluing lines. We do this by induction by proving the induction hypothesis in the following theorem:



\begin{theorem} Every knot of degree $d \leq 5$ in $\mathbb{RP}^{4}$ which lies on the $Q_{3,2}$ is isotopic to the glued knot obtained by gluing a knot of degree $d-1$ on the $Q_{3,2}$ with a line in $Q_{3,2}$.
\end{theorem}

\begin{proof}
	\textbf{Degree $1$} knots are already lines so we are done.

	\textbf{Degree $2$}:
	The projection of degree $2$ knots in $Q_{3,2}$ corresponds to degree $2$ knots in $\mathbb{RP}^{3}$ both of whose points of intersection with the plane at infinity lie on the special conic. There is only one rigid isotopy class of degree $2$ real rational knots in $\mathbb{RP}^{3}$ but the position of the points of intersection with the plane at infinity and orientation of the knots specifies a rigid isotopy class in $Q_{3,2}$ as shown in figure \ref{figure 7}.
	
	\begin{figure}[htbp]
		\centering
		\includegraphics[width=0.39\textwidth]{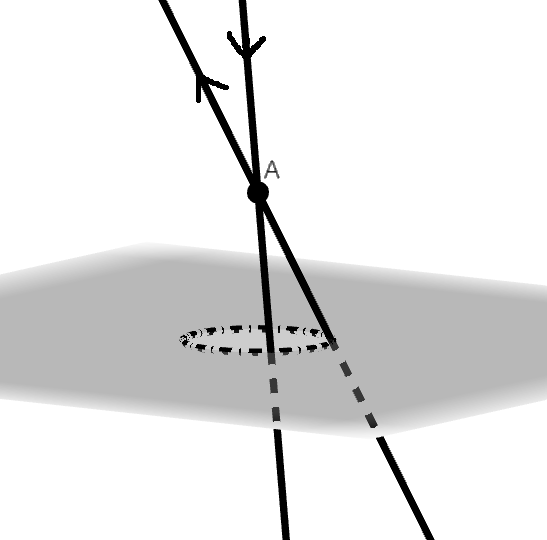}\hfill
		\includegraphics[width=0.39\textwidth]{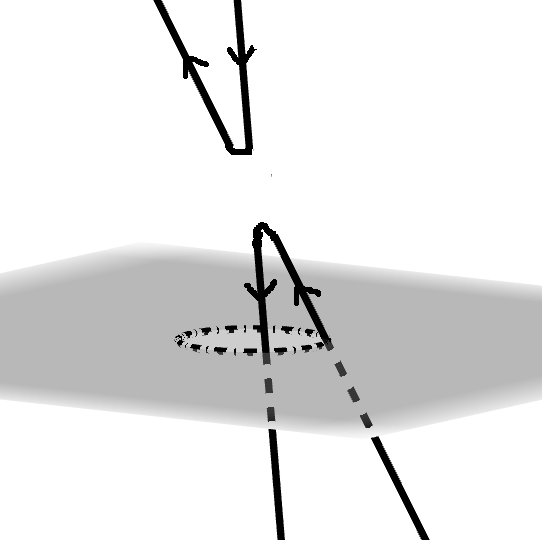}\hfill
		\vspace{0.06 cm}
		\includegraphics[width=0.39\textwidth]{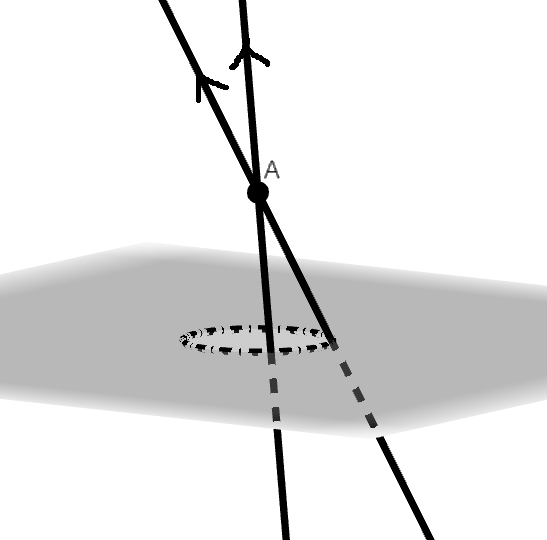}\hfill
		\includegraphics[width=0.39\textwidth]{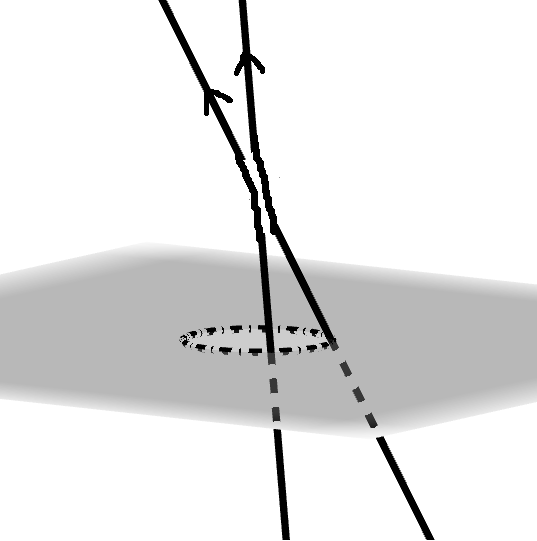}\hfill
		\includegraphics[width=0.39\textwidth]{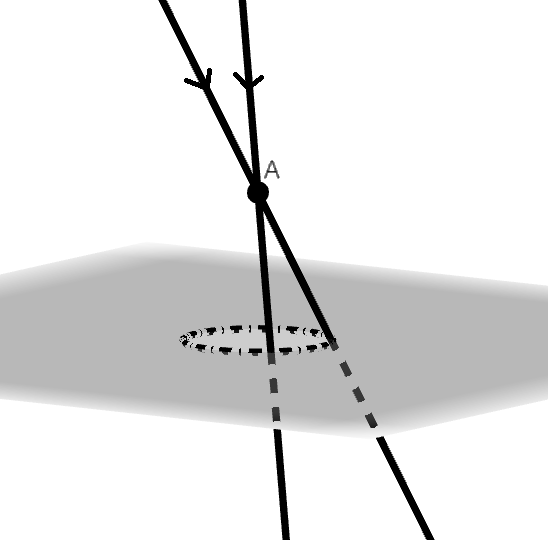}\hfill
		\includegraphics[width=0.39\textwidth]{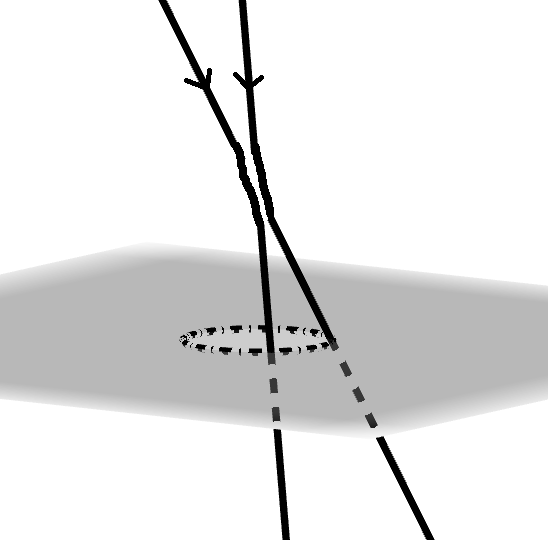}\hfill
		\caption{projection of degree $2$ real rational knots in $Q_{3,2}$ as glued knots.}
		\label{figure 7}
	\end{figure}
	
	Take a line $L_{1}$ in $\mathbb{RP}^{3}$ as shown in figure \ref{figure 7}. This line intersects the plane at infinity at a single point and it is trivial to find a conic of signature $(2,1)$ passing through this point. Choose other line $L_{2}$ which intersects $L_{1}$ at a point and intersects the plane at infinity at a point on the special conic. By using theorem \ref{glu}, glue these two lines in different ways by changing the orientation of $L_{1}$ and $L_{2}$ which gives us three different knots of degree $2$ as shown in figure \ref{figure 7}. These glued knots intersect the plane at infinity at two points both of which lie on the special conic and these can be pulled back to $Q_{3,2}$.
	
	To show that the pull backs correspond to three isotopy classes of degree $2$ in $Q_{3,2}$, project the degree $2$ knots in $Q_{3,2}$ from a point on the curve in $\mathbb{RP}^{3}$. Just before gluing, the projection of pair of lines is a line which intersects the plane at infinity at special conic point. After the gluing, tangent at point $p$ intersects plane of projection inside the special conic and outside the special conic with respect to the thing that which type of gluing we have done. This projection looks like as in figure \ref{figure 8}.
	\begin{figure}[htbp]
		\centering
		\includegraphics[width=0.39\textwidth]{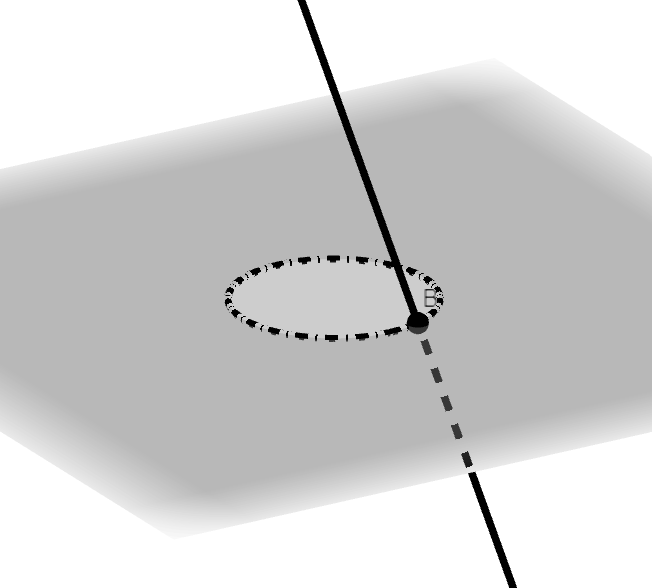}\hfill
		\includegraphics[width=0.39\textwidth]{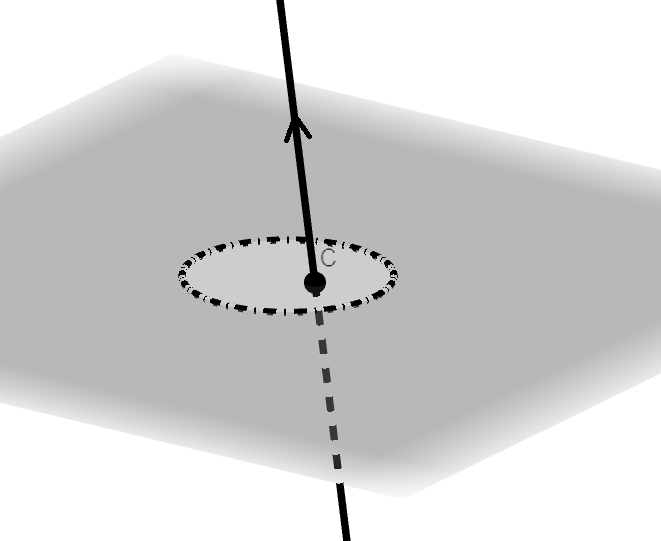}\hfill
		\includegraphics[width=0.39\textwidth]{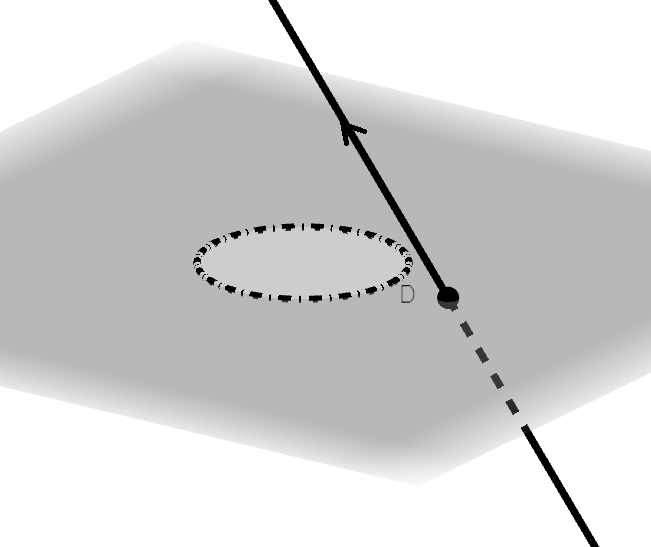}\hfill
		\includegraphics[width=0.39\textwidth]{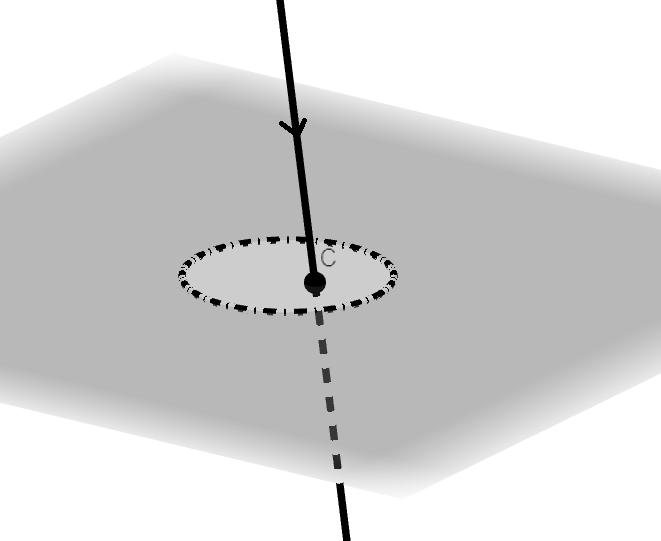}\hfill
		\caption{rigid isotopy classes of projection of degree $2$ curves in $Q_{3,2}$ from a point on curve.}
		\label{figure 8}
	\end{figure}

	\textbf{Degree 3:} Real rational knots of degree $3$ in $Q_{3,2}$ can be projected to degree $3$ real rational knot in $\mathbb{RP}^{3}$ all of whose points of intersection with the plane at infinity lie on the special conic. Any degree $3$ real rational knot in $\mathbb{RP}^{3}$ can be constructed by gluing a real rational knot of degree $2$ with a line in $\mathbb{RP}^{3}$. It is always possible to find a conic of signature $(2,1)$ passing through the intersection points of degree $2$ rational knot and the plane at infinity. Glue this two degree knot to a line passing through a point on $2$ degree knot and a point on the conic of signature $(2,1)$. The glued knot is a degree $3$ real rational knot in $\mathbb{RP}^{3}$ and all of its $3$ points of intersection with the plane at infinity lie on the conic of signature $(2,1)$ (as shown in figure \ref{figure 9}) which can be pulled back to $Q_{3,2}$. This pull back gives us a degree $3$ rigid isotopy class in $Q_{3,2}$.\\
	The second isotopy class of degree $3$ knot is constructed by gluing a degree $2$ knot with a line but with the opposite orientation.
	\begin{figure}[htbp]
		\centering
		\includegraphics[width=0.39\textwidth]{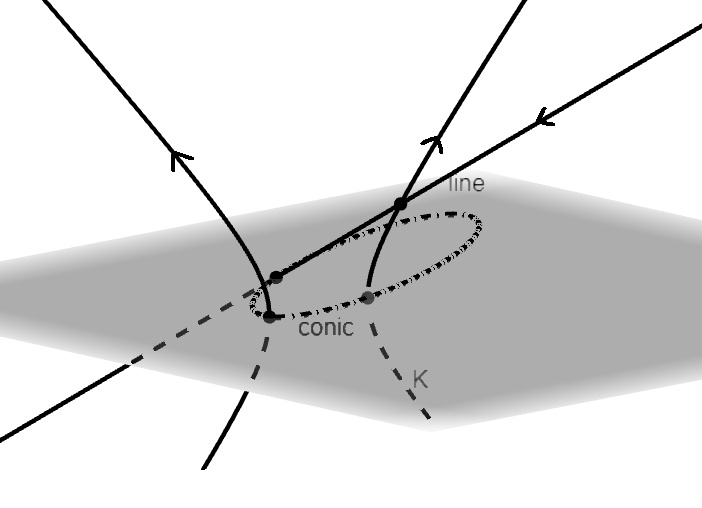}\hfill
		\includegraphics[width=0.39\textwidth]{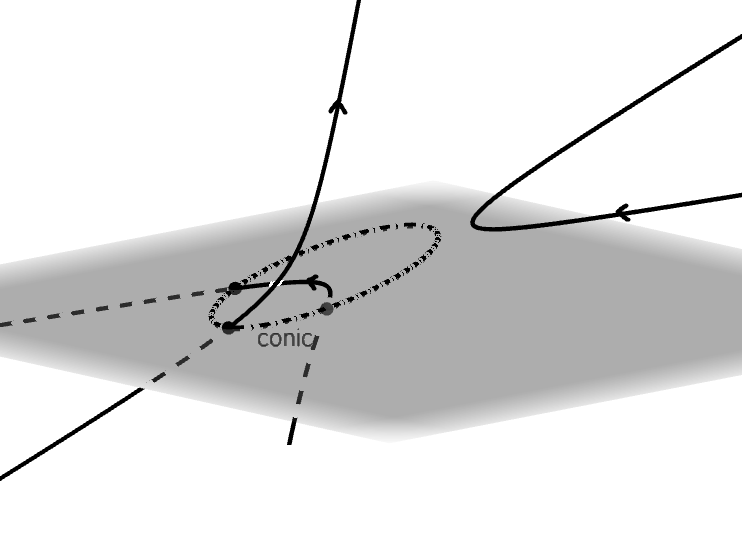}\hfill
		\includegraphics[width=0.39\textwidth]{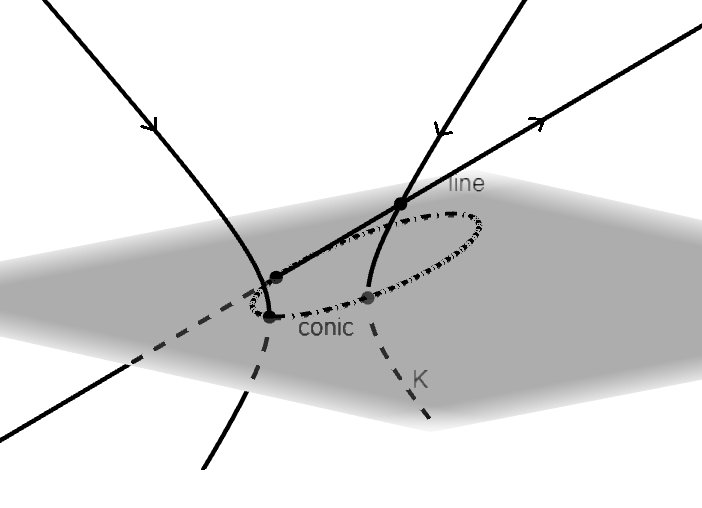}\hfill
		\includegraphics[width=0.39\textwidth]{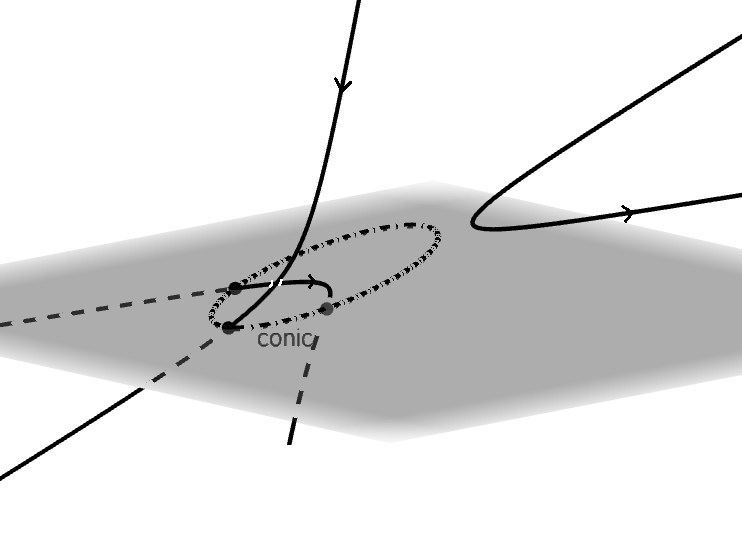}\hfill
		\caption{The projection of degree $3$ real rational knot in $Q_{3,2}$ as glued knot.}
		\label{figure 9}
	\end{figure}

	\textbf{Degree 4:} There are two classes of degree $4$ real rational knots lying on $Q_{3,2}$. When we project these from a point not on the knot but on the $Q_{3,2}$, the projected knot is a degree $4$ real rational knot in $\mathbb{RP}^{3}$ and all of its $4$ points of intersection with the plane at infinity lies on the special conic. Any real rational knot of degree $4$ in $\mathbb{RP}^{3}$ can be constructed by gluing a degree $3$ real rational knot with a line. 
	
	Any conic in $\mathbb{RP}^{2}$ can be determined by $5$ points in the plane. So, it is possible to find a conic of signature $(2,1)$ lying on the plane at infinity which passes through intersection points of degree $3$ knot and the plane at infinity. Take a line intersecting the conic lying in the plane at infinity and the degree $3$ knot at a point. Glue this line with both orientations to degree $3$ knot by using theorem \ref{glu} to get two real rational knots of degree 4 with encomplexed writhe $1$ and $3$ as shown in figure \ref{figure 10} and figure \ref{figure a} respectively.
	\begin{figure}[htbp]
		\centering
		\includegraphics[width=0.5\textwidth]{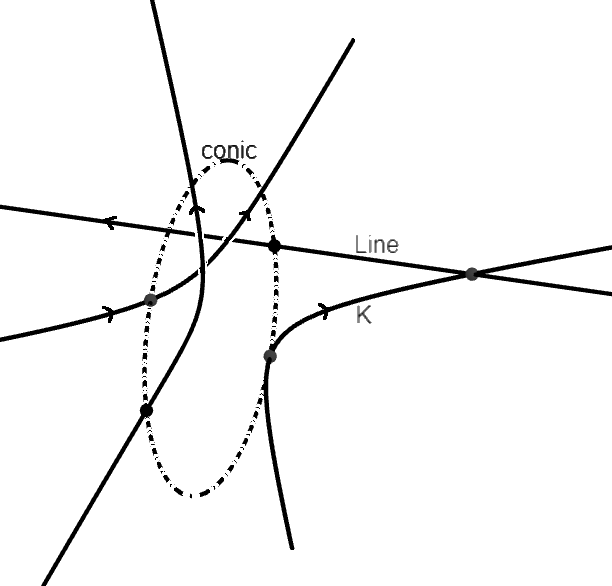}\hfill
		\includegraphics[width=0.5\textwidth]{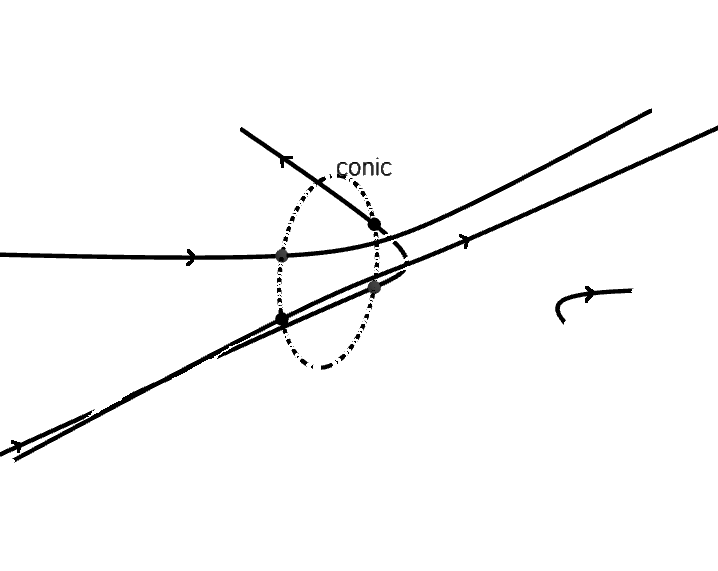}\hfill
		\caption{The projection of degree 4 knot in $Q_{3,2}$ of writhe $1$ as a glued knot.}
		\label{figure 10}
	\end{figure} 
	\begin{figure}[htbp]
		\centering
		\includegraphics[width=0.5\textwidth]{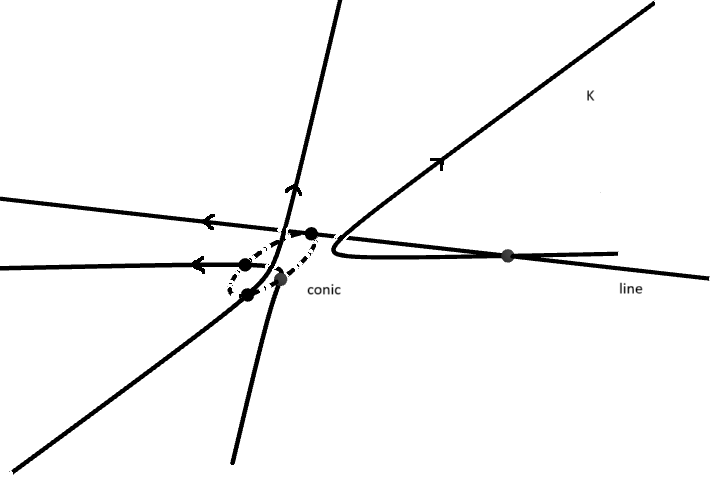}\hfill
		\includegraphics[width=0.5\textwidth]{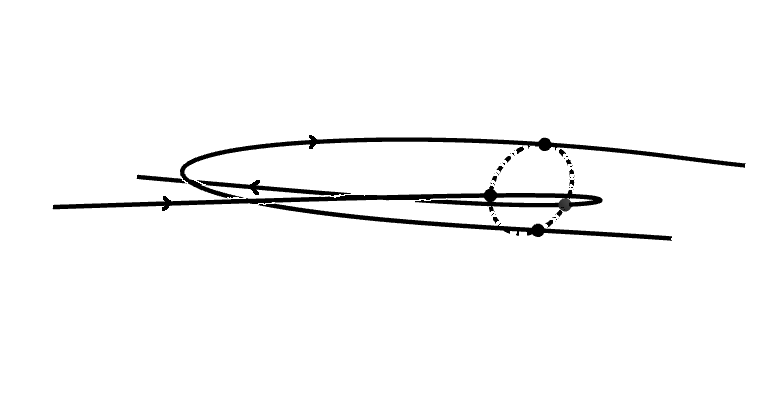}\hfill
		\caption{The projection of degree 4 knot in $Q_{3,2}$ of writhe $3$ as a glued knot.}
		\label{figure a}
	\end{figure}
	
	Gluing these two at point $p$ in $\mathbb{RP}^{3}$ gives us a two crossing knot and writhe $1$ knot of degree $4$ such that their points of intersection with the plane at infinity lies on the special conic. These glued curves can be pulled back to $Q_{3,2}$ and they represent two distinct classes of degree $4$.

	\textbf{Degree 5:} As seen in theorem \ref{deg5knots}, there are four rigid isotopy classes of degree $5$ real rational knots in $Q_{3,2}$ and after projection, these correspond to four different writhe knots of degree $5$ in $\mathbb{RP}^{3}$ up to sign. Any real rational knot of degree $5$ in $\mathbb{RP}^{3}$ can be constructed by gluing a degree $4$ real rational knot with a line.
	
	A degree $5$ knot in $Q_{3,2}$ which is a pull back of a real rational knot of encomplexed writhe $0$ of degree $5$ in $\mathbb{RP}^{3}$, can be constructed by gluing a degree $4$ real rational knot of encomplexed writhe $1$ with a line in $\mathbb{RP}^{3}$ as shown in figure \ref{b}.  The degree $4$ knot intersects the plane at infinity at $4$ points and there is a pencil of conics passing through these $4$ points. We choose conic of signature $(2,1)$ and glue this degree $4$ knot to a line by ensuring that intersection of the line with the plane at infinity also lies on the signature $(2,1)$ conic. The pull back of this is a degree $5$ knot in $Q_{3,2}$.
	\begin{figure}[htbp]
		\centering
		\includegraphics[width=0.5\textwidth]{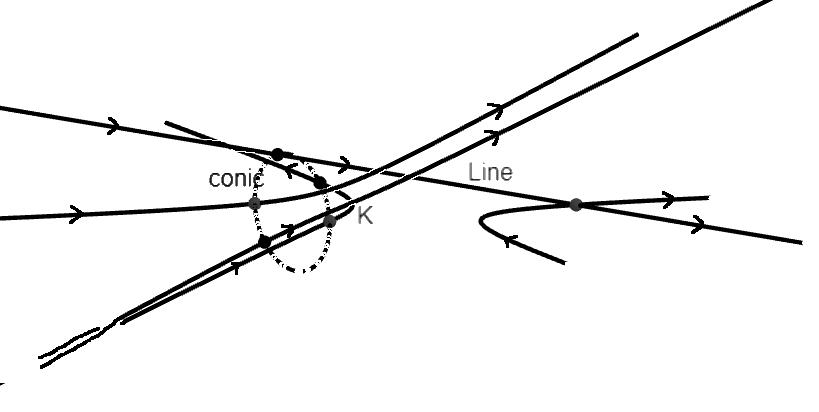}\hfill
		\includegraphics[width=0.5\textwidth]{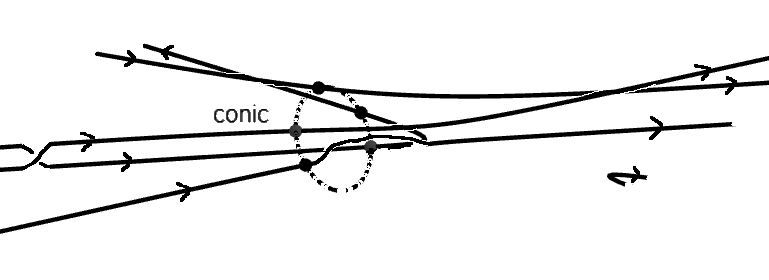}\hfill
		\caption{The projection of degree $5$ knot in $Q_{3,2}$ of writhe $0$ as a glued knot.}
		\label{b}
	\end{figure}  
	
	Consider a real rational knot of degree $4$ with encomplexed writhe $3$. This intersects the plane at infinity in four points and there is a pencil of conics passing through these $4$ points. Choose a conic of signature $(2,1)$ as shown in figure \ref{c}. Glue this with a line intersecting the knot at one point and a point on the special conic. The glued knot is a degree $5$ real rational knot of encomplexed writhe $2$ all of whose points of intersection lie on the special conic. This glued knot can be pulled back to a real rational knot of degree $5$ in $Q_{3,2}$.
	\begin{figure}[htbp]
		\centering
		\includegraphics[width=0.5\textwidth]{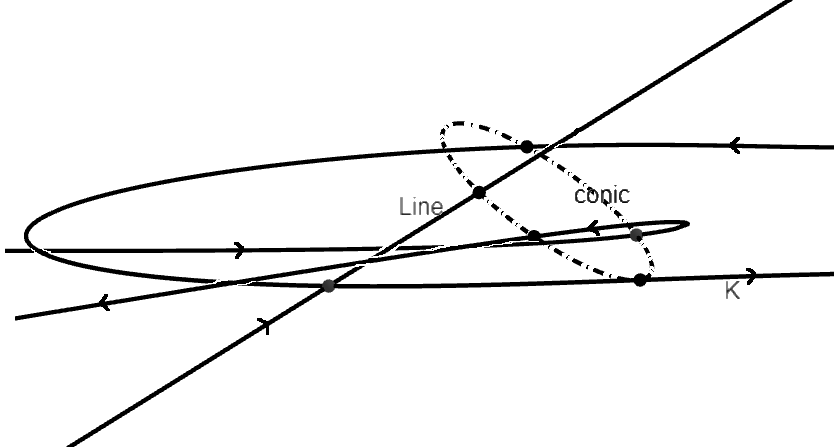}\hfill
		\includegraphics[width=0.5\textwidth]{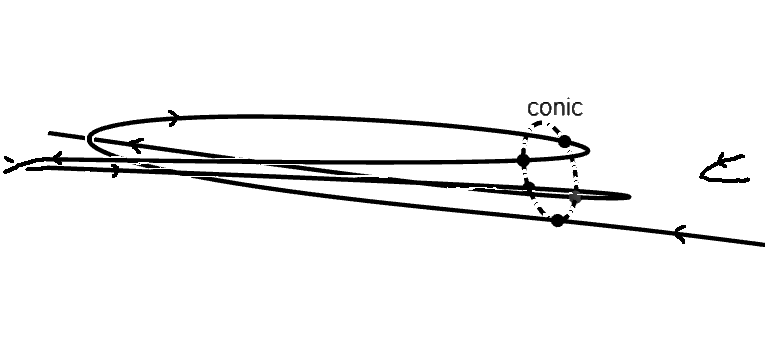}\hfill
		\caption{The projection of degree $5$ knot in $Q_{3,2}$ of writhe $2$ as a glued knot.}
		\label{c}
	\end{figure}  
	
	A real rational knot of degree $5$ of encomplexed writhe $4$ can be constructed by gluing a degree $4$ knot of encomplexed writhe $3$ with a line which intersects the degree $4$ knot and the plane at infinity at point. The glued knot is a degree $5$ knot with encomplexed writhe $4$. This glued knot intersects the plane at infinity at $5$ points and a conic of signature $(2,1)$ is determined by these $5$ points as shown in figure \ref{d}. We can then apply a projective transformation to transform the space so that the conic of signagure $(2,1)$ is sent to the special conic at infinity. Then the image, under this tranformation, of the glued real rational knot of degree $5$ with encomplexed writhe $4$ can be pulled back to $Q_{3,2}$ giving us an isotopy class of degree $5$ knots.
	\begin{figure}[htbp]
		\centering
		\includegraphics[width=0.5\textwidth]{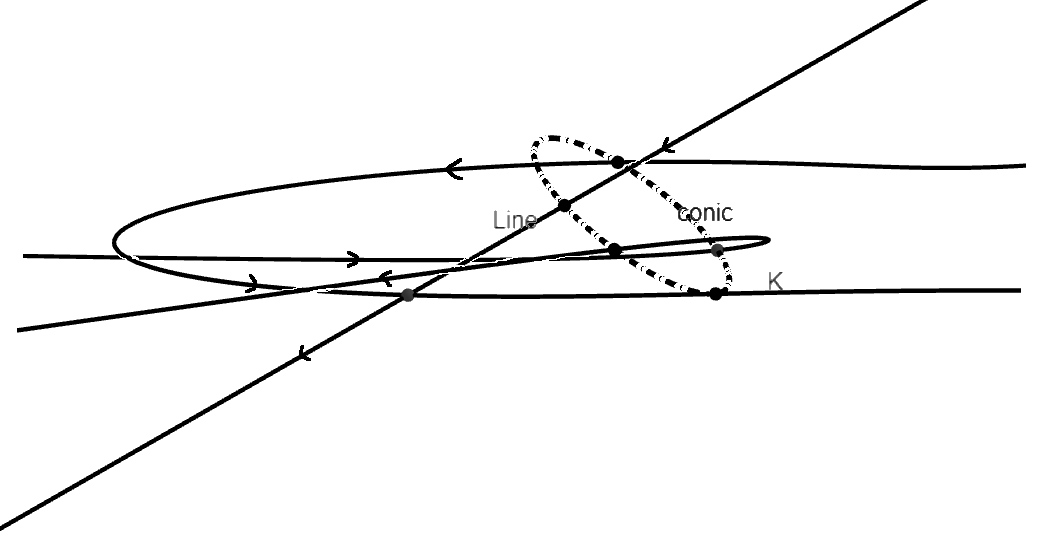}\hfill
		\includegraphics[width=0.5\textwidth]{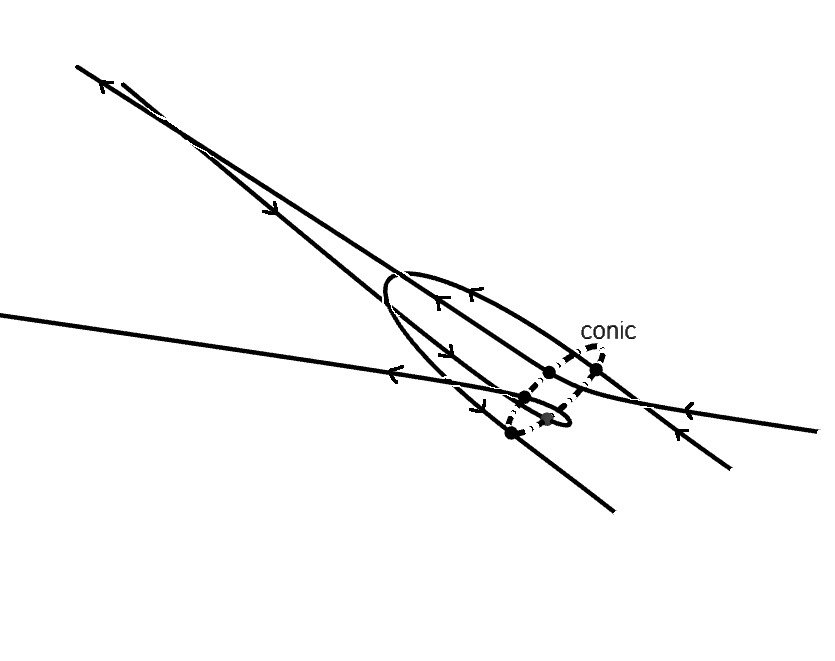}\hfill
		\caption{The projection of degree $5$ knot in $Q_{3,2}$ of writhe $4$ as a glued knot.}
		\label{d}
	\end{figure}
	
	Consider a real rational knot of degree $4$ with encomplexed writhe $3$. The degree $4$ knot intersects the plane at infinity at $4$ points and there is a pencil of conic passing through these. We choose a conic with signature $(2,1)$. Glue this degree $4$ knot with a line which intersects knot at a point and intersects the plane at infinity at a point which lies on the special conic. This glued knot is a degree $5$ real rational knot of writhe $6$ in $\mathbb{RP}^{3}$ and intersects the plane at infinity at $5$ points as shown in figure \ref{e}. As before, use a projective transformation that transforms the conic with signature $(2,1)$ to the special conic at infinity. The $5$ points of the image of the glued knot now lie on the special conic. So, this can be pulled back in $Q_{3,2}$ to a degree $5$ real rational knot.
	\begin{figure}[htbp]
		\centering
		\includegraphics[width=0.5\textwidth]{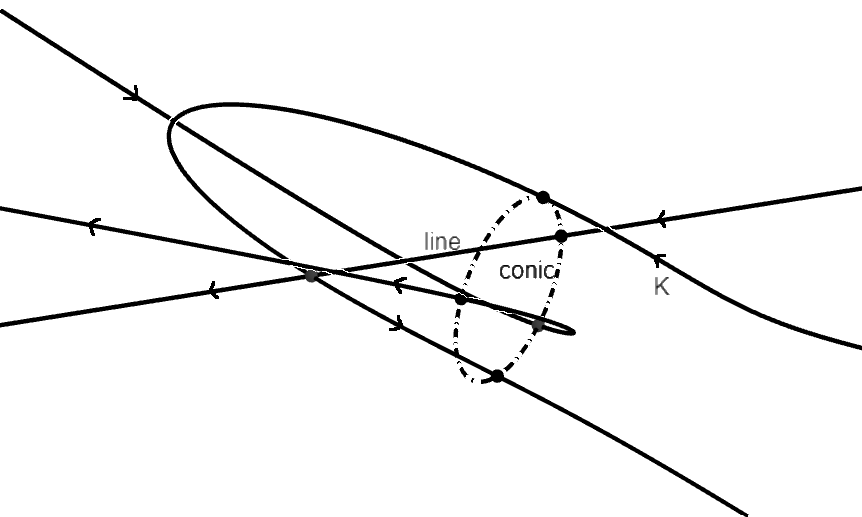}\hfill
		\includegraphics[width=0.5\textwidth]{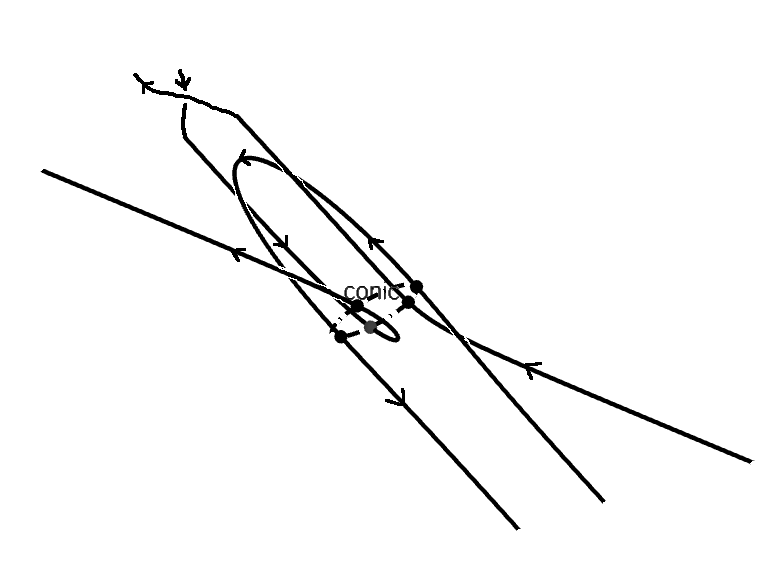}\hfill
		\caption{The projection of a degree $5$ knot in $Q_{3,2}$ of writhe $6$ as a glued knot.}
		\label{e}
	\end{figure}   
	
\end{proof}

\begin{corollary}
	Any real rational knot of degree $d \leq 5$ in $Q_{3,2}$ can also be constructed by gluing of $d$ lines in $Q_{3,2}$.
\end{corollary}

\begin{theorem}
	All real rational curves of degree $4$ with exactly one double point can be constructed by the gluing a degree $3$ knot with a line in $Q_{3,2}$.
\end{theorem}

\begin{proof}
	Consider a degree $3$ real rational knot $K$ in $Q_{3,2}$. Take a point on this knot, say $B$, and consider tangent plane $P$ at the point $B$ to $Q_{3,2}$. $P \cap Q_{3,2}$ is a cone having apex at point $B$. Take a line $L$ along this cone as in figure \ref{figure 11} and which intersects $K$ at $B$ and at one another point $I$, say. Now choose a point $m$ on $Q_{3,2}$ which does not lie on the line $L$ and also not lying on $K$ such that tangent plane at $m$ passes through point $B$. Project  $K \cup L$ from point $m$ in $\mathbb{RP}^{3}$. The projection of the degree $3$ knot in $\mathbb{RP}^{3}$  intersects projection of line $L$ at two points, one of these projection of point $B$ lies on the plane at infinity and glue these two on another intersecting point projection of $I$. This gluing gives us two different classes of degree $4$ walls. The third isotopy class of degree $4$ wall is constructed by gluing a degree $3$ knot with a line but with the opposite orientation.

	\begin{figure}[htbp]
		\centering
		\includegraphics[width=0.7\textwidth]{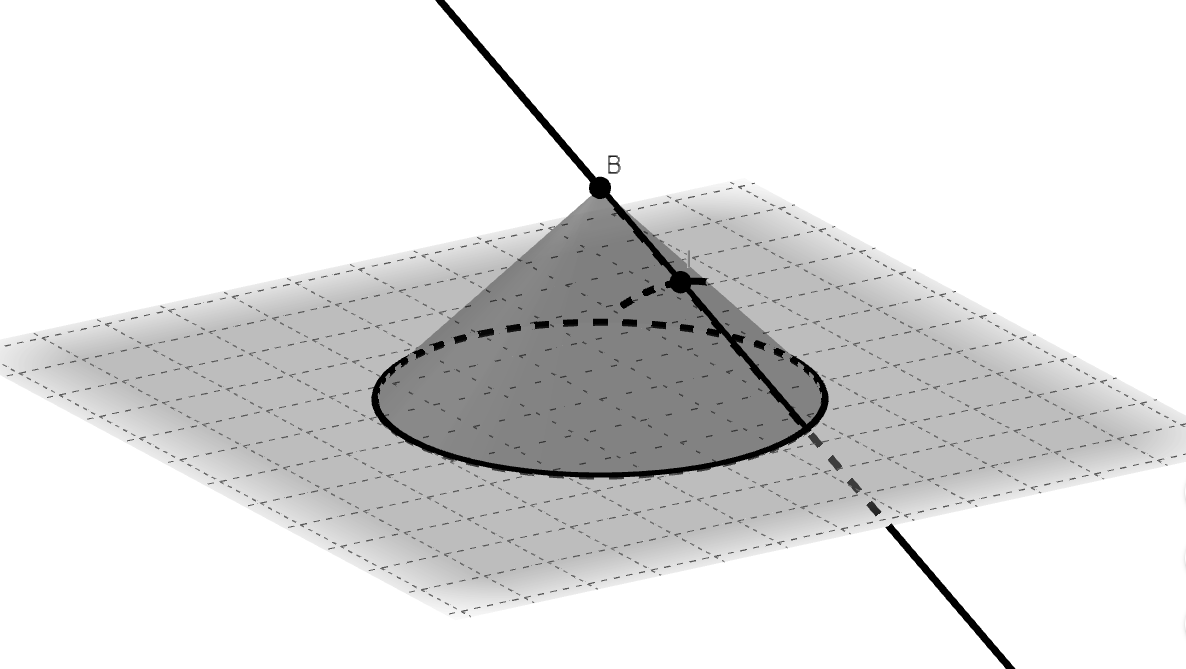}\hfill
		\caption{}
		\label{figure 11}
	\end{figure}
	We can check this by considering the situation just before gluing in $\mathbb{RP}^{4}$, as shown in figure \ref{figure 11}. The cone is the intersection of the tangent plane at the point $B$ with $Q_{3,2}$. As discussed above, we are gluing these two at the point $I$. If we project this glued curve from the double point $B$ then the projection of $B$ is along the tangents of strands at the point $B$ and one tangent is decided by gluing i.e. whether it lies inside the conic or outside the conic. Without loss of generality, assume that the tangent to the other strand is inside the conic. Then the projection of $K \cup L$ from the double point is a degree 2 knot which intersects the plane at infinity in two points; either both lie inside the conic or one lies inside and another outside the conic. These two correspond to two different classes of degree $4$ walls as proved in subsection \ref{4.4.1} and the third one is obtained by reversing the orientation of $L$.
\end{proof}

\section{General Results}
\textbf{Observation:} If we glue an ellipse to a real rational knot $K$ in $\mathbb{RP}^{3}$ in such a way that it adds no more crossings in the projection then glued knot is smoothly isotopic to the knot $K$.

\begin{theorem}\label{C}
	Any real rational knot of degree $d$ in $\mathbb{RP}^{3}$ can be smoothly isotoped to a real rational knot which can be pulled back to 3D hyperboloid.
\end{theorem}

\begin{proof}
	A real rational knot of degree $d$ in $\mathbb{RP}^{3}$ intersects the plane at infinity in $d$ points. Any ellipse not lying on the plane at infinity intersects the plane at infinity in two points. Consider $d$ ellipses which intersect the plane at infinity at the point of intersection of the knot with the plane at infinity and a point on the special conic as shown in figure~\ref{figure 12}. Glue these ellipses to the knot at the point of intersection. The glued knot is a knot in $\mathbb{RP}^{3}$ all of whose points of intersection with the plane at infinity lie on the special conic and which is smoothly isotopic to the given real rational knot. This glued knot can be pulled back to $Q_{3,2}$ but the pull back is not ncessarily an algebraic knot.
	\begin{figure}[htbp]
		\centering
		\includegraphics[width=0.3\textwidth]{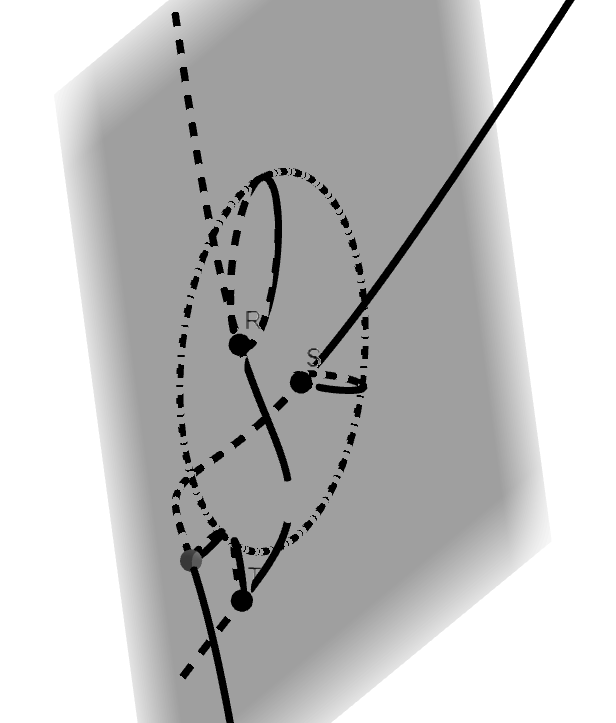}\hfill
		\caption{}
		\label{figure 12}
	\end{figure}  
\end{proof}

\begin{theorem} \label{B}
	If $K_{0}$ and $K_{1}$ are two real rational knots of any degree in $Q_{3,2}$ such that their projections in $\mathbb{RP}^{3}$ are smoothly isotopic then $K_{0}$ and $K_{1}$ are also smoothly isotopic in $Q_{3,2}$. 
\end{theorem}
\begin{proof}
	Consider a point $p$ in $Q_{3,2}$ which does not lie on both the knots $K_{0}$ and $K_{1}$. Project the knots from the point $p$ to a hyperplane in $\mathbb{RP}^{4}$. Suppose $K_{0}{'}$ and $K_{1}^{'}$ are the projections of the knots $K_{0}$ and $K_{1}$ from the point $p$. These projections are given to be smoothly isotopic in $\mathbb{RP}^{3}$ via some isotopy, say $K_{t}^{'}$. By using the theorem \ref{C}, we can isotope the path $K_{t}^{'}$ such that it can be pulled back in $Q_{3,2}$. This pullback gives us a smooth isotopy between $K_{0}$ and $K_{1}$ in $Q_{3,2}$.
\end{proof}

\begin{theorem}
	There are at most $14$ smooth isotopy classes of degree $6$ real rational knots in $Q_{3,2}$.
\end{theorem}
\begin{proof}
	There are $14$ smooth isotopy classes of degree ~6 real rational knots in $\mathbb{RP}^{3}$ as proved in \cite{MR3571386}. By using theorem \ref{B}, we can conclude that there are at most $14$ smooth  isotopy classes of degree $6$ real rational knots in $Q_{3,2}$.
\end{proof}

\bibliographystyle{amsplain}
\FloatBarrier
\bibliography{References/references} 

\end{document}